%% file: template-new.tex
\documentclass[11pt]{amsart} %TBC

\usepackage[hyphens]{url} 

\usepackage{amsmath}
\usepackage{amssymb}
\usepackage{amsthm}
\usepackage{mathtools}

\usepackage{nicefrac}

\usepackage{booktabs}
 \usepackage{tabularx,makecell,rotating,lscape}
\usepackage{longtable}
\usepackage{xltabular}
\usepackage{dcolumn}
\usepackage{multirow}
\usepackage{adjustbox}

\usepackage[backend=bibtex,maxcitenames=2,style=numeric,sortcites=false]{biblatex}
\usepackage{csquotes}

\usepackage[british]{babel}

\usepackage{pdflscape}

\usepackage[defaultlines=4,all]{nowidow}

\usepackage[ruled,vlined]{algorithm2e}

\usepackage{enumitem}
\usepackage{siunitx}
\usepackage{float}
\usepackage{color}
\usepackage{changebar}
\usepackage{acronym} 

\usepackage{blindtext}

\usepackage{geometry}
\usepackage{graphicx}
\usepackage[dvipsnames]{xcolor}
\usepackage{multicol}
\usepackage{savesym}
\usepackage{tikz-cd}
\usetikzlibrary{babel}
\usepackage{moresize}
\usepackage{changepage}
\usepackage{wrapfig}
\usepackage{adjustbox}
\usepackage{csquotes}
\usepackage{mathtools}
\usepackage{comment}
\savesymbol{bigtimes}
\savesymbol{widebar}
\usepackage{amsmath,amssymb,amsthm, mathabx,mathrsfs}
\usepackage{hyperref}
\restoresymbol{tx}{bigtimes}
\restoresymbol{tx}{widebar}
\newcommand{\MFN}{\widetilde{\textrm{FN}}}

\newcommand{\FN}{\textrm{FN}}

\newcommand{\rk}{\textrm{rk}\,}
\newcommand{\mc}{\max_{\le}}
\newcommand{\sd}{\textrm{sd}}
\newcommand{\Fox}{\textrm{Fox}}
\newcommand{\ev}{\textrm{ev}}
\newcommand{\Sat}{\textrm{Sat}}
 
\newcommand{\FNP}{\textrm{FN}^{\le}}
\newcommand{\Ff}{   \textrm{sd } \mathcal{F}}

\newcommand{\Conf}{\textrm{Conf}}
\newcommand{\Emb}{\textrm{Emb}}

\newcommand{\Kons}{\textrm{Konts}}

\newcommand{\sSet}{\textbf{sSet}}

\newcommand{\Ass}{\textbf{Ass}}

\newcommand{\sse}[1]{\mathscr{#1}}

\DeclareMathOperator{\Tot}{Tot}

\newcommand{\vect}[1]{\underline{#1}}

\newcommand{\Nerve}{\mathcal{N}}
\newcommand{\bt}[1]{ {#1}_{\bullet} }

\newcommand{\res}{\textrm{res}}
\newcommand{\Id}{\textrm{Id}}

\newcommand{\Chains}{\textbf{Ch}^{-}(\mathbb{Z}) }

\newcommand{\FNC}{\textrm{FN}^{\textrm{sd}}}

\newcommand{\BFNA}{\mathcal{B}(\FN_m,A)}

\newcommand{\BZ}{\textrm{sd\,}\NBZ}
\newcommand{\NBZ}{\mathrm{BZ}}

\newcommand{\bit}{\begin{itemize}} 
\newcommand{\eit}{\end{itemize}}
\newcommand{\PP}{\mathcal{P}}

\newcommand{\R}{\mathbb{R}}
\newcommand{\F}{\mathbb{F}}

\allowdisplaybreaks

\newtheoremstyle{mystyle}
  {3pt} % Space above
  {3pt} % Space below
  {} % Body font
  {} % Indent amount
  {\bfseries} % Theorem head font
  {.} % Punctuation after theorem head
  {.5em} % Space after theorem head
  {} % Theorem head spec (can be left empty, meaning `normal')
  
\theoremstyle{mystyle} 
\newtheorem{theorem}{Theorem}[section]

\newtheorem{definition}[theorem]{Definition}
\newtheorem{lemma}[theorem]{Lemma}
\newtheorem{remark}[theorem]{Remark}
\newtheorem*{recall*}{Recall}
\newtheorem*{notation}{Notation}
\newtheorem{proposition}[theorem]{Proposition}

\newtheorem{example}[theorem]{Example}
\theoremstyle{plain}
\newtheorem*{bd-theorem}{Theorem \ref{bdthm}}
\newtheorem*{mx-theorem}{Theorem \ref{sinha-mcpx}}
\newtheorem*{nc-theorem}{Theorem \ref{nc-theorem}}
\newtheorem*{ncc-theorem}{Theorem \ref{ncc-theorem}}
\newenvironment{code}
    {\begin{center}\ttfamily}
    {\end{center}}

\newtheoremstyle{spaced}
  {1em} % Space above
  {1em} % Space below
  {} % Body font
  {} % Indent amount
  {\bfseries} % Theorem head font
  {.} % Punctuation after theorem head
  {.5em} % Space after theorem head
  {} % Theorem head spec (can be left empty, meaning `normal')
 
\theoremstyle{spaced}
\newtheorem{software}[theorem]{Software}

\graphicspath{{figures/}}

\title{Non collapse of the Sinha spectral sequence for knots in $\R^3$}

\author{Andrea Marino}
\address{University of Victoria}
\author{Paolo Salvatore}
\thanks{The authors acknowledge the 
	MUR Excellence Department Project awarded to the Department of Mathematics, University of Rome Tor Vergata, CUP E83C18000100006} 

\address{Universit\`a di Roma Tor Vergata}

\MakeOuterQuote{"}

\begin{document}
%\renewcommand{\abstractname}{Streamline}
%\raggedbottom
%%%% FRONT MATTER
%\let\cleardoublepage\clearpage %TBC?

%\include{front/frontmtr} %TBC?

%\cleardoublepage% TBC
%%%% START MAIN BODY TEXT

%% Call the glaphdths wrapper.
%\startbody %TBC?

% When printing, check out https://jamesthorne.com/blog/adding-binding-offsets-in-latex-the-easy-way/
% To get binding offsets properly. Watch out the left-right order when including different chapters!! Do a completely new file.

\begin{abstract}
We give  an explicit description up to the third page of the Sinha homology mod 2 spectral sequence 
for the space of long knots in $\R^3$, 
that is conjecturally equivalent to the Vassiliev spectral sequence. 
The description arises from a multicomplex structure on the Fox Neuwirth chain complexes for euclidean configuration spaces.
A computer assisted calculation reveals a non trivial third page differential from a 2-dimensional class, in contrast to the rational case. 
\end{abstract}

\maketitle

\setcounter{tocdepth}{2}
\tableofcontents

\include{streamline-2}
\include{section0}
\include{section1}

\include{section5-new}

\include{applications}

\printbibliography

\end{document}

%% file: streamline-2.tex
\section{Introduction}
A  classical knot invariant can be thought of as an element of the zero dimensional cohomology group of the space of long knots in $\R^3$.  This perspective has become important since the groundbreaking work of Vassiliev \cite{Vassiliev}, who constructed his celebrated spectral sequence
approximating the cohomology of spaces of knots in $\R^m$ via resolutions of singularities.
 Vassiliev found combinatorial formulas \cite{vas-computation,vas-combinatorial} for 
some cohomology classes by making higher differentials explicit; in some cases, these admit a beautiful geometric interpretation \cite{teiblum-turchin, budney}. 
For $m=3$ and in degree 0 the spectral sequence does not see all invariants, but only those of \textit{finite type}. 
One main question in the study of finite-type invariants is the following: can any two knots be distinguished by a finite-type knot invariant? 
Though a (dis)proof is out of reach, partial progress in the understanding of Vassiliev invariants has been recently achieved thanks to another paradigm shift within the field: the use of \textit{embedding calculus}, introduced by Goodwillie and Weiss \cite{weiss,weiss-2}. 
The functor of embeddings from open subsets of the interval to $\R^m$ provides a homotopical approximation to the space
 $\Emb_m$ of long knots in $\R^m$ via its so called \textit{Taylor tower}. The $k$-th stage of this tower turns out to be determined by configuration spaces 
 of at most $k$ distinct points in $\R^m$, together with structural maps between them. Sinha, in its influential paper \cite{sinha2004operads}, shows that the embedding tower for the space of long knots modulo immersions $\overline{\Emb}_m$ arises from a  cosimplicial space of compactified configuration spaces $\Kons_m(\bullet)$, that he calls the Kontsevich spaces \cite{sinha2009topology}. 
Sinha's spectral sequence is the (co)homology spectral sequence of this cosimplicial space. Its second page coincides, up to a shift of bigradings, with the first page of the Vassiliev spectral sequence \cite{tourtchine2005bialgebra} for all $m \ge 2$. Even though they both converge to the same object for $m \ge 4$, the equivalence between the two is still conjectural. However, with rational coefficients and $m \ge 4$, it is verified because both spectral sequences collapse at the second (resp. first) page. 
\cite{volic}. Another bridge between the two perspectives comes from the recent work \cite{boavida,conant,kosanovic,kosanovic2} %in the homotopy framework, 
connecting the Taylor tower for spaces of knots with \textit{additive}\footnote{With respect to the connected sum of knots} finite-type invariants. Perhaps these facts, together with its easier description, have made the Sinha spectral sequence gain prominence over that by Vassiliev recently.

Turchin, Lambrechts, Volic \cite{Lambrechts_2010} and Tsopmene \cite{tsopmene} have proved that the Sinha spectral sequence collapses at the second page for $m\ge 3$, with rational coefficients. The proof relies on the formality of the pair of little disk operads $(\mathbb{E}_m, \mathbb{E}_1)$ over $\mathbb{Q}$.  Such collapse, together with explicit calculations by Turchin himself \cite{tourtchine2005bialgebra}, provides a deep insight into the structure of the rational cohomology of spaces of knots.
In finite characteristics Boavida and Horel \cite{boavida} have showed that the Sinha spectral sequence collapses in a range. Notice that the range is empty in characteristic 2. Little else is known in positive characteristics. In ambient dimension $m=2$, the spectral sequence does not collapse \cite{willwacher,goodwillie} over $\mathbb{Z}, \mathbb{Q},$ and $\mathbb{F}_p$ for any prime $p$.
This should not be surprising: the formality of $\mathbb{E}_m$, which is the main ingredient of the proof, fails to hold in general, as proved by the second named author \cite{salvatore2018planar} for $m=2$ over $\mathbb{F}_2$. 
In principle it is possible to write down explicit formulas for all differentials of the spectral sequence using 
appropriate combinatorial models for euclidean configuration spaces, like for example 
the surjection complexes introduced by McClure and Smith \cite{smith}.
However this approach turns out to be impractical\footnote{See section \ref{configurations} for further explanation.}. 
For computations it is essential to use complexes that are as small as possible. 

The smallest explicit cellular decomposition of the (one-point compactified) euclidean configuration spaces known to the authors is the so-called Fox-Neuwirth decomposition, originally due to Nakamura \cite{Nakamura}, 
 and much later rediscovered by Giusti and Sinha \cite{Giusti}. Cells of $\Conf_n(\mathbb{R}^m)$ are indexed by the so-called Fox-Neuwirth trees $\FNP_m(n)$, and these form a poset induced by the inclusions of cells. 
A Fox-Neuwirth tree is represented by a string of inequalities $\sigma(1) <_{i_1} \dots <_{i_{n-1}} \sigma(n)$ where $\sigma \in \Sigma_n$ is a permutation 
and the indice $i_j$ take values in $\{0,\dots,m-1\}$.  The corresponding cell of the configuration space has codimension equal to the sum 
of the indices $\sum_{j=0}^{n-1}{i_j}$ and consists of the $n$-tuples $(x_1,\dots,x_n)$ such that $x_j$ and $x_{j+1}$ have the first $i_j$
coordinates equal, and $(x_j)_{i_j+1} <(x_{j+1})_{i_j+1}$.
Now there is a dual cell decomposition of a regular CW complex $\NBZ_m(n)$ that is a deformation retract of the configuration space
$\Conf_n(\mathbb{R}^m)$
, and such that the cell indexed by a 
Fox Neuwirth tree has {\em dimension} equal to the sum of indices. The construction goes back to De Concini-Salvetti 
\cite{DeConcini} and Blagojevic-Ziegler \cite{blagojevic2013convex}.  Let $\FNP_m(n)$ be the poset of cells of $\NBZ_m(n)$, ordered by inclusion. 
 
 Mimicking the cosimplicial structure on Kontsevich spaces defined by Sinha, we are able to turn $\FNP_m(\bullet)$ into a cosimplicial poset (Section \ref{fnposet}). For example for $1 \leq i \leq m$ the co-degeneracy $d_i: \Kons_m(n) \to \Kons_m(n+1)$ in the Kontsevich space can be thought as replacing the $i$-th points of a configuration by two very close points $x_i,x_{i+1}$ such that 
 the oriented direction from $x_i$ to $x_{i+1}$ is "vertical" i.e. a positive multiple of the last canonical basis vector.
  The corresponding co-degeneracy $d_i$ in $\FNP_m(n)$ replaces the symbol $i$ in a string of inequality defining a Fox Neuwirth tree by $i<_{m-1} i+1$.  
In both cases one needs also to relabel adding one to the indices greater than $i$. 
By taking the nerve, followed by normalized simplicial chains, we get a (semi)cosimplicial chain complex $\FNC_m( \bullet)$. We call the associated total bicomplex the \textit{Barycentric Fox-Neuwirth bicomplex} (Def. \ref{bar-fn}), since its generators are ascending chains of trees, i.e. cells of the barycentric subdivision of $\NBZ_m(\bullet)$. This bicomplex has two commuting differentials $D_0$ and $D_1$. The first is the differential 
of the chain complex of the barycentric subdivision, and the second comes from the alternated sum of co-degeneracies. 
In the companion paper \cite{marinofox}, the first named author proves that the spectral sequence arising from the Barycentric Fox-Neuwirth bicomplex is isomorphic to the Sinha spectral sequence from the first page on.

Still, for explicit computations, the Barycentric Fox-Neuwirth bicomplex is too large to implement. To overcome this obstacle, we draw from homotopy transfer ideas: since the cellular chain complex of a barycentric subdivision is homotopy equivalent to the original cellular chain complex, it is possible to deform the bicomplex structure on the Barycentric Fox-Neuwirth complexes to a multicomplex structure on the original Fox-Neuwirth complexes. 

We are able to guess the first steps of this structure defining a truncated multicomplex $\MFN_3$ with differentials 
$D_i:\MFN_3(n)_d \to \MFN_3(n+i)_{d+i-1}$ for $i=0,1,2,3$, where $\MFN_3(m)_d$ is the $\F_2$-vector space with basis 
Fox Neuwirth trees in $\FNP_3(n)$ indexing cells of $\NBZ_m(n)$ of dimension $d$.  Here $D_0$ is the differential of the chain complex of
$\NBZ_m(n)$. Instead $D_1$ will be a sum of "co-degeneracies" $D_1^i$  wanting to replace $i$ by the expression $i <_0 i+1$, that geometrically looks like replacing the point 
$x_i$ by two very close points $x_i,x_{i+1}$ so that the direction from the first to the second is "horizontal" i.e. a positive multiple of the {\em first} canonical basis vector. 
This forces to distribute the other points sharing coordinates with $x_i$ into two subsets sharing coordinates with $x_i$ and $x_{i+1}$ in all possible ways, and take the sum.   For example 
$$D_1^3(13|2)= 13|2 || 4 + 3 || 14|2 + 13|| 4|2 + 3|2||14,$$
splits $3$ into $3$ and $4$, and distributes $1,\, 2$.
We use a notation replacing  $<_0, <_1$ and $<_2$ by
a double bar, a single bar, and an empty space respectively. 

The next operator $D_2$ is a sum of contributions $D_2^{ij}$ associated to distinct pairs $i$ and $j$, that wants to split simultaneously
$x_i$ into $x_i, x_{i+1}$, and $x_j$ into $x_j, x_{j+1}$, in the horizontal direction. 
If $x_i$ and $x_j$ share just one coordinate, then also their clones will share one coordinate, i.e. for example
$D_2(1|2)=(1|3 || 2 | 4)$ (here $1$ splits into $1$ and $2$, and $2$ splits into $3$ and $4$).  If $x_i$ and $x_j$ share two coordinates, then the second coordinate can be shared only by one of the pairs (either the original or the 
clone), in two possible ways, so we have $D_2(12)=13|| 2|4 + 3|1||2 4 $.
All other points are distributed according to a suitable rule, described by the theory of Fox polynomials we introduce in section \ref{foxpoly}. 
The last operator $D_3$ is a sum of contributions $D_3^{ijk}$ associated to distinct triples of $i,j,k$, wanting to split simultaneously $x_i,x_j,x_k$ into two each and has various terms according 
to the number of shared coordinates. The most notable contribution arises when $i,j,k$ share all 2 coordinates:
it has 10 terms, 
(see figure \ref{core-positions}) and the most interesting for our calculation are the terms $13||25||46$ and $35||16||24$ in $D_3(123)$, where $1,2,3$ split respectively into $\{1,2\},\{3,4\},\{5,6\}$. Even in this case the theory of Fox polynomials defines how to distribute the other points, and moreover
enables us to prove that indeed we have a truncated multicomplex satisfying the identities
$\sum_{i+j=k} D_i D_j =0$ for $k\leq 3$. The full verification is formal and lengthy and is available on github \cite{multicpx-eqn}.
Next we compare the multicomplex and the bicomplex, showing that they are equivalent.  
This follows by the barycentric deformation theorem \ref{bdthm}, and by introducing an order relation on Fox polynomials.

The upshot is:
\begin{mx-theorem} The first three pages of the Sinha spectral sequence for the space of long knots
 (modulo immersion) in $\R^3$ with coefficients in $\F_2$ are isomorphic to the 
 $3$-truncated spectral sequence associated to $\MFN_3$.
\end{mx-theorem}

Vassiliev conjectured that his spectral sequence in dimension 3 collapses at the first term, that is the second term 
$E^2$ of the Sinha spectral sequence. Now the differential $d^2$ is zero by dimensional reasons, so the first possible non trivial differential is $d^3$.
We cannot find differentials hitting the 0-th dimensional line, because there is no torsion known in the $E^2$ term
on that line. 
%and then the dimension the complexes grow out of our computational capability. 
The first known 2-torsion class is 1-dimensional and lives in  $E^3_{7,8}$. Therefore the first possible source
of a non-trivial differential (a $Tor$-class not coming from integral homology) is 2-dimensional and lives in $E^3_{6,8}$. 

%So we look for differentials hitting 1-dimensional classes. The smallest possibility is to look at a differential 
We must look at the differential $d^3:E^3_{6,8} \to E^3_{9,10}$ hitting 1-dimensional classes. 

The $E^2=E^3$ terms over $\F_2$ have been computed in a range by Turchin: we have
$E^3_{6,8}=(\F_2)^2, \, E^3_{9,10}= (\F_2)^{11}$. 
% \underline{E}^3(6,8)=\F^2$ and $\underline{E}^3_{9,10}=(\F_2)^4$. 
%The isomorphism between the spectral sequences arising from the two bicomplexes is an essential part of the present article. It is realized by constructing a zig-zag of semicosimplicial spaces between the Kontsevich spaces and the Blagojevic Ziegler model of configuration spaces, through a series of objects that we call "Spaces of Weighted Fox-Neuwirth Trees".
We start from a generator $[vas] \in \underline{E}^3_{6,8}$ (a class in $E^1_{6,8}=H_8(\Conf_6(\R^3))$),
that we need to represent by a cycle $vas  \in \MFN_3(6)_8$, and then play a spectral sequence tic tac toe for multicomplexes, to get 
$d^3[vas] \in H_{10}(\Conf_9(\R^3))$.  

Namely we must find $x$ and $y$ such that  $D_0(x)=D_1(vas)$ and $D_0(y)=D_2(vas)+D_1(x)$.
Then $d^3[vas]$ is represented by $D_1(y)+D_2(x)+D_3(vas) \in \MFN_3(9)_{10}$.

%The presentation by Fox-Neuwirth trees can be employed to make explicit calculations
 We implement the calculation in MATLAB, as explained in section \ref{usersguide}. The heaviest part of the computation is to find $y$, i.,e. solving a linear system, and the performance is optimized by exploiting the action of the symmetric group on Fox-Neuwirth trees (it takes 4 hours on a PC with 32Gb Ram).
The algorithm is interesting in itself, and consists of Gauss reduction in group rings of symmetric groups of increasingly smaller order. 

Surprisingly it turns out that $d^3[vas]$ is not trivial. 
% The last chapter is devoted to describe the procedure in detail, answering a long-standing conjecture that goes back to Vassiliev:

\begin{nc-theorem}The Sinha homology spectral sequence for the space of long knots in $\R^3$ (modulo immersions) 
 with $\F_2$ coefficients does not collapse at the third page, since there is a non-trivial differential $d^3:E_{6,8}^3 \to E_{9,10}^3$. 
\end{nc-theorem}

We recall that there is a version of the Sinha spectral sequence  $\underline{E}$
for the actual space of knots in $\R^3$ (not modulo immersions), 
and a natural map of spectral sequences $E \to \underline{E}$, that is a surjection on the second page. 
The second page of $\underline{E}$ is isomorphic to the first page of the Vassiliev spectral sequence. 
The non collapse holds also for $\underline{E}$, and so conjecturally for the Vassiliev spectral sequence.
 \begin{ncc-theorem}
The differential $d^3:\underline{E}^3_{6,8} \cong \F_2 \to \underline{E}^3_{9,10} \cong (\F_2)^4$ is non trivial.
\end{ncc-theorem} 

%By dualizing, since we work over $\mathbb{F}_2$, the same result holds in cohomology.

 As a corollary one can deduce that the operad  $\mathbb{E}_3$ is not 
 multiplicatively 
formal over $\mathbb{F}_2$. Since formality plays a central role in the description of finite-type invariants over rational coefficients, this result highlights the need for new techniques in the case of  positive characteristics.

It is possible to use the same approach to study the differential 
$d^3:E^3_{7,10} \to E^3_{10,12}$, that is potentially non-trivial, since there is a 2-torsion class in $E^2_{8,10}$. 
However the computation will be approximately 100 times heavier. 
Extending our techniques to other differentials of $E^3$, or worse to knots in higher dimensions (4 and more) seems prohibitive for dimensional reasons.

Still we would like to understand better the properties of our multicomplex. Can we construct higher differentials $D_i$? Is there 
a geometric reason we got this formulae? What happens in higher dimensions, or in odd characteristics ?
So far only the 2-dimensional case can be completely worked out, as explained in remark \ref{casem=2}.
Also, it would be nice to have a geometric explanation for the non-trivial differential.

\
Plan of the paper: 

In section \ref{bano} we give the background on Fox Neuwirth trees,  cosimplicial Kontsevich spaces, spectral sequences of knots, and the theory of multicomplexes .
In section \ref{bfns} we define the cosimplicial poset $\FNP_m(\bullet)$ of Fox Neuwirth trees, and prove the barycentric deformation theorem.
In section \ref{fn-mcpx-3} we define the (truncated) multicomplex structure on the Fox Neuwirth complexes in dimension 3, and prove that it is equivalent to 
the bicomplex coming from $\FNP_m(\bullet)$. 
In the last section \ref{usersguide} we describe the algorithms and the software used to find the non trivial differential $d^3$.

\

P. S. thanks Victor Turchin for inviting him to KSU where this flow of ideas started, and for several discussions on the topology of the knot spaces. 
Part of this work appears in the PhD thesis of A.M., supervised by P. S.

\nopagebreak

%% file: section0.tex
\section{Background and Notation} \label{bano}
In the following, we go through the essential concepts and notations necessary for the article. 

\subsection{Configuration spaces and Fox-Neuwirth cells} \label{configurations}
Given $n,m\ge 1$, the (euclidean) configuration space of $n$ points in dimension $m$ is defined as
$$ \Conf_n(\mathbb{R}^m) : \{ (x_1, \ldots, x_n) \in (\mathbb{R}^m)^n: x_i \neq x_j \ \ \textrm{ for } i \neq j \}\ . $$
\begin{comment}
it is a model for the little disk operad \cite{sinha2010homology}
\end{comment} 
We briefly introduce a combinatorial description of such space - the Fox-Neuwirth stratification - that will be fundamental for our investigation; see \cite{Giusti} for a thorough account. Every configuration determines a permutation: if we denote by $<$ the lexicographical order on $\mathbb{R}^m$, for every point $\vect{x} \in \Conf_n(\mathbb{R}^m)$ there exists a unique permutation $\sigma \in \Sigma_n$ such that $x_{\sigma(1)} < \ldots < x_{\sigma(n)}$. However, this is not enough to decompose $\Conf_n(\mathbb{R}^m)$ nicely; the set of points yielding a given permutation has substrata of different dimension. This motivates the following definition.
\begin{definition}A depth-ordering of height $m$ on a set $S$ is the data of
\begin{itemize}
\item A linear order on $S$;
\item For each strict inequality $x < y$ in $S$, a depth index $d(x,y) \in \{0, \ldots, m-1\}$ such that
$$ d(x,z) = \min\{d(x,y), d(y,z) \} \ .$$
\end{itemize}
We write $x <_a y$ as a notation for $d(x,y) = a$.
\end{definition}

There is an evident depth-ordering on $\mathbb{R}^m$ with linear order given by the lexicographical order, and depth index equal to the number of shared coordinates. On the other hand, in case of the finite set $S=\{1, \ldots, n\}$, a depth-ordering $\Gamma$ is the same as:
\begin{itemize}
\item A permutation $\sigma \in \Sigma_n$, that determines the linear order;
\item Depth indices $a_1, \ldots, a_{n-1} \in \{0, \ldots, m-1\}$ such that $\sigma(1) <_{a_1} \sigma(2) <_{a_2} \ldots <_{a_{n-1}} \sigma(n) $. The other depth-indices are determined by the min-rule.
\end{itemize}
We denote by $\FNP_m(n)$ the set of depth-orderings of height $m$ on $\{1, \ldots, n\}$. There is a pictorial way to represent such depth-orderings, that will be important in later sections. Given $\Gamma= (\sigma, a_{\bullet}) \in \FNP_m(n)$ we can draw a planar tree of height $m$ and $n$ leaves with the following properties:
\begin{itemize}
\item Leaves are labeled according to $\sigma(1), \ldots, \sigma(n)$;
\item Downward paths starting at $\sigma(i)$ and $\sigma(i+1)$ merge at height $a_i$.
\end{itemize}
\vspace{0.5cm}
\begin{figure}[h]
	\centering
	\includegraphics{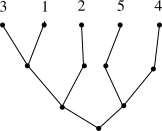}
	\caption{The tree associated to the Fox-Neuwirth cell $3 <_2 1 <_1 2 <_0 5 <_1 4 \in \FNP_3(5)$}
	\label{fox cell}
\end{figure}

Driven by the intuition above, we define the set of configurations that yields a given depth-ordering:
\begin{definition} For a depth-ordering $\Gamma = (\sigma, a_{\bullet}) \in \FNP_m(n)$, define
$$ \Conf(\Gamma) := \{ \vect{x} \in \Conf_n(\mathbb{R}^m) : x_{\sigma(1)} <_{a_1} \ldots <_{a_{n-1}} x_{\sigma(n)} \} \ .$$
\end{definition}
This gives the \textit{Fox-Neuwirth} stratification, first introduced it \cite{FN} for $m=2$, and later generalized to arbitrary dimensions by Nakamura \cite{Nakamura}. It is easy to see \cite{Giusti} that the following holds:
\begin{theorem} \label{FN-theorem}For any $\Gamma \in \FNP_m(n)$, the subspace $\Conf(\Gamma)$ is homeomorphic to a Euclidean ball of dimension $mn - \sum a_i$, and is a cell of an equivariant
CW decomposition of the one-point compactification $\Conf_n(\mathbb{R}^m)^+$.
\end{theorem}
The exit poset of such stratification \footnote{This is the usual terminology in the context of stratifications of what would be called \textit{face poset} in the context of CW-complexes: $\Lambda' \le \Lambda$ if $\Lambda' \subset \widebar{\Lambda}$.} can be determined explicitly (\cite{blagojevic2013convex}, Lemma 3.3): $\Conf(\Gamma') \subset \overline{\Conf(\Gamma)}$ when for all $\alpha, \beta \in \{1, \ldots,n\}$
$$ \alpha <_r \beta \text{ in } \Gamma \ \  \Rightarrow \ \ \textbf{either  } \alpha <_s \beta \text{ with } s \leq r \textbf{  or  } \alpha >_s \beta \text{ with } s < r \text{ in } \Gamma'\ . $$
From now on, we consider $\FNP_m(n)$ to be a poset with the \textbf{reverse} exit poset structure.\footnote{That is, minimal elements are open cells of $\Conf_n(\mathbb{R}^m)$, corresponding to corollas.} At last, we outline a method to describe the (co)homology of $\Conf_n(\mathbb{R}^m)$ via Fox-Neuwirth trees. Building on the work of De Concini and Salvetti \cite{DeConcini}, Blagojevic and Ziegler showed in \cite{blagojevic2013convex} the existence of a subspace $\NBZ_m(n) \subset \Conf_n(\mathbb{R}^m)$ with the following features:
\begin{itemize}
    \item The inclusion $\NBZ_m(n) \to \Conf_n(\mathbb{R}^m)$ is a homotopy equivalence;
    \item $\NBZ_m(n)$ has a regular CW-complex structure, with a cell  $c(\Gamma)$ for each $\Gamma \in \FNP_m(n)$;
    \item The poset of cells has the opposite order with respect to the inclusion of strata: $c(\Gamma) \subset \overline{c(\Gamma')}$ iff $\Conf(\Gamma') \subset \overline{\Conf(\Gamma)}$;
    \item If $\Gamma = (\sigma, a_{\bullet})$, then the dimension of $c(\Gamma)$ is $\sum a_i$.
\end{itemize}
As a consequence, one can compute the (co)homology of $\Conf_n(\mathbb{R}^m)$ via the (dual of the) cellular chain complex of $\NBZ_m(n)$. We denote such complex by $\MFN_m(n)$. The work \cite{Giusti} is the first to give a complete description of signs in the differential, though in the cochain complex framework. We can conveniently reformulate it as follows. If $\Gamma < \Gamma'$ and $\dim c(\Gamma) = \dim c(\Gamma')-1$, the sign $\epsilon(\Gamma, \Gamma') = (-1)^{\kappa(\Gamma, \Gamma')}$ with which $\Gamma$ appears in $\partial \Gamma'$ is given by 
$$ \kappa(\Gamma, \Gamma') =p + m(n-1) + \sum_{k=1}^{p-1} \min \{ a_k, a_p+1\} + (a_p+1) + \sum_{k=p+1}^{n-1} \min \{ a_k, a_p\}\ , $$
where $\Gamma = (\sigma, a_{\bullet}), \Gamma' = (\sigma', a'_{\bullet})$ and $p$ is the minimum index such that $a_p < a'_p$. For the sake of completeness, let us remark that $C_*( \Conf_n(\mathbb{R}^m))$ has another well-established combinatorial model, provided by the components of the McClure-Smith \cite{smith} surjection operad $\mathcal{S}_m(n)$. To demonstrate the difference in terms of size between the two models, one can see for example\footnote{See \cite{cacti}, Theorem 4.13, for the generating function of the number of cells in $\mathcal{S}_2$. We used Wolfram Alpha to compute the Taylor Series up to order 10.} that the first values of 
$$q_k := \frac{ \rk \mathcal{S}_2(k)_{k-1}}{ \rk \MFN_2(k)_{k-1} }$$
are given by\footnote{Note that $q_k$ is always an integer, as generators of $\MFN_2(k)_{k-1}$ have the form $\sigma(1) <_1 \ldots <_1 \sigma(k)$ for $\sigma \in \Sigma_k$, and $\mathcal{S}_2(k)_{k-1}$ has a free action of $\Sigma_k$.}
\begin{center}
\begin{tabular}{|c||c|c|c|c|c|c|c|}
\hline
$k$ & 4 &5 & 6 & 7 & 8 & 9 & 10  \\ \hline
$q_k$ & 5 & 14&42&132&429&1430&4862 \\ \hline
\end{tabular}
\end{center}
A direct combinatorial comparison between the two frameworks in the case $m=2$ is treated in Tourtchine's article \cite{dyer}. Later on, we will also consider the barycentric subdivision $\BZ_m(n)$ of $\NBZ_m(n)$ because of its geometric relation with $\Conf_n(\mathbb{R}^m)$. As it turns out, $\BZ_m(n)$ is an explicit realization\footnote{Denoted \cite{blagojevic2013convex} as $\Ff (m,n)$.} of $\Nerve(\FNP_m(n))$ inside the configuration space: more precisely, Theorem 3.13 of \cite{blagojevic2013convex} shows that $\BZ_m(n)$ can be realized inside\footnote{Corresponding to $W_n^{\oplus d}$ in the cited work.} $\Conf_n(\mathbb{R}^m)$ as the boundary of a PL star-shaped closed subset. Explicitly, given $(\sigma, a_{\bullet})$ a Fox-Neuwirth cell in $\FNP_m(n)$, the vertices $v(\sigma, a_{\bullet})$  of $\BZ_m(n)$ are built in the following way: set $v_{\sigma(1)}(\sigma,a_{\bullet})=0$, and then define inductively
$$ v_{\sigma(p+1)} (\sigma, a_{\bullet} ) = v_{\sigma(p)} (\sigma, a_{\bullet}) + e_{a_p +1} \ .$$
Given a chain $(\sigma_0, a_{\bullet}^0) < \ldots < (\sigma_d, a_{\bullet}^d)$ in $\Nerve(\FNP_m(n))_d$, the corresponding face in $\Conf_n(\mathbb{R}^m)$ is the convex hull of the vertices $v(\sigma_0, a_0), \ldots , v(\sigma_d, a_d)$.  

In figure \ref{fig:zb-triangulation}, we depicted the case $m=3, n=2$ with the following convention: the number of bars between two labels is $2$ minus the depth, and $w(\Gamma) = v_2(\Gamma)$ represents the non-zero component. What we get is a PL 2-sphere, which is what we expect as a deformation retract of $\Conf_2(\mathbb{R}^3)$. In the picture, one can verify that there is a 1-1 correspondence between chains in the poset $\FNP_2(3)$ and simplices of the sphere.

\begin{figure}[h]
	\centering
	\includegraphics{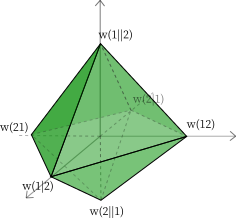}
        \caption{An illustration of $\BZ_3(2)$ as a PL-sphere}
	\label{fig:zb-triangulation}
\end{figure}

\subsection{Cosimplicial structure on the Kontsevich spaces} \label{kons-cosimp}
Recall that the \textit{little disk operad} $\mathbb{E}_m$ in degree $n$ is a collection of $n$ disks of dimension $m$ inside a disk of the same dimension. The configuration space $\Conf_n(\mathbb{R}^m)$ is homotopy equivalent to the $n$-th entry $\mathbb{E}_m(n)$ of the $m$-dimensional little disk operad: the equivalence $\mathbb{E}_m(n) \to \Conf_n(\mathbb{R}^m)$ is simply given by taking the centers of the disks. The operad structure on $\mathbb{E}_m$ is given by inserting a collection of disks into a disk; in the case of configuration spaces, since we only remember centers, it is not clear how to "insert" a configuration of points into a point. A naive idea would be to choose a sufficiently small radius $\epsilon$, substitute the $i$-th point with a disk, and insert the adequately scaled configuration of points into the disk. However, when doing this multiple times, the operadic constraints would hold only up to homotopy. In order to fix this issue, we should somehow take the limit of this procedure for $\epsilon \to 0$. A quantity that remains unchanged when scaling a configuration $\vect{x}$ in a disk of radius $\epsilon \to 0$ is
$$ (*) \ \ \ \phi_{ij}(\vect{x}) = \frac{x_i - x_j}{||x_i - x_j||} \in S^{m-1} \ ,$$
since it is homogeneous. This motivates the following
\begin{definition} The space $\Kons_m(n)$ is the closure in $(S^{m-1})^{\binom{n}{2}}$ of the image of
$$ \phi = (\phi_{ij})_{i \neq j} : \Conf_n(\mathbb{R}^m) \to (S^{m-1})^{\binom{n}{2}} \ ,$$
where $\phi_{ij}$ is defined as in equation $(*)$.
\end{definition}
In \cite{sinha2004manifold}, Sinha collected and reviewed various compactifications of configuration spaces. In particular, he showed that the map $\Conf_n(\mathbb{R}^m) \to \Kons_m(n)$ is a homotopy equivalence. In a successive paper \cite{sinha2004operads}, he also showed that the above heuristics on $\Kons_m$ yields an operadic structure (Theorem 4.5). A picture to keep in mind to think about Kontsevich space is the one below: a collection of nested clouds, each cloud containing a configuration.

\begin{figure}[H]
    \centering
    \includegraphics[scale=0.25]{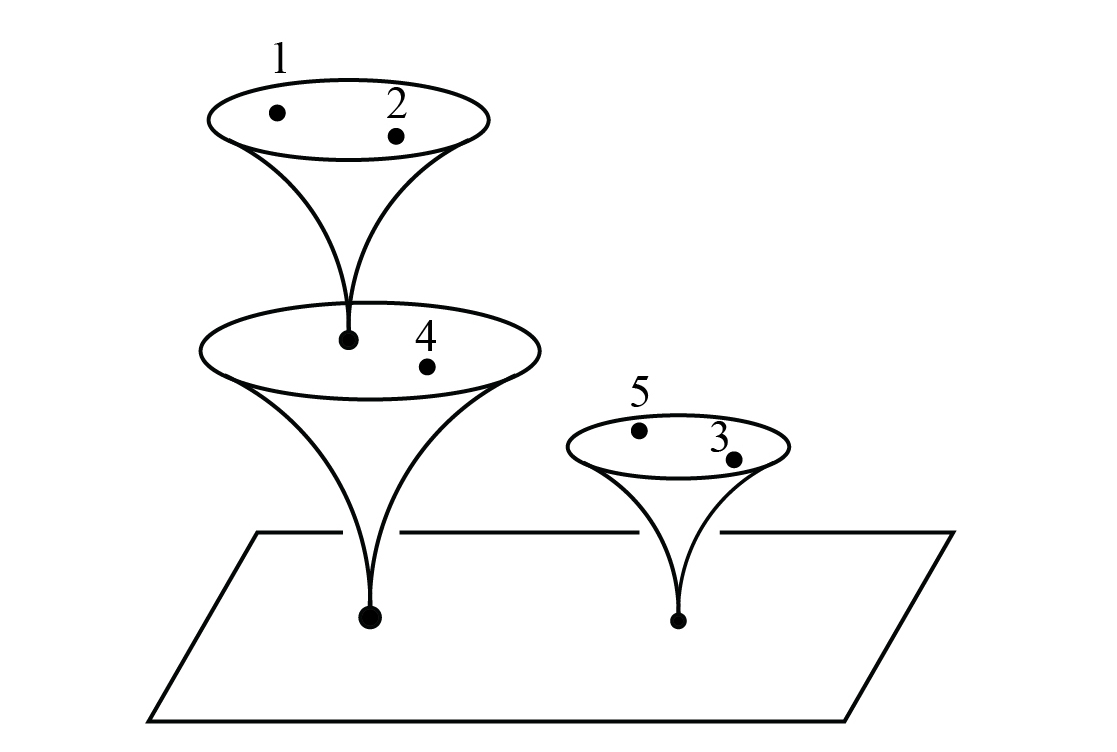}
    \caption{A point in $\Kons_2(5)$}
    \label{Kontsevich}
\end{figure}

There is another feature about this operad that makes it interesting. McClure and Smith \cite{mcclure2001solution} have shown that any \textit{multiplicative operad} $\mathcal{O}$ - that is, a non-symmetric operad equipped with a morphism $\Ass \to \mathcal{O}$ from the associative operad  - defines a cosimplicial structure. In practice, this amounts to a distinguished binary operation $\mu \in \mathcal{O}^2$  such that $\mu( \mu(-,-),-) = \mu(-, \mu(-,-))$ (a \textit{multiplication}) and a unit $e \in \mathcal{O}^0$ such that $\mu(e,-) = \mu(-,e) = \Id$.  The maps $d_i : \mathcal{O}^n \to \mathcal{O}^{n+1}$ for $i=0, \ldots, n+1$ and $s_j: \mathcal{O}^{n+1} \to \mathcal{O}^n$ for $j=0, \ldots, n$ are defined as:
\begin{align*}
    \boxed{1 \le i \le n} & \ \ \ \ \  d_i :& \hspace{-3cm} \mathcal{O}^n \overset{(\Id,\mu)}{\to} \mathcal{O}^n \times \mathcal{O}^2 \overset{\circ_i}{\to} \mathcal{O}^{n+1} \ , \\
    \boxed{i=0} & \ \ \ \ \  d_0 :& \hspace{-3cm} \mathcal{O}^n \overset{(\mu,\Id)}{\to} \mathcal{O}^2 \times \mathcal{O}^n \overset{\circ_2}{\to} \mathcal{O}^{n+1} \ ,\\
    \boxed{i=n+1} & \ \ \ \ \  d_{n+1} :&  \hspace{-3cm}\mathcal{O}^n \overset{(\mu,\Id)}{\to} \mathcal{O}^2 \times \mathcal{O}^n \overset{\circ_1}{\to} \mathcal{O}^{n+1} \ ,\\
    \boxed{0 \le j \le n} & \ \ \ \ \ s_j :& \hspace{-3cm}\mathcal{O}^{n+1} \overset{( \Id,e)}{\to} \mathcal{O}^{n+1} \times \mathcal{O}^0 \overset{\circ_{j+1}}{\to} \mathcal{O}^n \ ,
\end{align*}

%diagrammatically summarized in the picture \ref{smith}.  TOLTO PER ORA, DA CAMBIARE INDICI
The axioms for the multiplication ensure that it is a cosimplicial space. In our case, we can choose any\footnote{In \cite{sinha2004operads} the multiplication is chosen to be the southpole.} $\mu \in \Kons_m(2) \cong S^{m-1}$, while the unit is the only point.
The corresponding cosimplicial moves are:
\begin{itemize}
    \item $s_j$ forgets the $(j+1)$-th point ;
    \item $d_i$ for $i=1, \ldots,n$ doubles the $i$-th point, placing $(i+1)$ infinitesimally close to $i$ in direction $\mu$;
    
    \item $d_0$ inserts a new point with label $1$ at $\infty$ in direction $-\mu$;
    \item $d_{n+1}$ inserts a new point with label $n+1$ at $\infty$ in direction $\mu$
\end{itemize} 
followed by relabelings.

%\begin{figure}[H]
%    \centering
%    \includegraphics[scale=0.15]{smith.png}
 %   \caption{The McClure-Smith cosimplicial structure on a multiplicative operad}
 %   \label{smith}
%\end{figure}

We will set $\mu = e_m$ in the remainder of the article. Since the sphere is homogeneous, the cosimplicial objects corresponding to different multiplications are equivalent. In contrast, when looking for a combinatorial analog of point duplication on Fox-Neuwirth trees, different directions yield different structures, as the depth-ordering on $\mathbb{R}^m$ is sensitive to the coordinates on which two points differ.

\subsection{Spectral Sequences for Knots} \label{sp-seq-intro}
The space of knots considered in this article is constructed in the following way. Consider the space of embeddings $ \Emb_*(\mathbb{I}, \mathbb{I}^m) $ of the interval $\mathbb{I}$ into the cube $\mathbb{I}^m$ such that endpoints and tangent vectors at those endpoints are fixed centers of opposite faces of the cube. The pointwise derivative yields a map
$$ \Emb_*(\mathbb{I}, \mathbb{I}^m) \to  \Omega S^{m-1} \ .$$
The \textit{space of knots modulo immersions} $\overline{\Emb}_m$ is the homotopy fiber of such map. Let us outline the connection with Kontsevich spaces proved in \cite{sinha2004operads}.

Given an abelian group $A$, consider the cosimplicial chain complex $C_*(\Kons_m, A)$ obtained by taking the singular chain complexes of the Kontsevich spaces. By taking the alternated sum $\ \sum (-1)^i C_*(d_i)$ of cofaces, we get a horizontal differential. Together with the vertical, inner differential of the chain complexes, this provides a second-quadrant bicomplex. We also obtain an equivalent\footnote{That is, with spectral sequence isomorphic from the second page on} bicomplex if we take \textit{normalized} chains, that is we restrict each vertical chain complex to the (intersection of the) kernels of co-degeneracies in each degree.

\begin{comment}
\footnote{In the present paper, we intuitively think of cosimplicial chain complexes as having chain complexes lying on rows. However, in all the applications we will deal with second-quadrant bicomplexes with chain complexes on the the columns (non-negative degrees) and action going toward the left (negative degrees).}
\end{comment}

 For $m \ge 4$ the spectral sequence associated to this bicomplex converges to the homology of the homotopy totalization of $\Kons_m$. Furthermore, Sinha proves that such totalization is homotopy equivalent to the spae of knots modulo immersions $\overline{\Emb}_m$. In the interesting case $m =3$, it is not clear if the spectral sequence converges, and what the potential limiting object should be. Nevertheless, when taking rational coefficients, the limit contains \cite{tourtchine2005bialgebra} all the $\mathbb{Q}$-valued Vassiliev knot invariants. For $m=2$, the spectral sequence converges to zero in positive characteristics, but the limit over $\mathbb{Q}$ contains the Grothendieck-Teichmuller Lie Algebra \cite{willwacher}. Lastly, let us stress that the Sinha Spectral Sequence is meaningful for all $m \ge 2$ because of its deep connection to the pair $(\mathbb{E}_1, \mathbb{E}_m)$, especially regarding the formality of the latter.

There are several results stating whether such spectral sequence collapses at the second page or not. The answer turns out to be heavily dependent on the coefficient group $A$. Let us illustrate the current situation:
\begin{enumerate}
\item For $A= \mathbb{Q}, m\ge 4$, the collapse has been proved by Lambrechts, Volic, Tourtchine \cite{Lambrechts_2010};
\item The proof was later generalized to include $m=3$ by Tsopméné \cite{tsopmene};
\item The second named author \cite{salvatore} implicitly found an obstruction to collapse for $A=\mathbb{F}_2, m=2$ in his article about the non-formality of $\mathbb{E}_2$ over $\mathbb{F}_2$, ;
\item Moriya \cite{moriya2023differentials} found an obstruction to collapse for $ A=\mathbb{F}_3, m=2$, as well as reproved Salvatore's result over $\mathbb{F}_2$;
\item Salvatore later observed that non-collapse for $m=2$ in positive characteristics follows from a theorem of Goodwillie \cite{goodwillie} about the mod p homology of cosimplicial spaces\footnote{See \cite{goodwillie} for a thorough explanation.};
\item Turchin and Willwacher \cite{willwacher} found an obstruction to collapse for $A=\mathbb{Q}, \mathbb{Z}, m=2$.
\end{enumerate} 

Let us also mention the existence of a few special "framing" elements in the spectral sequence that are relevant to the investigation of collapse. They occur because the limiting object is the space of knots modulo immersions, and not the usual space of knots. 
This space $\overline{\Emb}_m$ is the homotopy fiber of the derivative map
$$ \Emb_*(\mathbb{I}, \mathbb{I}^m) \to  \Omega S^{m-1} \ .$$
Sinha proved in \cite{sinha2004operads} that this map is null-homotopic, which yields an equivalence $\overline{\Emb}_m \simeq \Emb_*(\mathbb{I}, \mathbb{I}^m) \times \Omega^2 S^{m-1}$. We then expect to find elements coming from $H_*(\Omega^2 S^{m-1})$ in the spectral sequence. Indeed, since such elements contribute to the homology of the final object, they survive to the $\infty$ page. If we are looking for an obstruction to the collapse of the spectral sequence, it is important to bear in mind this phenomenon. This will be further elucidated in the last section \ref{usersguide}
 addressing the case $m=3$.

Regarding actual computations, the second page of the spectral sequence has been extensively studied by Tourtchine \cite{tourtchine2005bialgebra}, who gave a combinatorial description of the chain complexes appearing in the first page and of their homology (which constitutes the second page). Remarkably enough, its description shows that the first and second pages of the spectral sequence only depend\footnote{Beside a bigrading shift, which depends on $m$} on the parity of the ambient dimension $m$; in case we work over $\mathbb{F}_2$, even this distinction disappears. 
 Since there are different bigrading conventions in the three approaches (Sinha, Vassiliev, Tourtchine), let us state the bigrading shifts explicitly:

\adjustbox{scale=1.25,center}{%
\begin{tikzcd}
	& \substack{\textrm{Sinha} \\ (p,q)} \\ \\
	\substack{\textrm{Tourtchine} \\ (i,j)} & & \substack{\textrm{Vassiliev} \\ (n,d)}
	\arrow["\substack{p = -j \\ q = i(m-1)}", from=3-1, to=1-2]
	\arrow["\substack{n=-i \\ d= mi-j}"', from=3-1, to=3-3]
	\arrow["\substack{p=d+mn \\ q=-(m-1)n}"', from=3-3, to=1-2]
\end{tikzcd}
}

This will be especially important when referring to Tourtchine tables. The Tourtchine-to-Vassiliev bigrading shift is stated\footnote{Unluckily with $(p,q)$ instead of $(n,d)$ and $d$ in place of $m$... Apparently $(p,q)$ are the most trendy indices.} in \cite{tourtchine2005bialgebra}, Appendix B. It may be helpful to keep in mind that the $p$ index in Sinha is \textit{minus} the number of points.

\begin{comment}
Another couple of quirks to keep in mind are the following:
\begin{itemize}
\item The $p$ index in Sinha is \textit{minus} the number of points;
\item Our notation for the homology spectral sequence is $E_{pq}^r$ (since we don't use the cohomology spectral sequence), while Tourtchine uses $E_{pq}^r$ for cohomology and $E^{pq}_r$ for homology.
\end{itemize}
\end{comment}

At last, note that the transformations going out from Tourtchine are not invertible on the integers, but only injective; the gradings which are not in the image of the transformation correspond to vanishing elements of the spectral sequence. Indeed, there are many zero elements in the Sinha Spectral Sequence\footnote{ We stress the ambient dimension $m$ of the spectral sequence, as in $E_{pq}^r(m)$.}:
\begin{enumerate}
\item If $q$ is not divisible by $(m-1)$ we have that $E_{pq}^1(m) \subseteq H_q(\Conf_{-p}(\mathbb{R}^m)) = 0$, since the latter is generated by "planetary systems" \cite{sinha2010homology}, and an orbit of a point around another is parametrized by $S^{m-1}$. This vanishing also shows that $E_{r}^{pq}(m) \cong E_{r+1}^{pq}(m) $ unless $r \equiv 1 \pmod{m-1}$ or $r=0$, since  otherwise the differentials always depart from or land in zero.
\item In Tourtchine's bigrading, if $(i,j)$ does not respect $i+1 \le j \le 2i$ (see right after proposition 3.3), then the first page of the spectral sequence in position $(i,j)$ vanishes. Translating to Sinha's bigrading, we have that non-zero terms sit in $(-p,q)$ such that (see figure \ref{vanish-seq}):
\begin{align*}
q & \le (p-1)(m-1) \ ,\\
q & \ge p \frac{(m-1)}{2}\ .
\end{align*}
\end{enumerate}

\begin{figure}[H]
    \centering
    \includegraphics[scale=0.38]{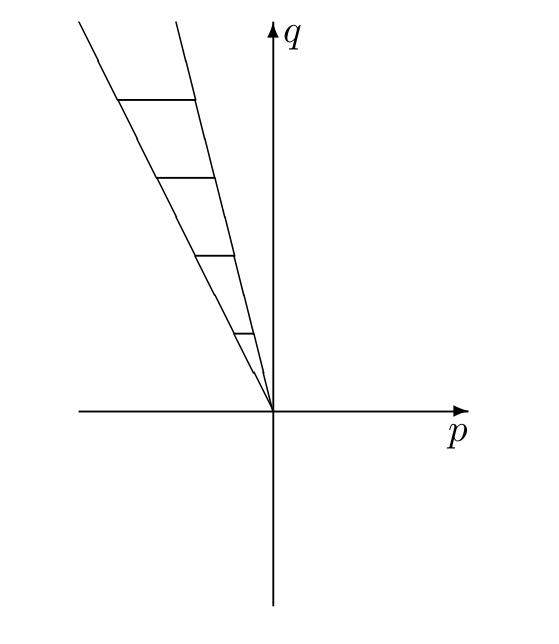}
    \caption{A sketch of the "lower line" and "upper line" for Sinha Spectral Sequence}
    \label{vanish-seq}
\end{figure}

This observation shows that the number of non-zero generators on an antidiagonal is finite for $m \ge 4$. Since the spectral sequence collapse at the second page over the rational numbers, summing non-vanishing terms along antidiagonals yields:
$$ H_r( E_m, \mathbb{Q}) \cong \bigoplus_{\substack{\frac{m+r-1}{m-2} < p \le \frac{2r}{m-3}\\ p \equiv -r \pmod{d-1} }} E_2^{-p,p+r}(m) \otimes \mathbb{Q} \ .$$
Note that the formula is non-trivial, as it implies $H_r(E_m, \mathbb{Q} ) = 0$ when $1\le r \le m-4$.

At last, let us mention the connection with the formality of little disk operads. A multiplicative operad $\mathcal{O}$ in chain complexes is said to be \textit{multiplicatively formal} if there exists a zig-zag of quasi-isomorphisms between multiplicative operads
$$ \mathcal{O} \stackrel{\simeq}{\leftarrow} \ldots  \stackrel{\simeq}{\rightarrow} H_{\bullet} \mathcal{O} $$
that preserves the multiplication. In the previous section, we explained how to construct a cosimplicial chain complex from a multiplicative operad in spaces. The associated spectral sequence, in case such operad is multiplicatively formal, will collapse at the second page. 

When considering $\mathcal{O} = C_*(\mathbb{E}_m,A)$, the multiplication is induced by the standard inclusion $\mathbb{E}_1 \to \mathbb{E}_m$. In some cases, this provides a bridge between the collapse of the Sinha spectral sequence and the relative formality of the pair $(C_*(\mathbb{E}_1, A) , C_*(\mathbb{E}_m, A))$. 

\subsection{The theory of multicomplexes} \label{multicpx-back}
A multicomplex\footnote{Also known as a twisted chain complex or a $D_{\infty}$-module.} is an algebraic structure generalizing the notion of a bicomplex. The structure involves a family of higher “differentials” indexed by the non-negative integers, meant to encode commutations holding only up to homotopy. The formal definition is the following (from \cite{livernet}): 
\begin{definition} \label{multicpx-def} Let $k$ be a commutative unital ring. A multicomplex over $k$ is a 
$(\mathbb{Z}, \mathbb{Z})$-graded $k$ module $C_{**}$ equipped with maps $D_i: C_{**} \to C_{**}$ for $i \ge 0$ with bidegree $|D_i| = (-i,i-1)$ such that
$$ \sum_{i+j=k} D_i D_j = 0 \ \ \ \textrm{for all } k \ge 0 \ .$$
A morphism of multicomplexes $f: (C_{**}, D_i) \to (C_{**}', D_i')$ is given by maps $f_i : C_{**} \to C_{**}'$ for $i \ge 0$ of bidegree $|f_i| = (-i,i)$ satisfying
$$ \sum_{i+j = k} f_i D_j = \sum_{i+j=k} D'_i f_j \ \ \textrm{for all } k \ge 0 \ .$$
\end{definition}
If we set $D_i = 0$ for $i \ge 2$, we recover the notion of a bicomplex. The latter can be alternatively seen as a "chain complex of chain complexes". There is an analog reformulation for multicomplexes, too. If we fix $p\in \mathbb{Z}$, the $\mathbb{Z}$-graded module $C_{p*}$ is a chain complex; indeed, $D_0$ has bidegree $(0,-1)$, and the multicomplex law for $k=0$ implies $D_0^2 = 0$. Thus, in the spirit of "homotopy bicomplexes", we can rewrite:
\begin{definition}Let $k$ be a commutative unital ring. A multicomplex is a collection of chain complexes $C_n \in \textbf{Ch}(k)$ equipped wih a collection of maps $D_i : C_n \to C_{n-i}[i-1] $ for $i \ge 1$ such that
$$ \sum_{i+j=k} D_i D_j = 0 \ \ \ \textrm{for all } k \ge 0 \ ,$$
where $D_0$ is the internal differential of the chain complexes $C_n$. A morphism of multicomplexes $f : (C_*, D_i) \to (C'_*, D'_i)$ is a collection of maps $f_i : C_n \to C'_{n-i}[i]$ such that 
$$ \sum_{i+j=k} D'_i f_j = \sum_{i+j = k} f_i D_j \ \ \textrm{for all } k \ge 0 \ .$$
\end{definition}

We denote the category of multicomplexes as $\textrm{m}\textbf{Ch}(k)$. Sometimes a different sign convention, giving an isomorphic category, is used \cite{cirici2017derived}. There are several notions of homotopy equivalence between multicomplexes; one that is natural from the point of view of homotopy bicomplexes is the following:
\begin{definition} Let $C_*, D_*$ be multicomplexes. A \textit{pointwise} homotopy equivalence is a morphism of multicomplexes $f_*: C_* \to D_*$ such that $f_0$ is a homotopy equivalence.
\end{definition}

%da spostare sotto

\begin{remark}
A nice feature of multicomplexes over bicomplexes is that they are \textit{stable under homotopy deformations}. In other words, given a multicomplex $C_*$ and a collection of chain complexes $\{C'_n\}_{n \in \mathbb{N} }$ together with homotopy equivalences $f_0^{(n)}: C_n \to C'_n$, there exists an extension of $\{C'_n\}_{n \in \mathbb{N} }$ to a multicomplex $C'_*$ and an extension of $f_0$ to a pointwise homotopy equivalence $f: C_* \to C'_*$. See the beautiful master thesis \cite{garda}, chapter 2, for a thorough account of this Homotopy Transfer Theorem; the precise formula for the transferred differential is given in Theorem 2.2.1.  
\end{remark}

A multicomplex has an associated total complex, with a filtration, and thus an associated spectral sequence. 
Notice that a pointwise homotopy equivalence provides an isomorphism of spectral sequences from the first page on. This provides a bridge with the definition of $E_1$-quasi-isomorphism introduced in \cite{cirici2017derived}. 
%spostato qui sotto
For the sake of simplicity, from now on we restrict to multicomplexes $\textrm{m}\Chains$ in which $C_n$ is a non-negatively graded chain complex (which is bounded below according to \cite{garda}). 

\begin{definition} \label{multicpx-tot} The totalization $\Tot C$ of a multicomplex $C \in \textrm{m}\Chains$ is a chain complex defined by
$$ (\Tot C)_n := \bigoplus_{a+b = n} C_{ab} \ ,$$
with differential given by $ D:= \sum_{i \ge 0} D_i $. The total complex admits a filtration 
$$ (F_p \Tot C)_n := \bigoplus_{\substack{a+b = n \\ a \le p}} C_{ab} \ . $$
\end{definition}

The spectral sequence of such a filtered complex can be described explicitly. Given $x \in C_{pq}$, we say that $x \in Z_{pq}^r$  if there exist \textbf{witnesses} $z_{p-j} \in C_{p-i,q+i}$ for all $1 \le j \le r-1$ such that $D_0 x = 0$ and 
$$ D_k x = \sum_{i=0}^{k-1} D_i z_{p-k+i} \ \ \ \textrm{for all } 1 \le k \le r-1\ . $$
Analogously, we say that $x \in B_{pq}^r$ if there exist witnesses $c_{p+k} \in C_{p+k,q-k+1}$ for all $0 \le k \le r-1$ s.t. 
\begin{align*}
x & = \sum_{k=0}^{r-1} D_k c_{p+k} \ ,\\
0 & = \sum_{k = \ell}^{r-1} D_{k-\ell}c_{p+k} \ \ \textrm{for all } 1 \le \ell \le r-1 \ . 
\end{align*}
Terms of the spectral sequence are then computed by $E_{pq}^r := Z_{pq}^r/ B_{pq}^r$. The differential is given by:
$$ (*) \delta_{pq}^r([x]) = [ D_r x - D_{r-1} z_{p-1} - \ldots D_1 z_{p-r+1}] \ ,$$
for any choice of witnesses $(z_{p-1}, \ldots, z_{p-r+1})$. In particular, note that for $r \ge 2$ the differential $\delta^r$ is \textbf{not}\footnote{See also \cite{hurtubise}, Example 4, for an explicit case of $\delta^2 \neq D_2$.} induced by $D_r$, as it happens in the classical case for $r=0,1$. Although the map $D_r$ from the multicomplex structure could seem a lift of $\delta^r$, it ultimately makes the differential in the spectral sequence more complicated, adding "higher order" contributions. Indeed, in the bicomplex case, formula $(*)$ simplifies to $D_1 z_{p-r+1}$. Nevertheless, using multicomplexes can be convenient for computations in spectral sequences: we can pay the price of a complicated differential to obtain smaller modules in each bidegree. We explain this approach for the case at hand in Section \ref{bicomplex}. 

%% file: section1.tex
\section{Barycentric Fox-Neuwirth Structure} \label{bfns}

\subsection{Definition of cosimplicial moves} \label{fox-intuition}
We want to introduce a simpler model for the Kontsevich cosimplicial spaces at the level of chains. Despite Fox-Neuwirth strata do not induce a cell structure on the Kontsevich spaces (see \cite{voronov}), we want to understand heuristically what would happen to them under Kontsevich cosimplicial moves if the extension existed. Recall that the cosimplicial structure on $\Kons_m$ depends on the choice of a direction $\mu$ (see \ref{kons-cosimp}); in this section, we set $\mu = e_m$. For the sake of visualization, we consider an example in ambient dimension $m=3$. Consider a point $(x_1, x_2, x_3) \in \Conf( 1 <_2 3 <_0 2)$ as in the following picture:
\begin{figure}[H]
\centering
\includegraphics[width=5.5cm]{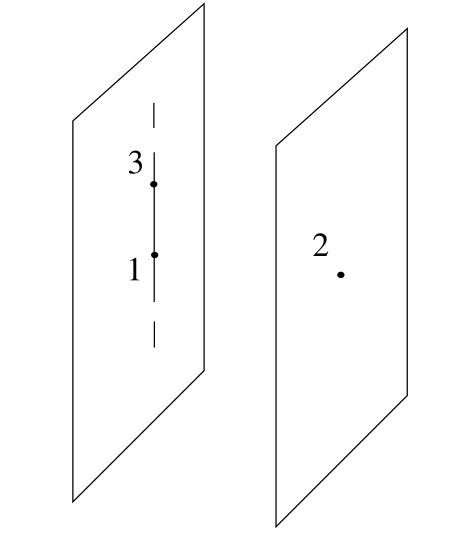}
\end{figure}
and its image $\phi(x_1, x_2, x_3) \in \Kons_3(3)$. Applying $d_3$ we get
\begin{figure}[H]
\centering
\includegraphics[width=5.5cm]{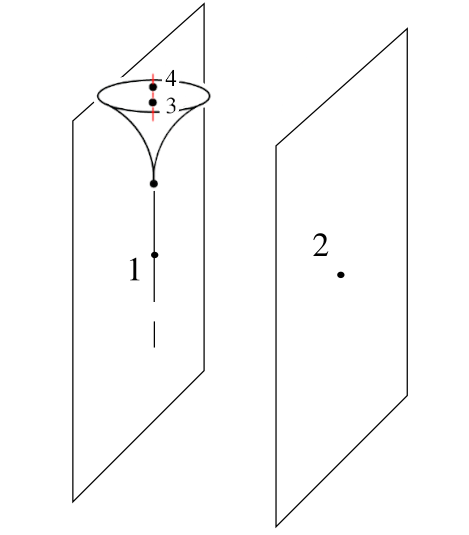}
\end{figure}
If we consider a small perturbation instead of an infinitesimal one, we will obtain a point in $\Conf(1 <_2 3 <_2 4 <_0 2)$. For this reason we consider the definition 
$$d_3( 1 <_2 3 <_0 2) = 1 <_2 3 <_2 4 <_0 2 \ .$$ 
Let us stress that for a general doubling, all labels greater than the doubled label have to be raised by one, for example:
$$d_1( 1<_2 3 <_0 2) = 1<_2 2 <_2 4 <_0 3\ .$$
On the other hand, if we apply the "extremal coface" $d_0$ to $\phi(x_1, x_2, x_3)$, we will get an infinitely distant point in the opposite direction of $\mu$.  
Susbstituting "infinitely distant" with "very far" we get a configuration of points. 
In the case $\mu=e_1$,
the configuration would be in $\Conf(1<_0 2 <_2 4 <_0 3)$. 
We can adapt this formula to $\mu =e_3$ by attaching the new point to the leftmost point, but with a vertical displacement : %poco chiaro
$$d_0(1 <_2 3 <_0 2) = 1 <_2 2 <_2 4 <_0 3 \ .$$ 
This choice turns out to respect the cosimplicial identity $d_0 d_0 = d_1 d_0$, so we take it as a definition. An analogous heuristics in the opposite direction suggests 
$$d_4(1 <_2 3 <_0 2) = 1 <_2 3 <_0 2 <_2 4 \ . $$
At last, note that the definition of $s_j$ readily generalizes to Fox-Neuwirth trees: it consists in the elimination of the $(j+1)$-th point. 
The conjectured cosimplicial structure can be illustrated in the following way:
\begin{figure}[h]
	\centering
	\includegraphics{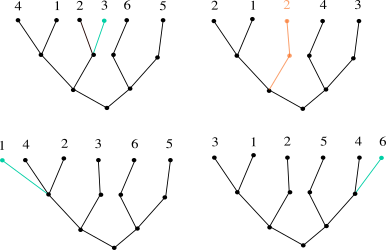}
	\caption{Coface $d_2$, codegeneracy $s_1$, external cofaces of the cell $3 <_2 1 <_1 2 <_0 5 <_1 4$}
	\label{fox action}
\end{figure}

\subsection{On Fox-Neuwirth posets}
\label{fnposet}
We want to define a cosimplicial structure on the Fox-Neuwirth posets that turns $\FNP_m(n)$ into a functor
$$\FNP_m( \bullet) : \Delta \to \textrm{Poset} \ .$$ 
Let us define the image of standard generators\footnote{We denote by $d_i, s_j$ both the generators of $\Delta$ and the cosimplicial structure maps; the correct interpretation can be deduced from the context.} of $\Delta$:

\begin{definition} \label{fn-cosimplicial}
Given a depth-ordering $\Gamma =(\sigma, a) \in \FNP_m(n)$, we construct the depth-orderings $ d_i \Gamma, s_j \Gamma$  in the following way. Let $\alpha := \sigma^{-1}(i)$ and $\beta := \sigma^{-1}(j+1)$.
\begin{itemize}
\item \textit{Internal cofaces, $0 < i < n+1$}. 
$$ d_i(\sigma)(k) = \left\{ \begin{array}{ll}
        i,  & \textrm{if }  k =\alpha\\
       d_i (\sigma( s_{\alpha}(k))) , & \textrm{ otherwise} 
    \end{array} \right. \ \ \ \ \ \ d_i(a)_k= \left\{ \begin{array}{ll}
        m-1,  & \textrm{if }  k =\alpha\\
       a_{s_{\alpha}(k)} , & \textrm{ otherwise} 
    \end{array} \right. 
    $$
\item \textit{Left extremal coface, $i=0$}
$$ d_0(\sigma)(k) = \left\{ \begin{array}{ll}
        1,  & \textrm{if }  k =1\\
       \sigma(k-1)+1 , & \textrm{ otherwise} 
    \end{array} \right. \ \ \ \ \ \ d_0(a)_k= \left\{ \begin{array}{ll}
        m-1,  & \textrm{if }  k =1\\
       a_{k-1} , & \textrm{ otherwise} 
    \end{array} \right.
    $$
\item \textit{Right extremal coface, $i=n+1$} 
$$ d_{n+1}(\sigma)(k) = \left\{ \begin{array}{ll}
        n+1,  & \textrm{if }  k =n+1\\
       \sigma(k) , & \textrm{ otherwise} 
    \end{array} \right.  \ \ \ \ \ \ d_{n+1}(a)_k= \left\{ \begin{array}{ll}
        m-1,  & \textrm{if }  k = n\\
       a_k , & \textrm{ otherwise} 
    \end{array} \right.
    $$
\item \textit{Codegeneracy, $0 \le j \le n-1$}
$$ s_j(\sigma)(k) = s_{j+1}( \sigma ( d_{\beta}(k) ) )   \ \ \ \ \ \  s_j(a)_k= \left\{ \begin{array}{ll}
        \min \{ a_{\beta-1}, a_{\beta} \},  & \textrm{if }  k = \beta-1 \\
       a_{ d_{\beta}(k) }, & \textrm{ otherwise} 
    \end{array} \right.
    $$
\end{itemize}
\end{definition}
\begin{remark} Informally, the cosimplicial structure acts in the following way:
\begin{itemize}
\item The $i$-th internal coface substitutes the branch labeled $i$ with a small fork labeled $i, i+1$;
\item The left (resp. right) extremal coface inserts $1$ (resp. $n+1$) together with a small fork, at the left (resp. right) of the tree;
\item The $j$-th codegeneracy eliminates the branch labeled $(j+1)$. 
\end{itemize}
\end{remark}
We verify that this definition actually yields a cosimplicial structure. 
\begin{lemma} For any $\Gamma \in \FNP_m(n)$, we have:
$$0 \le i < j \le n: \ \ \  d_j d_i \Gamma = d_i d_{j-1} \Gamma \ ,$$
$$0 \le j < i \le n: \ \ \  s_j s_i \Gamma = s_{i-1} s_j \Gamma \ ,$$
$$0 \le i < j \le n: \ \ \  s_j d_i \Gamma = \left\{ \begin{array}{ll}
       d_i s_{j-1},  & \textrm{if }  i <j \\
       \textrm{Id} & \textrm{if }  i \in {j,j+1} \\
       d_{i-1} s_j,  & \textrm{if }  i >j+1 \\ 
    \end{array} \right.$$
\end{lemma}
\begin{proof} The lemma can be proved directly using explicit formulas, and it is left to the reader. In what follows, we sketch a more conceptual explanation that might clarify why the result holds.  \newline

The cosimplicial structure is a sort of "semidirect product" of cosimplicial structures on the two components (permutation and depth): 
\begin{itemize}
\item The action on permutations $(d^{\Sigma}, s^{\Sigma})$ does not depend on depths, and can be described as the McClure-Smith construction applied to the $\textbf{Ass}$ operad with multiplication $\mu = 12$ (see Section \ref{kons-cosimp});
\item The action on depths  $(d^{\mathcal{D}}, s^{\mathcal{D}})$ adds or remove a fixed value ($m-1$) in the tuple of depths, in a position depending on the permutation. When removing a value, the depth has to be adjusted so that the resulting tuple still respects the min-rule of a depth order.
\end{itemize}
Informally, we can write:
$$ d_i(\sigma, a) = (d_i^{\Sigma}(\sigma), d^{\mathcal{D}}_{\sigma^{-1}(i)}(a) ), \ \ \ s_j(\sigma, a) = (s_j^{\Sigma}(\sigma), s^{\mathcal{D}}_{\sigma^{-1}(j+1)}(a)) $$
The commutation on the first component will follow automatically from the commutation of the McClure-Smith construction. One can check that $d^{\mathcal{D}}, s^{\mathcal{D}}$ actually define a cosimplicial structure on $D(m)_n := \{0, \ldots, m-1\}^{n-1}$. On the second component, we will have to check that\footnote{Cosimplicial maps in the subscripts have the $\Sigma$ dropped for notational convenience.}:
$$ d^{\mathcal{D}}_{ d_i(\sigma)^{-1}(j)} d^{\mathcal{D}}_{\sigma^{-1}(i)} = d^{\mathcal{D}}_{d_j(\sigma)^{-1}(i)} d^{\mathcal{D}}_{\sigma^{-1}(j-1)} $$
Let us set $\alpha := \sigma^{-1}(i)$ and $\beta = \sigma^{-1}(j-1)$. Suppose that $i$ appears before (or is equal to ) $j-1$ in the depth-order defined by $\Gamma$, that is $\alpha \le \beta$. It is easy to see that
$$ d_i(\sigma)^{-1}(j) = \beta+1, d_j(\sigma)^{-1}(i) = \alpha $$
At this point, the verification reduces to $d^{\mathcal{D}}_{\beta+1} d^{\mathcal{D}}_{\alpha} = d^{\mathcal{D}}_{\alpha} d^{\mathcal{D}}_{\beta} $ which is a cosimplicial identity. On the other hand, if $\alpha > \beta$, we have that
$$ d_i(\sigma)^{-1}(j) = \beta, d_j(\sigma)^{-1}(i) = \alpha+1 $$
reducing the verification to $d^{\mathcal{D}}_{\beta} d^{\mathcal{D}}_{\alpha} = d^{\mathcal{D}}_{\alpha+1} d^{\mathcal{D}}_{\beta} $, which is again a cosimplicial identity.
\end{proof}

\subsubsection{On Fox-Neuwirth chains} \label{nerve-fn}
In this section, we show how to turn the combinatorial model we constructed into a semicosimplicial chain complex. Recall that a poset is in particular a category. The nerve functor
$$ \Nerve: \textrm{Cat} \to \sSet $$
provides a way to turn our posets into simplicial sets. By post-composition we get a functor
$$ \Nerve(\FNP_m( \bullet)) : \Delta_s \to \sSet \ .$$
Taking normalized simplicial chains $NC_*^{\Delta}: \sSet \to \Chains$ we get a functor
$$ NC_*^{\Delta}(\Nerve(\FNP_m( \bullet))) : \Delta_s \to \Chains \ . $$
Since the nerve of a face poset is usually associated to its barycentric subdivision, we denote by
$$ \FNC_m( \bullet) : \Delta_s \to \Chains $$
the resulting semicosimplicial chain complex. Explicitly, $\FNC_m(n)_d$ is generated by chains $\Gamma_0 < \ldots < \Gamma_d$ of Fox-Neuwirth trees $\Gamma_i \in \FNP_m(n)$. The differential is given by
$$\partial ( \Gamma_0 < \ldots < \Gamma_d) = \sum_{i=0}^d (-1)^i ( \Gamma_0 < \ldots < \hat{\Gamma}_i < \ldots < \Gamma_r )\ . $$
This is the chain-level manifestation of the fact that the face poset of the simplicial set $\Nerve(\FNP_m(n))$ is made up of chains $\Gamma_0 < \ldots < \Gamma_d$ ordered by refinement.
\subsubsection{On subdivided Blagojevic-Ziegler space} \label{geom-model}
Let us connect the algebraic and the geometric model. By the very definition (see section \ref{configurations}), we have
$$(*) \ \ \  \BZ_m(n) \cong | \Nerve(\FNP_m(n)) | \ ,$$
so that, at the level of chains, there is a quasi isomorphism
$$ \FNC_m(n)\simeq NC_*^{\Delta}(\Nerve(\FNP_m(n))) \to C_*(\BZ_m(n) ) \ .$$
Also, the geometric realization homeomorphism (*) gives us a way to endow $\BZ_m(n)$ with a geometric semicosimplicial structure. Indeed, by post-composing the functor $ \FNP_m(\bullet) : \Delta_s \to \textrm{Poset} $ with the functor $|\Nerve( -) | : \textrm{Poset} \to \textrm{Top}$ we get a semicosimplicial space $\Delta_s \to \textrm{Top}$. By transporting along the isomorphisms (*) we get a semicosimplicial structure $\BZ_m : \Delta_s \to \textrm{Top}$. Explicitly, the morphism $\phi \in \Delta$ sends the vertex $v(\sigma,i)$ to $v( \phi_*(\sigma,i))$, where $\phi_*$ is the semicosimplicial action on the Fox-Neuwirth cells. This map can be extended to faces by linear interpolation, and the abstract construction above ensures that the map is actually well-defined. A generalized proof that the linear combination of such configurations is still a configuration is explicitly given by the "walking-man formula" \cite{marinofox}.

\subsubsection{Associated bicomplex and multicomplex} \label{bicomplex}
Recall that the homotopy totalization of a (semi)cosimplicial chain complex can be obtained in the following way (see \cite{bunke2013differential}, 4.23). Firstly, construct a bicomplex by taking the alternated sum $\sum_i (-1)^i d_i$ in the cosimplicial direction; secondly, take the total complex associated to this bicomplex. The homology of such total complex is approximated by the spectral sequence of the original bicomplex.

If we apply this procedure to $\FNC_m(n)$, we get the following
\begin{definition} \label{bar-fn} For any $m \ge 2$, the \textit{Barycentric Fox-Neuwirth Bicomplex} with coefficients in an abelian group $A$ is defined as
$$ \BFNA_{dn} = \FNC_m(-n)_d \otimes A \ ,$$
with vertical, downward differential induced by
$$ \partial_v( \Gamma_0 < \ldots < \Gamma_d) = \sum_i (-1)^i \Gamma_0 < \ldots < \hat{\Gamma}_i < \ldots < \Gamma_d \ ,$$
and horizontal, leftward differential defined as
$$ \partial_h(\Gamma_0 < \ldots < \Gamma_d) = \sum_i (-1)^i d_i \Gamma_0 < \ldots < d_i \Gamma_d \ .$$
Here $\Gamma_0 < \ldots < \Gamma_d$ is a chain of Fox-Neuwirth trees with $n$ leaves and height $m$. The spectral sequence that starts by taking vertical homology is denoted by $E_{pq}^r(\FNC_m,A)$. The horizontal spectral sequence associated to the \textit{dual} bicomplex $\BFNA^{\vee}$ is denoted by $ E^{pq}_r(\FNC_m, A)$.
\end{definition}

\begin{comment}
\footnote{Mind the difference with \cite{tourtchine2005bialgebra}: since we will use almost only the homology spectral sequence, we preferred to adopt the natural notation $E^{pq}_r$ for the most used spectral sequence.} 
\end{comment}
In \cite{marinofox}, we prove the following
\begin{theorem} \label{barycentric-spseq} The spectral sequence associated to the bicomplex $\BFNA_{dn}$ is isomorphic to the Sinha Spectral Sequence with $A$-coefficients from first page on.
\end{theorem}

However, the number of chains of Fox-Neuwirth trees increases fastly with $d,n$. In the next section, we will introduce a method to deform the Barycentric Fox-Neuwirth Bicomplex into a smaller structure, that of a \textit{truncated multicomplex}.

\subsection{The Barycentric Deformation Theorem}
Before introducing the theorem, we need to define the notion of truncated multicomplexes. It can be thought as an analogue of truncated simplicial sets. This framework allows us to compute the first $r$ pages of the spectral sequence associated to a multicomplex, even if we only know the first differentials $D_0, \ldots, D_r$. A similar phenomenon happens with simplicial sets, in that we are able to compute the first $(d-1)$ degrees of the homology from the $d$-skeleton. 
\subsubsection{Truncated Multicomplexes}
The definitions are similar to the ones given in subsection \ref{multicpx-back}.

\begin{definition} Let $k$ be a commutative unital ring. A \textit{truncated} $r$-multicomplex over $k$ is a $(\mathbb{Z}, \mathbb{Z})$-graded $k$ module $C_{**}$ equipped with maps $D_i: C_{**} \to C_{**}$ for $0 \le i \le r$ with bidegree $|D_i| = (-i,i-1)$ such that
$$ \sum_{i+j=k} D_i D_j = 0 \ \ \ \textrm{for all } 0 \le k \le r .$$
A morphism of truncated $r$-multicomplexes $f: (C_{**}, D_i) \to (C_{**}', D_i')$ is given by maps $f_i : C_{**} \to C_{**}'$ for $0 \le i \le r$ of bidegree $|f_i| = (-i,i)$ satisfying
$$ \sum_{i+j = k} f_i D_j = \sum_{i+j=k} D'_i f_j \ \ \textrm{for all } 0 \le k \le r \ .$$
\end{definition}

\begin{remark} Note that the definition is slightly different from the one of a $r$-multicomplex as presented in \cite{multicpx-model}. Despite the definitions both involve a truncated differential structure, a truncated $r$-multicomplex is only subjected to equations with $k\le r$. For example, while a $1$-multicomplex is a bicomplex, a truncated $1$-multicomplex satisfies $D_0^2=0, D_1D_0=D_0D_1$ but not $D_1^2=0$.
\end{remark}

Referring to the construction \ref{multicpx-tot}, one can see that the construction of $k$-cycles and $k$-boundaries only uses differentials up to degree $p$. This motivates the following
\begin{definition} An $r$-truncated spectral sequence is the datum of abelian groups $E^k_{pq}$ for $k\le r$ and differentials $\delta^k : E^k_{pq}\to E^k_{p-k,q+k-1}$ satisfying $(\delta^k)^2 = 0$, together with isomorphisms  $\alpha_{pq}^k: H_{\bullet}(E^k_{pq}, \delta^k)\to E^{k+1}_{pq}$ for $k\le r-1$.

A morphism of $r$-truncated spectral sequences is given by morphisms $f_{pq}^k : E_{pq}^k \to (E^k_{pq})', k\le r$ that commutes with the isomorphisms $\alpha_{pq}^k$ for all $k\le r-1$. 
\end{definition}

Therefore, we have:
\begin{lemma} Given a truncated $r$-multicomplex, the quotients $E_{pq}^k := Z_{pq}^k / B_{pq}^k $ define a  $r$-truncated spectral sequence. Formulas for cycles, boundaries and differentials are given in \ref{multicpx-tot}. 
\end{lemma}

As in the case of usual spectral sequences, we have the following
\begin{lemma} \label{tr-sp-seq}Given a morphism of truncated $r$-multicomplexes $f: C \to C'$, there exists an induced morphism of spectral sequences $E(f) : E(C) \to E(C')$. Moreover, if $E^k_{pq}(f)$ is an isomorphism for all $p,q$ and a given value of $k$, then $E^s_{pq}(f)$ is an isomorphism for all $r \ge s \ge k$. 
\end{lemma}

\begin{proof} The first part can be proved exactly as in the case of multicomplexes. It descends from the definition of $k$-cycles and $k$-boundaries being functorial.

The second part is a classical fact for spectral sequences that carries over to the truncated case without modifications. 
\end{proof}

At last, let us notice that there is a truncation functor $\tau_r$ from the category of multicomplexes to the category of truncated $r$-multicomplexes. It is given by simply forgetting the differentials $D_k, k \ge r+1$. The same truncation functor exists for spectral sequences, forgetting the pages with $k \ge r+1$. As the spectral sequence of a multicomplex up to page $k$ only relies on the first $k$ differentials, we have:
\begin{lemma} \label{truncation} For any multicomplex $C$ and integer $k\ge 0$, the following isomorphism of $k$-truncated spectral sequences holds:
$$ E(\tau_k C) \cong \tau_k E(C) $$
\end{lemma}

\subsubsection{Framework and statement of the theorem}

Consider the following setting:
\begin{itemize}
\item A sequence of regular CW-complexes $X_n$. 
\item The face poset $\mathcal{P}_n$ of closed cells in $X_n$, ordered by inclusion.
\item A semicosimplicial structure on $\mathcal{P}_n$, that is a functor $\mathcal{P}_{\bullet}: \Delta_s \to \textrm{Poset}$.
\item The barycentric subdvision $\sse{X}_n := \Nerve(\mathcal{P}_n)$ of the regular CW-complexes.
\item A ring of coefficients $k$; for the sake of simplicity, we will restrict to $k=\mathbb{F}_2$.
\item The chain complex $C_n(\bullet) := C^{CW}_{\bullet}(X_n,k)$, and its subdivided version $K_n(\bullet):=NC_{\bullet}^{\Delta}(\sse{X}_n,k)$ of normalized simplicial chains;
\end{itemize}
Note that $\sse{X}_{\bullet}$ inherits a semicosimplicial structure from $\mathcal{P}_{\bullet}$ by applying the nerve construction. Our goal is to present the Bousfeld-Kan Spectral Sequence of $\sse{X}_{\bullet}$ in terms of cells in $C_n(\bullet)$.
\

Before stating the theorem, we need a few definitions.
\begin{definition} \label{subdivision-sign} Consider a regular $CW$ complex $X$ and its barycentric subdivision $X^{sd}$. Suppose that for any cell $\sigma$ in $X$
$$\partial(\sigma) = \sum_{\tau \lhd \sigma} \epsilon(\tau, \sigma) \tau \ ,$$
for some coefficients $\epsilon(\tau,\sigma)=\pm 1 \in \mathbb{Z}$. Then the subdivision map
$$\sd : C_*^{CW}(X) \to C_*^{\Delta}(X^{sd}) $$
 is given by
$$ \sd(\sigma) = \sum_{\bt \tau \in \mc(\sigma) } \epsilon(\bt \tau) \bt \tau  \ ,$$
where $\mc(\sigma)$ is the set of maximal chains $\tau_0 \lhd \ldots \lhd \tau_d = \sigma$, and  $\epsilon(\bt \tau) := \epsilon( \tau_0, \tau_1) \ldots \epsilon(\tau_{d-1}, \tau_d)$.
\end{definition}
An easy calculation shows that this must be the subdivision formula
 for the equation $\sd \partial = \partial \sd$ to hold. 
\begin{definition} For $I = \{i_1, \ldots, i_p\} \subset [n]$ a multiset with $i_1 \le \ldots \le i_p$ define $d_I : [n] \to [n+p]$ as $d_{i_1} \ldots d_{i_p}$. By convention, let us set $d_{\emptyset} = \textrm{Id}$.
\end{definition}
Referring to the setting given at the beginning of the subsection:
\begin{definition} For $\sigma \in \mathcal{P}_n$, define 
$$\mathcal{P}(\le \sigma) := \{\sigma' \in \mathcal{P}_n: \sigma' \le \sigma \} \ .$$
\end{definition}

\begin{definition} \label{mu-notation} For $I \subset [n]$ a multiset and $\sigma \in \mathcal{P}_n$, define $\mu_I(\sigma)$ as the $\mathbb{F}_2$ vector space generated by $\mathcal{P}( \le d_I \sigma) $. By convention, set $\mu(\sigma) := \mu_{\emptyset}(\sigma)$.
\end{definition}

\begin{definition} For $\sigma \in \PP_n$ and $I$ a multiset of size $|I|=p$, define $\tau_I(\sigma)$ as the augmented chain complex that has:
\begin{itemize}
\item The $\mathbb{F}_2$ vector space generated by $\Nerve \PP(\le d_I \sigma)$ as the underlying vector space;
\item The grading and the augmentation inherited by $K_{n+p}(\bullet)$.
\end{itemize}

\end{definition}

Let us outline a simple observation about $\tau_I(\sigma)$.
\begin{lemma} \label{contractible} The augmented chain complex $\tau_I(\sigma)$ is an augmented acyclic complex
, i.e. the kernel of the augmentation has trivial homology. 
\end{lemma}
\begin{proof} Since the category $\PP(\le d_I \sigma)$ has a final element, its nerve $\Nerve \PP(\le d_I \sigma)$ is a contractible simplicial set. It follows that its simplicial chain complex  $\tau_I(\sigma)$ is an augmented acyclic chain complex.
\end{proof}

We now turn our attention to the combinatorics of multisets.
\begin{definition} \label{union-with-shift}Consider two finite multisets $J \subset [n], I \subset [n+p]$, where $p = |J|$. We define their \textit{union with shift} $I \vee J$ as the multiset $L$ such that $d_I d_J = d_L$.
\end{definition}

\begin{example} Consider $I = \{1,1,4\}$ and $J= \{1,2\}$. Let us use the cosimplicial identity
$$ d_j d_i = d_i d_{j-1} , \ \ j > i ,$$
in order to move cofaces with large indices from right to left in the expression $d_I d_J$. Highlighting in bold the letters where we use the aforementioned identity, we can write:
$$ d_I d_J = d_1 d_1 {\textbf d_4 \textbf d_1} d_2 = d_1 d_1 d_1 {\textbf d_3 \textbf d_2} = d_1 d_1 d_1 d_2 d_2  \ .$$
Thus $I \vee J = \{1,1,1,2,2\}$.
\end{example}

\begin{theorem}[Barycentric Deformation Theorem] \label{bdthm}Fix a number $r\ge 0$. Consider a collection of maps
$$D_I : C_n(d) \to C_{n+i}(d+i-1) $$
for all $d \ge 0, I \subset [n+1], n \ge 0, |I| = i $ with $0\le i\le r$. Suppose that:
\begin{enumerate}
    \item For all multisets $K$
$$ \sum_{I \vee J = K} D_I D_J = 0 \ .$$
    \item For all cells $\sigma \in \mathcal{P}_n$, we have $D_I(\sigma) \in \mu_I(\sigma)$.
\end{enumerate}
    Then there exists an $r$-truncated multicomplex structure on $C$ with differentials given by $D_i = \sum_{|I| = i} D_I$, and a map of $r$-truncated multicomplexes $  \sd_{\bullet} :  C \to  \tau_r K$, such that $\sd_0$ is the standard subdivision map on rows. In particular, for all $1\le s \le r$ and all $p,q$ we have an isomorphism
$$ E^{pq}_s(C , D_i) \cong E^{pq}_s \left ( K, \sum (-1)^i d_i \right ) \ .$$
\end{theorem}

\begin{remark} \label{strategy-barycentric} The idea to show the theorem is the following. We set $D_i := \sum_{|I|=i} D_I$, and the first condition yields the (truncated) multicomplex equation. Likewise, we look for maps $\sd_I$, indexed by appropriate multisets $I$, satisfying
$$ \sum_{I \vee J} \sd_I D_J = \sum_{I \vee J} D_I \sd_J \ .$$
This is equivalent to
$$ \sum_{\substack{I \vee J = K\\ J \neq \emptyset}} \sd_I D_J -  \sum_{\substack{I \vee J = K\\ I \neq \emptyset}} D_I \sd_J = D_0 \sd_K - \sd_K D_0 $$
 that is, the left-hand side is null-homotopic. As the left-hand side only involves multisets with cardinality smaller than $|K|$, this provides a strategy to find $\sd_K$ inductively as null-homotopies of a given map. The algebraic tool employed to show a map is null-homotopic is the acyclic covering lemma, presented in the next subsection 
\end{remark}
\subsubsection{Acyclic Covering Lemma}
\begin{lemma} \label{acyclic-covering} Let $X_{\bullet}, Y_{\bullet}$ be two 
non-negatively chain complexes (over $\F_2$) and $f: X_{\bullet} \to Y_{\bullet}$ a  
chain map. 
Fix a homogeneous basis $\mathcal{B}$ of $\bt X$. 
Consider a family of acyclic subcomplexes $\tau(x) \subseteq Y_{\bullet}$ 
such that $\tau(x) \subseteq \tau(y)$ if $x$ is a summand of $\partial(y)$. 
Suppose that $f(x) \in \tau(x)$ for all $x \in \mathcal{B}$. Then there exists a homotopy $h$ from $f$ to $0$, 
that is $h : \bt X \to \bt Y[1] $ such that
$$f(x) = h\partial(x) +  \partial h(x) \ ,$$
for all $x \in X_{\bullet}$, and furthermore $h(x) \in \tau(x)$.
\end{lemma}

\subsubsection{Main Proof}

We are ready to prove Theorem \ref{bdthm}.

\begin{proof} Recall Remark \ref{strategy-barycentric} regarding the strategy of the proof. The first subdivision map $\sd_0$ is given by Def \ref{subdivision-sign}. The proof is articulated in three steps.

 \textbf{First step.} In this part, we show the existence of maps indexed by multisets that contribute to the higher subdivision maps. Specifically, we prove that for any finite multiset $L$ of size $|L| = \ell$ with $r \ge \ell \ge 1$ there exists a map\footnote{Not a \textit{chain} map, since it is the higher analog of a homotopy.} $\sd_L: C_n \to K_{n+\ell}[\ell]$ such that 
$$ \sum_{I \vee J = L } \sd_I D_J  = \sum_{I \vee J = L } D_I \sd_J  \ ,$$
$$ \sd_L(\sigma) \in \tau_L(\sigma) \ \ \forall \sigma \in \FNP_m(n) \ .$$
Here for symmetry we are using symbols $D_I$ also for linear self-maps of $K_n(\bullet)$ so that $D_0=D_\emptyset$ is its differential, $D_{\{i\}}=d_i$ is a co-degeneracy and $D_I= 0$ for $|I|>1$.

We show the formula above by induction on $\ell$. If $\ell=0$, we are forced to choose $I=J=L = \emptyset$, thus the equation boils down to
$$ \sd_0 D_0 = D_0 \sd_0 \ ,$$
which is true because the subdivision map commutes with the differential. It is straightforward to see from the explicit expression that each chain contributing to $\sd_0(\sigma)$ is bounded from above by $\sigma$, which shows $\sd_0(\sigma) \in \tau_{\emptyset}(\sigma)$. 

Now suppose the thesis is true for all $L'$ of size strictly less than $\ell$. Define 
$$ f_L := \sum_{\substack{I \vee J = L \\J \neq \emptyset } } \sd_I D_J + \sum_{\substack{I \vee J = L \\I \neq \emptyset } }D_I \sd_J  \ .$$
It is well defined since it only uses $\sd_{K}$ with $K$ of size smaller than $\ell$. Recall that we are working over $\mathbb{F}_2$, so that signs are irrelevant. Note that the thesis is equivalent to 
$$ f_L = \sd_L D_0 + D_0 \sd_L \ , \ \ (*)$$
with $\sd_L$ such that $\sd_L(\sigma) \in \tau_L(\sigma)$ for all $\sigma \in \PP_n$. We want to use Lemma \ref{acyclic-covering} with
$$\bt X = C_n, \ \ \ \bt Y= K_{n+\ell}[\ell-1], \ \ \ \mathcal{B}=\PP_n, \ \ \ \tau (\sigma) = \tilde{\tau}_L(\sigma)[\ell-1], \ \ \ f=f_L \ ,$$
where $\tilde{\tau}_L(\sigma)$ is the kernel of the augmentation
$\tau_L(\sigma) \to \F_2$.

We have to check that $f_L(\sigma) \in \tilde{\tau}_L(\sigma)[l-1]$. For $l=1$ this follows from the fact
that the subdivision map $sd_0$ is a map of augmented chain complexes. For $l>1$ this is obvious for degree reasons.
By Lemma \ref{contractible}, we have that $\tilde{\tau}_L(\sigma)$ is acyclic.

It is not trivial, yet true, to see that $f_L$ is a chain map. Let us compute separately $D_0 f_L$ and $f_L D_0$, using the inductive hypothesis and the multicomplex law:
\begin{align*}
D_0 f_L & =  \sum_{ \substack{ I \vee J = L \\ I \neq \emptyset}} D_0D_I \sd_J + \sum_{ \substack{ I \vee J = L \\ I \neq \emptyset}} D_0 \sd_I D_J  \\
 & = \sum_{ \substack{ I \vee J = L \\ I \neq \emptyset}} \left ( D_I D_0 \sd_J + \sum_{ \substack{ A \vee B = I \\ A,B \neq \emptyset}} D_A D_B \sd_J \right ) \\
& + \sum_{ \substack{ I \vee J = L \\ J \neq \emptyset}} \left ( \sd_I D_0 D_J + \sum_{ \substack{ A \vee B = I \\ B \neq \emptyset}} \sd_A D_B D_J + \sum_{ \substack{ A \vee B = I \\ A \neq \emptyset}} D_A \sd_B  D_J \right )  \\
 & = \sum_{ \substack{ I \vee J = L \\ I \neq \emptyset}}  D_I D_0 \sd_J + \sum_{ \substack{ A \vee B \vee C = L \\ A,B \neq \emptyset}} D_A D_B \sd_C \\
& + \sum_{ \substack{ I \vee J = L \\ J \neq \emptyset}} \sd_I D_0 D_J + \sum_{ \substack{ A \vee B \vee C = L \\ B,C \neq \emptyset}} \sd_A D_B D_C + \sum_{ \substack{ A \vee B \vee C = L \\ A,C \neq \emptyset}} D_A \sd_B  D_C \ .
\end{align*}
On the other hand:
\begin{align*}
f_LD_0 & =  \sum_{ \substack{ I \vee J = L \\ I \neq \emptyset}} D_I \sd_JD_0 + \sum_{ \substack{ I \vee J = L \\ I \neq \emptyset}} \sd_I D_JD_0 \\
 & = \sum_{ \substack{ I \vee J = L \\ I \neq \emptyset}} \left ( D_I D_0 \sd_J + \sum_{ \substack{ B \vee C = J \\ B \neq \emptyset}} D_I D_B \sd_C + \sum_{ \substack{ B \vee C = J \\ B \neq \emptyset}} D_I \sd_B  D_C\right ) \\
& + \sum_{ \substack{ I \vee J = L \\ J \neq \emptyset}} \left ( \sd_I D_0 D_J + \sum_{ \substack{ B \vee C = J \\ B,C \neq \emptyset}} \sd_I D_B D_C  \right )  \\
 & = \sum_{ \substack{ I \vee J = L \\ I \neq \emptyset}} D_I D_0 \sd_J + \sum_{ \substack{ A \vee B \vee C = L \\ A,B \neq \emptyset}} D_A D_B \sd_C + \\
 & +\sum_{ \substack{ A \vee B \vee C = L \\ A,C \neq \emptyset}} D_A \sd_B  D_C + \sum_{ \substack{ I \vee J = L \\ J \neq \emptyset}} \sd_I D_0 D_J + \sum_{ \substack{ A \vee B \vee C = L \\ B,C \neq \emptyset}} \sd_A D_B D_C \ .
\end{align*}
We are left with showing that $f_L(\sigma)$ is contained in $\tau_L(\sigma)$ for all $\sigma \in \PP_n$. Let us start with the term on the left. By hypothesis, we have $D_J(\sigma) \subset \tau_J(\sigma)$, so that there exists a finite set $S$ such that
$$ D_J(\sigma) = \sum_{s \in S} \Lambda_s  \ ,$$
with $\Lambda_s \le d_J \sigma$. Applying $\sd_I$ we get
$$ \sd_I D_J(\sigma) = \sum_{s \in S} \sd_I \Lambda_s  \ .$$
By inductive hypothesis we have $\sd_I(\Lambda_s) \in \tau_I(\Lambda_s)$. Since $\tau_I$ is monotone with respect to the argument and $\Lambda_s \le d_J\sigma$, we have 
$$\sd_I(\Lambda_s) \in \tau_I(d_J \sigma)= \mathbb{F}_2 \Nerve \PP(\le d_I d_J\sigma) = \mathbb{F}_2 \Nerve \PP(\le d_L \sigma) = \tau_L(\sigma) \ .$$
Now let us analyze the term on the right. We have that $\sd_J(\sigma) \in \tau_J(\sigma)$, thus there exists a finite set $S$ such that
$$ \sd_J(\sigma) = \sum_{s \in S} \Lambda_s \ ,$$
with $\Lambda_s \le d_J \sigma$. If $|I| \ge 2$, $D_I$ is zero and there is nothing to show. If $I= \{i\}$, then 
$$ D_I \sd_J(\sigma) = d_i \sum_{s \in S} \Lambda_s = \sum_{s \in S} d_i \Lambda_s \in \tau_L(\sigma)  \ .$$
There are no other cases since $I \neq \emptyset$. 

We can then apply Lemma \ref{acyclic-covering} and obtain the existence of a map $\sd_L : C_n \to K_{n+\ell}[\ell-1]$ that satisfies the equation $(*)$. 

\textbf{Second step.} Here, we put together the contributions coming from multiset-indexed subdivision maps and show that they provide higher subdivision maps. Set
$$ \sd_i := \sum_{ |I| = i } \sd_I  \ .$$
We have to show, for all $r \ge k \ge 1$, that
$$ \sum_{i+j=k} \sd_i D_j +  \sum_{i+j=k} D_i \sd_j = 0 \ .$$
Let us use the splitting into multiset-indexed maps:
\begin{align*}
& \sum_{i+j=k}\left ( \sum_{ |I| = i} \sd_I \right ) \left ( \sum_{|J|=j} D_J \right )  + \sum_{i+j=k} \left ( \sum_{|I|=i} D_I \right )  \left ( \sum_{ |J| = j} \sd_J \right ) = \\
& \sum_{i+j=k} \sum_{ |I| = i, |J| = j} \sd_I D_J  + \sum_{i+j=k}  \sum_{|I|=i, |J| = j } D_I  \sd_J  = \\
& \sum_{|K| = k } \left ( \sum_{ I \vee J = K } \sd_I D_J + \sum_{ I \vee J = K } D_I \sd_J \right ) = \boxed{0} \ .
\end{align*}
Each summand is zero because of what we showed in the first part. The second equality (reparametrization) deserves a brief explanation. Both sides parametrize couples $(I,J)$ such that $|I \vee J | = |I| + |J| = k $: in the LHS we first choose the size of $I$ ad $J$ and then which elements they contain; in the RHS we first choose $I \vee J = K$ and then how to split it in two multisets. 

\textbf{Third step.} Let us draw the conclusions regarding the spectral sequence. Since the employed map $\sd_0$ is the subdivision map on rows, it is a quasi-isomorphism by standard arguments. Thanks to Lemma \ref{tr-sp-seq},  we deduce that there is an isomorphism for all $p,q$ and $1\le s \le r$
$$ E^{pq}_s(C, D_i) \cong E^{pq}_s \left (\tau_r K, \sum (-1)^i d_i \right ) \ .$$
Lemma \ref{truncation} ensures that the same isomorphism holds without the truncation $\tau_r$, and the theorem is proved.
\end{proof}

%% file: section5-new.tex
\section{Fox-Neuwirth Multicomplex in dim 3} \label{fn-mcpx-3}
Before diving into the rather technical description of differentials, we want to motivate their definition geometrically. In Section \ref{fox-intuition}, we sketched how Kontsevich cosimplicial structure would act on Fox-Neuwirth strata, if they would extend from $\Conf_n(\mathbb{R}^m)$ to $\Kons_m(n)$. Here we adopt the same approach, but we consider the cosimplicial structure on Kontsevich spaces that doubles points along the horizontal \textbf{direction}\footnote{Recall that the cosimplicial structure on $\Kons_m$ depends on the choice of a \textit{multiplication} $\mu \in  \Kons_m(2) \cong S^{m-1}$ that prescribes the direction of doubling. See Section \ref{kons-cosimp} for further details.} $e_1$. For the sake of simplicity, we will use the "bar notation" to represent Fox-Neuwirth trees: no bars mean $<_2$; one bar means $<_1$; two bars mean $<_0$. Let us consider a point $\vect{x} \in \Conf(412|53||76|8)$ as in the following picture:
\begin{figure}[H]
\centering
\includegraphics[width=5.5cm]{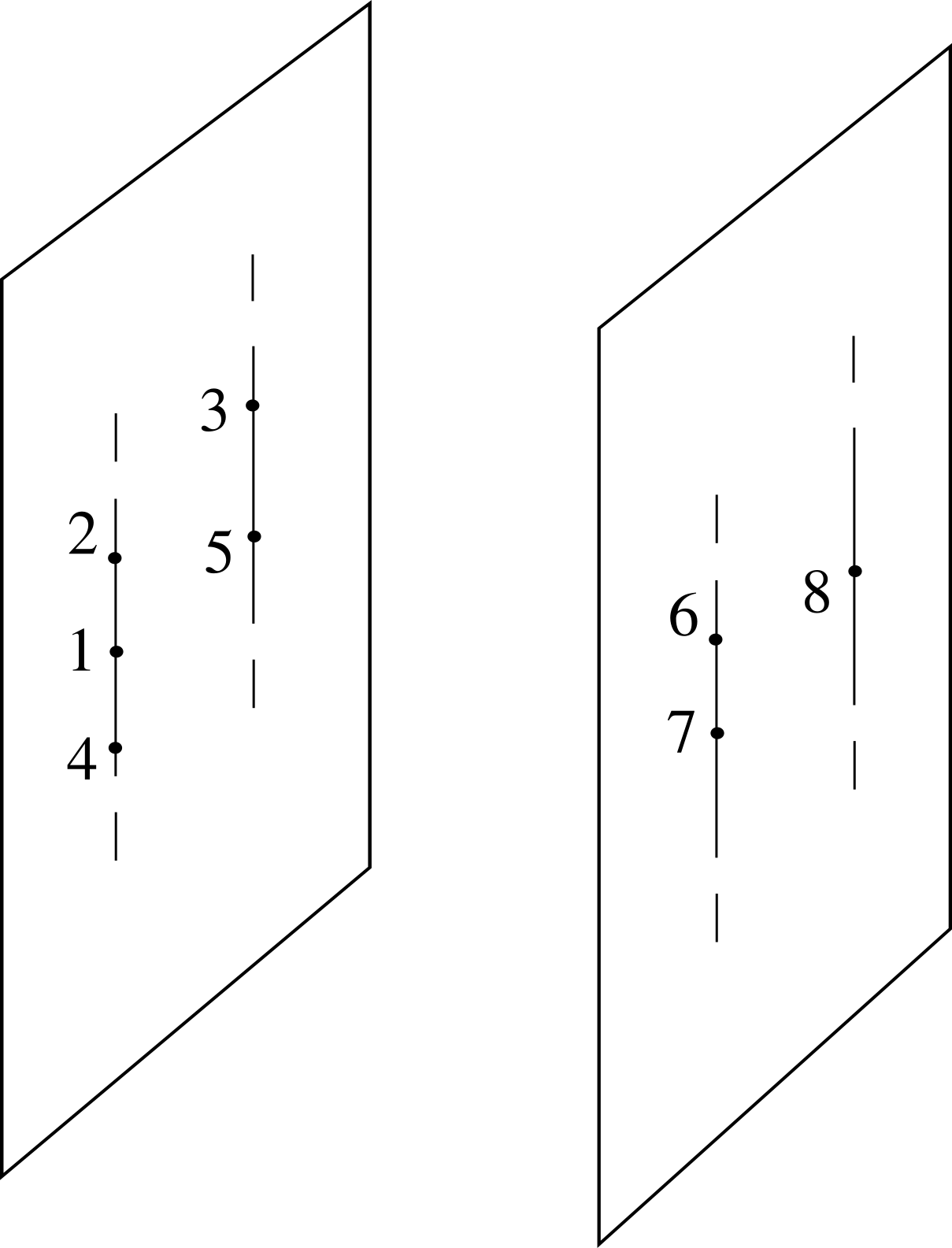}
\end{figure}
And its image $\ phi(\vect{x}) \in \Kons_3(8)$. Applying the differential $d_3$ of $\Kons_3$ we get
\begin{figure}[H]
\centering
\includegraphics[width=9cm]{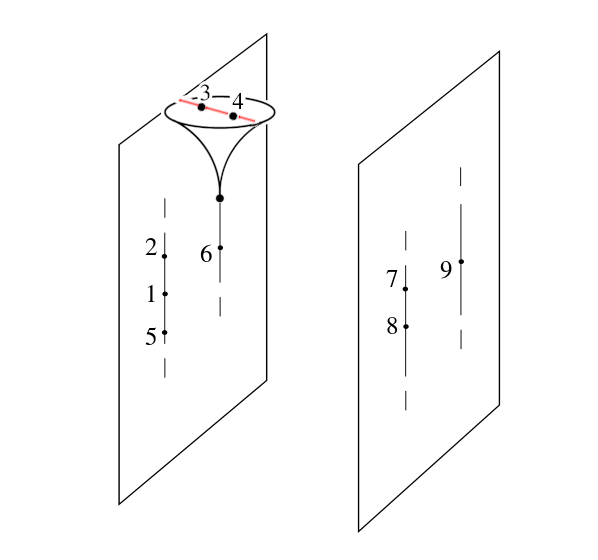}
\end{figure}
When we "resolve" the infintesimal horizontal perturbation, the new point $4$ will lie on a different plane. Consequently, we need to distribute not only the points aligned with $3$ but also those lying on the same plane. However, not all distributions are admissbile, since $D_1$ has to preserve the dimension of the cell. One can visualize the (co)dimension of a Fox-Neuwirth cell based on "anchorings". Imagine to anchor each point on a line to the point above with a vertical arrow; if there is no such point, anchor it with an horizontal arrow to the line right to it; if there is no such line, do not put any anchor. Applying this paradigm to $\vect{x}$ gives
\begin{figure}[H]
\centering
\includegraphics[width=5.5cm]{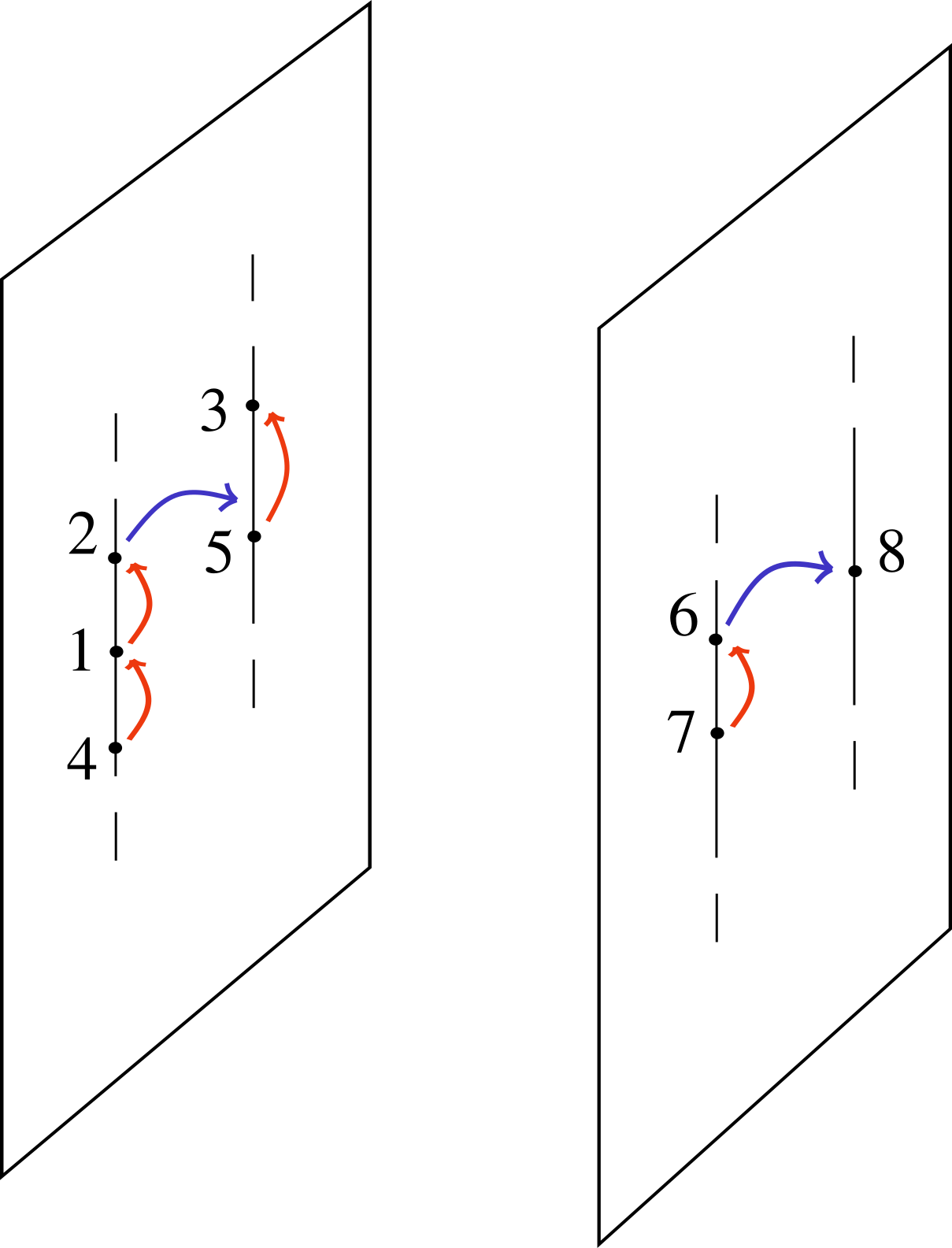}
\end{figure}
The codimension of the cell to which the point belongs is then given by
$$ 2 \times \textcolor{red}{(\textrm{\# vertical anchors})} + \textcolor{blue}{(\textrm{\# horizontal anchors})}$$
Indeed, a vertical anchor means two equal coordinates, and an horizontal anchor means one equal coordinate. If we want to leave the number of vertical and horizontal arrows unchanged when distributing, we are free to choose where to send $5$, but points $4,1,2$ should move together: otherwise, we would change a vertical anchor into an horizontal one. All in all, we make the educated guess:
$$ d_3(412|53||76|8) = 512|63||4||76|8+63||512|4||76|8+512|3||64||76|8+3||512|64||76|8$$
On the other hand, adding a point to $\pm \infty$ in the horizontal direction requires no distribution, as the new point will not share any component with the others. The sum of all such doublings, as well as the extremal contributions, turns out to provide a working formula for $D_1$.

Unfortunately, our geometric understanding does not extend to higher differentials $D_2, D_3$: not all contributions that could be sensible to include make their appearence. While a partial understanding of this phenomenon is possible for $D_2$ (see Remark \ref{d2-explanation}), things become more intricate  for $D_3$: see figure \ref{core-positions} to get a grasp of the
formula we found.

We bridged this gap with an algebraic approach. When solving for $D_2(\Gamma)$ in the multicomplex equation
$$ D_0 D_2(\Gamma) = D_2 D_0(\Gamma) + D_1 D_1(\Gamma) $$
the right-hand side only involves terms with lower differentials or with smaller Fox-Neuwirth Trees $\Lambda < \Gamma$. We can thus proceed to find $D_k(\Gamma)$ by induction on $k+a$, where $a$ is the codimension\footnote{Since we work with the reverse face poset structure, the differential $D_0$ decrease the codimension by one.} of the stratum associated to $\Gamma$. Formulas for $D_2(12), D_3(123)$ and $ D_3(12|3)$ were found in this way, and then extended to arbitrary Fox-Neuwirth trees via the distribution mechanism.
\subsection{Fox Polynomials} \label{foxpoly}
To streamline the explanation of the $D_i$ formulas, we employ a schematic presentation of "distributions" that still happen in dimension 3, which we call "Fox polynomials".
\begin{definition} Consider the following alphabet:
$$\Sigma = \{X_1, X_2, \ldots \} \cup \{Y_1, Y_2, \ldots \} \cup \{\partial Y_1, \partial Y_2, \ldots \} \cup \{Z_1, Z_2, \ldots \} \cup \{1,2,3, \ldots \} \cup \{"|"\}  \ .$$
Call $X_i$ 2-cloud variables, $Y_j$ 1-cloud variables, $\partial Y_j$ derived 1-cloud variables, $Z_k$ 0-cloud variables, numbers $1,2,3..$ called constants and $"|"$ bars. 
A Fox monomial $M$ is a word in $\Sigma$ such that
\begin{enumerate}
\item $M$ doesn't start or finish with a bar;
\item There can't be more than $2$ consecutive bars;
\item 1-clouds (derived or not) are always separated by a bar from other letters;
\item 0-clouds are separated by two bars from other letters;
\item Each 0-cloud is repeated at most once;
\item There is at most one derived variable;
\item Numbers must be distinct.
\end{enumerate}
A Fox polynomial (over $\mathbb{F}_2$) is a $\mathbb{F}_2$ linear combination of Fox monomials.
We denote by $\Fox_3[X_i, Y_j, Z_k]_{IJK}$ the set of Fox polynomials with variables $\{X_i\}_{i \in I} \cup \{Y_j, \partial Y_j\}_{j \in J} \cup \{Z_k\}_{k \in K} $.
\end{definition}

\begin{example}
The following are Fox monomials:
$$ Z_1||X_11X_22X_1, \ \ \ Y_1 | X_1 34 X_2 5 | \partial Y_2 | Y_3||Z_5, \ \ \ X_5X_7X_9||Y_1|Y_2||X_5X_7, \ \ \ Y_1|X_1iX_2jX_3|\partial Y_2 \ \ \ \textrm{ for } i\neq j \textrm{ in } \mathbb{N}_{>0}  \ .$$
\end{example}
\begin{definition} The support $|L|$ of a Fox monomial $L$ is the set of numbers appearing in $L$.
\end{definition}
\begin{definition} A Fox monomial in which variables do not appear is called numerical. If bars do not appear, it is called a numerical $2$-cloud. If double bars do not appear, it is called a numerical $1$-cloud. 
\end{definition}

We can define a few operations on Fox monomials.
\begin{definition}[Relabeling] Given an injective map $\sigma : \mathbb{N} \to \mathbb{N}$ , the map
$$\Fox_{\sigma} : \Fox_3[X_i,Y_j,Z_k]_{IJK} \to \Fox[X_i, Y_j, Z_k]_{IJK}$$
is obtained by applying $\sigma$ to all the constants appearing in the polynomial. For $i \in \mathbb{N}$, we denote as $\Delta_i$ the map $\Fox_{d_i}$, where $d_i$ is the standard map $d_i(x) = x+ 1_{x \ge i}$. For a finite set of indices $i_1 < \ldots < i_k$, we denote as $\Delta_{i_1 \ldots i_k} = \Delta_{i_1} \ldots \Delta_{i_k}$
\end{definition}

\begin{definition}[Elimination] \label{elimination} Given $\ell \in |L|$, define $e_{\ell}(L)$ in the following way. Suppose
$$L= L_1 |^a \ell |^b L_2  \ ,$$
with $L_1, L_2$ Fox monomials and $a,b \in \{0,1,2\}$. Then 
$$ e_{\ell}(L) = L_1 |^{\max\{a,b\}} L_2 \ .$$
\end{definition}
It is easy to check that conditions 1-7 are still satisfied.
\begin{lemma} \label{commutation} If $\ell \neq \ell'$ in $|L|$ for $L$ Fox monomial, then
$$ e_{\ell} e_{\ell'}(L) = e_{\ell'} e_{\ell}(L)  \ .$$
\end{lemma}
The proof is left to the reader. This allows us to give the following
\begin{definition}[Restriction] Consider $S \subset |L|$ and suppose $|L| \setminus S = \{p_1, \ldots, p_k \}$. Then
$$ \res_S L := e_{p_1} \ldots e_{p_k} (L)  \ .$$
The order doesn't matter because of Lemma \ref{commutation}. We define two variants $\res_S^X, \res_S^Y$, which are equal to $\res_S$ in case $S \neq \emptyset$, but return $\emptyset_X, \emptyset_Y$ respectively in case $S=\emptyset$. These are dummy symbols that obey the following rules. Given $L_1, L_2$ monomials and $p,q \in \{0,1,2\}$:
\begin{itemize}
\item $L_1|^p \emptyset_X |^q L_2 = L_1 |^{\max\{p,q\} } L_2$ if $p=0$ or $q=0$, and $0$ otherwise;
\item $L_1 |^p \emptyset_Y |^q L_2 = L_1 |^{\max \{p,q\} } L_2 $ if $p=1$ or $q=1$, and $0$ otherwise.
\end{itemize}
In case the dummy symbol is at the beginning resp. at the end of the word (that is $L_1 = \emptyset$ resp. $L_2 = \emptyset$), we set by convention $p=3$ resp. $q=3$.
%modo più elegante?
 In other words, the symbol disappears if the number of bars close to it are in the right quantity, and yields zero otherwise. %non molto chiaro
\end{definition}
We need another preliminary definition.
\begin{definition}[Upper-differential] Given $\vect{b} = b_1 \ldots b_p$ a numerical $2$-cloud, define
$$ \partial \vect{b} = \sum_{ \substack{U \sqcup V = |\vect{b}| \\ U,V \neq \emptyset }} \res_U( \vect{b})\  |\ \res_V (\vect{b}) \ ,$$
with the convention that it yields zero if the summation is empty. If $\vect{a} = \vect{a}_1 | \ldots |\vect{a}_n$ is a numerical $1$-cloud, where $a_i$ are numerical $2$-clouds, define
$$ \partial \vect{a} = \sum_{i=1}^n \vect{a}_1 | \ldots | \partial \vect{a}_i | \ldots | \vect{a}_n \ .$$
\end{definition}

We are ready to define the fundamental operation on Fox monomials, that relate them with Fox-Neuwirth trees.

\begin{definition}[X-evaluation]Given $M$ a Fox monomial in $\Fox_3[X_i,Y_j,Z_k]_{IJK}$, an index $\alpha \in I$ and $T$ a numerical $2$-cloud such that $|T| \cap |M| = \emptyset$, suppose 
$$ M = A_0 |^{p_1} X_{\alpha} |^{q_1} A_1 \ldots A_{d-1} |^{p_d} X_{\alpha} |^{q_d} A_d \ ,$$
with $p_k, q_k \in \{0,1,2\}$ and $A_k$ Fox monomials in $\Fox_3[X_i, Y_j, Z_k]_{I\setminus \alpha, JK}$. Define
$$ \ev^{X_i}_{T}(M) = \sum_{U_1 \sqcup \ldots \sqcup U_d = |T|  }  A_0 |^{p_1} \res^X_{U_1}(T) \ |^{q_1} A_1 \ldots A_{d-1} |^{p_d} \res^X_{U_d}(T) |^{q_d} A_d   \ .$$
\end{definition}
Since the definition is quite intricated, let us give an
\begin{example} Consider $M= Y_1 | X_15 | Y_2 || X_1 $, $\alpha = 1$ and $T=12$. We can partition $12$ in four ways: $(\emptyset,12), (1,2), (2,1), (12, \emptyset)$. Each of these partitions gives rise to a term in the sum:
\begin{align*}
(\emptyset,12) \ \ & \Rightarrow \ \ Y_1 | \emptyset_X 5 | Y_2 || 12 = Y_1 | 5 | Y_2 || 12 \\
(1,2) \ \ & \Rightarrow \ \ Y_1 | 15 | Y_2 || 2  \\
(2,1) \ \ & \Rightarrow \ \ Y_1 | 25 | Y_2 || 1 \\
(12,\emptyset) \ \ & \Rightarrow \ \ Y_1 | 125 | Y_2 || \emptyset_X = 0 \ .
\end{align*}
Thus
$$\ev^{X_1}_{12}(M) = Y_1 | 5 | Y_2 || 12+Y_1 | 15 | Y_2 || 2 +Y_1 | 25 | Y_2 || 1 \ .$$
\end{example}
There is an analogous definition for $Y$ variables, that is slightly more intricated because of the derived variables.
\begin{definition}[Y-evaluation] Given $M$ a Fox monomial in $\Fox_3[X_i,Y_j,Z_k]_{IJK}$, an index $\beta \in J$ and $S = s_1 | \ldots | s_h$ a numerical $1$-cloud such that $|S| \cap |M| = \emptyset$, suppose 
$$ M = A_0 |^{p_1} Y_j |^{q_1} A_1 \ldots A_r |^{p_r} \partial Y_j |^{q_r} A_{r+1} \ldots A_d |^{p_{d+1}} Y_j |^{q_{d+1}} A_{d+1}  \ ,$$
with $p_k, q_k \in \{0,1,2\}$ and $A_k$ Fox monomials in $\Fox_3[X_i, Y_j, Z_k]_{I, J\setminus \beta,K}$. Define
$$ \ev^{Y_j}_{S}(M) = $$
$$\sum_{(U_1, \ldots, U_{d+1}) \in \Sat(S) }  A_0 |^{p_1} \res^Y_{U_1}(S) \ |^{q_1} A_1 \ldots A_r |^{p_r} \partial \res^Y_{U_r}(S) |^{q_r} A_{r+1} \ldots A_d |^{p_{d+1}} \res^Y_{U_{d+1}}(S) |^{q_{d+1}} A_{d+1}  \ ,$$
where 
$$\Sat(S) = \{(U_1, \ldots, U_n): n \in \mathbb{N}, \ \ U_1 \sqcup \ldots \sqcup U_n = |S|  $$
$$ \ \ \{U_1, \ldots, U_n\} \textrm{ is coarser than } \{ |s_1|, \ldots, |s_h| \} \textrm{ as partitions of } |S| \} \ .$$
Note that $\partial \res^Y_{U_r}(S)$ is in general a summation of terms $\sum_{\lambda} A_{\lambda}$. We use the convention about bars being bilinear, that is $L_1 |^p (\sum_{\lambda} A_{\lambda}) |^q L_2 = \sum_{\lambda} L_1 |^p A_{\lambda} |^q L_2$. If the summation is zero, the whole term is zero.
\end{definition}
As there are some subtleties with respect to the previous definition, let us give another example:
\begin{example} Consider $M= X_1X_2||Y_2||\partial Y_2|X_12X_2|Y_2 $, $\beta =2$ and $S=89|4$. As in the previous case, each partition of $S$ in three numerical clouds give rise to a term in the summation. However, we must take partitions to be coarser than $(89,4)$, that is each numerical 2-cloud must be contained in a block of our partition.
\begin{align*}
(4,89,\emptyset) \ \ \Rightarrow &  \ \ X_1X_2||Y_2||\partial Y_2|X_12X_2|Y_2 = X_1X_2 || 4|| \partial 89 | X_1 2 X_2 | \emptyset_Y = \\
= & \ \ X_1X_2 || 4|| ( 8|9 + 9|8) | X_1 2 X_2 \\
= & \ \ X_1X_2 || 4||8|9| X_1 2 X_2+X_1X_2 || 567||9|8 | X_1 2 X_2\\
(89,4,\emptyset) \ \ \Rightarrow &  \ \ X_1X_2||89||\partial 4|X_12X_2|Y_2 = 0 \\
(4,\emptyset, 89) \ \ \Rightarrow &  \ \ X_1X_2||4||\emptyset_Y|X_12X_2|89 = X_1X_2||4||X_12X_2|89 \\
(89,\emptyset, 4) \ \ \Rightarrow &  \ \ X_1X_2||89||\emptyset_Y|X_12X_2|4 = X_1X_2||89||X_12X_2|4 \\
(\emptyset,4,89) \ \ \Rightarrow & \ \ X_1X_2||\emptyset_Y||4|X_12X_2|89 = 0 \\
(\emptyset,89,4) \ \ \Rightarrow & \ \ X_1X_2||\emptyset_Y||89|X_12X_2|4 = 0 \\
(89|4,\emptyset, \emptyset) \ \ & \Rightarrow \ \ X_1X_2||89|4||\emptyset_Y|X_12X_2|\emptyset_Y = X_1X_2||89|4||X_1 X_2\ .
\end{align*}
Thus
\begin{align*}
\ev^{Y_2}_{4|89}(X_1X_2||Y_2||\partial Y_2|X_12X_2|Y_2) & = X_1X_2 || 4||8|9| X_1 2 X_2+X_1X_2 || 4||9|8 | X_1 2 X_2 \\
& +X_1X_2||4||X_12X_2|89+X_1X_2||89||X_12X_2|4+X_1X_2||4|89||X_1 X_2 \ .
\end{align*}
\end{example}

It is easy to see that evaluations in different variables commute, since the various partitions are independent. Furthermore, we can define the $Z-$evaluation in a 0-cloud as the literal substitution, since a $Z$ variable can only occur once. Thereby we can give the following
\begin{definition}[Total evaluation] Consider three sets of indices $I = \{i_1, \ldots, i_p\}, J = \{j_1, \ldots, j_q\}, K= \{k_1, \ldots, k_r\}$, a Fox monomial $M \in \Fox_3[X_i,Y_j,Z_k]_{IJK}$, a tuple of numerical 2-clouds $\bar{X}  = ( \bar{X}_{i_1}, \ldots, \bar{X}_{i_p} )$, a tuple of numerical 1-clouds $\bar{Y} = (\bar{Y}_{j_1}, \ldots, \bar{Y}_{j_q})$ and a tuple of numerical 0-clouds $\bar{Z} = (\bar{Z}_{k_1}, \ldots, \bar{Z}_{k_r} )$. Denote by
$$ | \bar{X} \bar{Y} \bar{Z} | :=  \left ( \bigcup_{i \in I} |\bar{X}_i| \right ) \cup \left ( \bigcup_{j \in J} |\bar{Y}_j| \right ) \cup  \left ( \bigcup_{k \in K} |\bar{Z}_k| \right )  \ .$$ 
Suppose that the variables have pairwise disjoint support, that is $|L| \cap |L'| = \emptyset$ for all $L \neq L'$ in $\{\bar{X}_{i_1}, \ldots, \bar{X}_{i_p}, \bar{Y}_{j_1}, \ldots, \bar{Y}_{j_q}, \bar{Z}_{k_1}, \ldots, \bar{Z}_{k_r} \} $, and also
$$ |\bar{X} \bar{Y} \bar{Z} | \cap |M| = \emptyset \ .$$
Define the total evaluation of $M$ in $\bar{X}, \bar{Y}, \bar{Z}$ as the composition of successive evaluations
$$\ev_{\bar{X} \bar{Y} \bar{Z} }(M) := \ev^{Z_{k_r}}_{\bar{Z}_{k_r}} \ldots \ev^{Z_{k_1}}_{\bar{Z}_{k_1}} \ev_{\bar{Y}_{j_q} }^{Y_{j_q}} \ldots \ev_{\bar{Y}_{j_1} }^{Y_{j_1}} \ev_{\bar{X}_{j_p} }^{X_{j_p}} \ldots \ev_{\bar{X}_{j_1} }^{X_{j_1}}(M)  \ .$$
We extend this map linearly to polynomials, with the convention that $\ev_{\bar{X} \bar{Y} \bar{Z}}(M)  = 0$ whenever $|M| \cap | \bar{X} \bar{Y} \bar{Z} | \neq \emptyset$.
\end{definition}
\vspace{3pt}
Note that if the clouds tuples $\bar{X}, \bar{Y}, \bar{Z}$ and the monomial $M$ we use are such that
$$ | \bar{X} \bar{Y} \bar{Z} | \cup |M| = \{1, \ldots, n \} \ ,$$
then $\ev_{\bar{X} \bar{Y} \bar{Z} } (M)$ lives in $\MFN_3(n)$, since we can interpret a numerical monomial as a Fox-Neuwirth tree. 
\begin{notation} Sometimes we will use the standard notation for evaluation of polynomials and write $M(\bar{X}, \bar{Y}, \bar{Z})$ in place of $\ev_{\bar{X} \bar{Y} \bar{Z}}$.
\end{notation}

\subsection{Differential formulas}
We are ready to give formulas for higher differentials. We will analyze the expressions $D_0, D_1, D_2, D_3$ separately, giving examples and insights on the way to lighten the technicality of the section.

First of all we write the differential $D_0$ of $\MFN_3(n)_{\bullet}$ using our formalism.
\begin{proposition}
$D_0$ is the only linear operator satisfying 
$$D_0(Z || W)=D_0(Z)||W+ Z||D_0(W)$$
and such that for each $1$-cloud $Y$ we have that 
$$D_0(Y)= Y || Y + \partial(Y)$$
\end{proposition}

 We proceed with $D_1$.
\begin{definition}[First Differential] \label{d1-fn} There is a linear map 
$$D_1 : \MFN_3(n)_{\bullet} \to \MFN_3(n+1)_{\bullet}$$
given as a sum
$$ D_1 = D_1^0 + \left ( \sum_{i=1}^n D_1^{i} \right ) +D_1^{n+1}$$
such that for all $\Gamma$
$$ D_1^0(\Gamma) = 1 || \Delta_0 \Gamma \ ,$$
$$ D_1^{n+1}(\Gamma) = \Delta_{n+1}\Gamma || (n+1) = \Gamma || (n+1) \ ,$$
and for all $\bar{Z}_1, \bar{Z}_2$ numerical 0-clouds, $\bar{Y}_1, \bar{Y}_2$ numerical 1-clouds, $\bar{X}_1, \bar{X}_2$ numerical 2-clouds:
$$ D_1^i(\bar{Z}_1||\bar{Y}_1|\bar{X}_1i\bar{X}_2|\bar{Y}_2||\bar{Z}_2) =\hat{D}_1^i(\Delta_i \bar{X}, \Delta_i \bar{Y}, \Delta_i \bar{Z}) \ ,$$
where $\hat{D}_1^i$ is the polynomial
$$ \hat{D}_1^i := Z_1||Y_1 | X_1 i X_2 | Y_2 || Y_1 | X_1 (i+1) X_2 | Y_2 || Z_2  \ .$$
\end{definition}
The definition can seem quite complicated, but it is easy to compute in practice. Let us give an example.
\begin{example} Given $\Gamma = 1|23||7|548||9$, we want to compute $D_1^4$. We can uniquely locate which $\bar{Z}_1, \bar{Z}_2$,$\bar{Y}_1, \bar{Y}_2$, $\bar{X}_1, \bar{X}_2$ to use:
\begin{itemize}
    \item $\bar{Z}$'s are the parts separated by a double bar from $i=4$, that is $L_1 = 1|23, L_2 = 9.$
\item $\bar{Y}$'s are separated by a bar, but not a double bar from $4$, hence $Y_1 = 7$ and $Y_2 = \emptyset$.
\item $\bar{X}$'s are not separated from $4$: $X_1 = 5$ and $X_2 = 8$.
\end{itemize}
Then we have
\begin{align*}
D_1^4(\Gamma) & = \Delta_4(1|23) || \ev_{(\Delta_45,\Delta_48),(\Delta_47, \emptyset)}(Y_1|X_14X_2|Y_2||Y_1|X_15X_2|Y_2) || \Delta_4(9) =\\
& = 1|23 || \ev_{(6,9),(8,\emptyset)}(Y_1|X_14X_2|Y_2||Y_1|X_15X_2|Y_2) || 10 \ .
\end{align*}
Each of the variables appear twice. We have one possibility of distribution for $Y_2=\emptyset$ and two for the other three variables: 8 in total. Thus, using bilinearity of double bars:
\begin{align*}
& = 1|23 || 8|649|\emptyset_Y || \emptyset_Y| \emptyset_X 5 \emptyset_X | \emptyset_Y || 10 
+ 1|23 || \emptyset_Y|649|\emptyset_Y || 8| \emptyset_X 5 \emptyset_X | \emptyset_Y || 10 +\\
& +1|23 || 8|\emptyset_X 49|\emptyset_Y || \emptyset_Y| 6 5 \emptyset_X | \emptyset_Y || 10 
+ 1|23 || \emptyset_Y|\emptyset_X 49|\emptyset_Y || 8| 6 5 \emptyset_X | \emptyset_Y || 10 +\\
& + 1|23 || 8|64\emptyset_X |\emptyset_Y || \emptyset_Y| \emptyset_X 59 | \emptyset_Y || 10 
+ 1|23 || \emptyset_Y|64\emptyset_X |\emptyset_Y || 8| \emptyset_X 5 9| \emptyset_Y || 10 +\\
& + 1|23 || 8|\emptyset_X 4\emptyset_X |\emptyset_Y || \emptyset_Y| 6 59 | \emptyset_Y || 10 
+ 1|23 || \emptyset_Y|\emptyset_X 4\emptyset_X |\emptyset_Y || 8| 6 5 9 | \emptyset_Y || 10 =\\
\end{align*}
We can simplify the dummy symbols to obtain a neater expression:
\begin{align*}
& = 1|23 || 8|649|| 5 || 10 + 1|23 || 649|| 8|5 || 10 +1|23 || 8|49|| 6 5|| 10 + 1|23 || 49|| 8| 6 5 || 10 +\\
& + 1|23 || 8|64||59 || 10 + 1|23 ||64|| 8|  5 9 || 10 +1|23 || 8| 4||6 59 || 10 + 1|23 ||4 || 8| 6 5 9 || 10 \ .
\end{align*}
\end{example}
Let us analyze the formula of the differential before going on. We can summarise the expression for the central contributions ($i=1, \ldots, n$) as a three-step process:
\begin{itemize}
    \item Double $i$ into $i||i'$, where $i' = (i+1)$ is the new point;
    \item Relabel the other points with $d_i$;
    \item Distribute them between $i$ and $i'$ in a way that preserves the depth relation with the doubled point.
\end{itemize}
 We can say that the "core expression" of $D_1^i$ is $i||i'$; the rest of the process, on the other hand, is algorithmic. We will see how this behavior generalizes to higher differentials, but with different core expressions. Regarding the extremal contributions $D_1^0, D_1^{n+1}$, they don't have a higher analogue. Indeed, since they commute with other contributions, there will be no associated $D_2$ "to force commutativity". 

\begin{definition}[Second Differential] \label{d2-fn} There is a linear map 
$$D_2 : \MFN_3(n)_{\bullet} \to \MFN_3(n+2)_{\bullet+1}$$
given as a sum
$$ D_2 =\sum_{1 \le i < j \le n }^n D_2^{ij} $$
such that for all $\bar{Z}_1, \bar{Z}_2$ numerical 0-clouds, $\bar{Y}_1, \bar{Y}_2, \bar{Y}_3$ numerical 1-clouds, $\bar{X}_1, \bar{X}_2, \bar{X}_3, \bar{X}_4$ numerical 2-clouds, we have
$$ D_2^{ij}(\bar{Z}_1||\bar{Y}_1|\bar{X}_1i\bar{X}_2| \bar{Y}_2 | \bar{X}_3 j \bar{X}_4|\bar{Y}_3||\bar{Z}_2) = \hat{D}_2^{i|j}(\Delta_{ij} \bar{X}, \Delta_{ij} \bar{Y}, \Delta_{ij} \bar{Z} ) \ ,$$
$$ D_2^{ij}(\bar{Z}_1||\bar{Y}_1|\bar{X}_1i\bar{X}_2j \bar{X}_3|\bar{Y}_2||\bar{Z}_2) = \hat{D}_2^{ij}(\Delta_{ij} \bar{X}, \Delta_{ij} \bar{Y}, \Delta_{ij} \bar{Z} ) \ ,$$
where $\hat{D}_2^{ij}, \hat{D}_2^{i|j}$ are the polynomials
\begin{align*}
\hat{D}_2^{i|j} & := Z_1 || Y_1 | X_1 i X_2| Y_2 | X_3(j+1) X_4| Y_3 || Y_1 | X_1 (i+1) X_2 | Y_2 | X_3 (j+2) X_4| Y_3 || Z_2\\
\hat{D}_2^{ij} & := Z_1 ||Y_1 | X_1 i X_2(j+1) X_3| Y_2 || Y_1 | X_1 (i+1) X_2 X_3 | X_1X_2(j+2) X_3| Y_2 || Z_2 +\\
& +Z_1 || Y_1 | X_1  X_2(j+1) X_3|X_1 i X_2 X_3 | Y_2 || Y_1 | X_1 (i+1) X_2(j+2) X_3| Y_2 || Z_2 \ .
\end{align*}
Furthermore, $D_2^{ij}(\Gamma) = 0$ whenever $i <_0 j$ in $\Gamma$. Finally, let $\sigma = (i j), \tau = (i \ j+1)(i+1 \ j+2)$  be permutations in cycle notation. Then
$$ D_2^{ij}( \Fox_{\sigma} \Gamma) = \Fox_{\tau} D_2^{ij}(\Gamma) \ .$$
\end{definition}

\begin{remark} \label{d2-explanation}
As you can see, the overall formula is quite similar, but with a few differences. Since it depends on two indices $i,j$, we have to distinguish depending on the depth index $d(i,j) \in \{0,1,2\}$ and $i < j$ or $j < i$ with respect to the order in $\Gamma$. The symmetry condition allows one to reduce to $i < j$, and the case $d(i,j) = 0$ yields zero. The other two remaining cases have an explicit formula.

The case $i <_1 j$ has core expression $i|j||i'|j'$, which is quite similar to the first differential case: we double the two points "in parallel", relabeling conveniently. The case $i <_2 j$ instead, is quite new: it has two core expressions
$$ ij || i'|j', \ \ \ j|i || i'j' \ ,$$
with the second one presenting an unexpected inversion in the first half. We have an intuitive explanation for this. When solving the multicomplex equation
$$ D_0 D_2(ij) = D_1^2(ij) + D_2 D_0(ij) $$ 
we can choose an alternative expression for $D_2(ij)$ by adding $D_0$ of something. For example, adding 
$$ D_0(ij||i'j') = i|j||i'j'+ j|i||i'j' + ij||i'|j'+ij||j'|i'$$
we would get an alternative expression
$$  i|j || i'j' + ij || j'|i' \ ,$$
with the inversion on the second half. If we were working over $\mathbb{Q}$, the best candidate would be the average of the two candidates for $D_2$, giving a symmetric expression. Since we can't divide by $2$ in $\mathbb{F}_2$, we have to make an arbitrary choice, resulting in a non-symmetric expression. This is even more evident in the case of $D_3$:
\end{remark}
\begin{definition}[Third Differential] \label{d3-fn} There is a linear map 
$$D_3 : \MFN_3(n)_{\bullet} \to \MFN_3(n+3)_{\bullet+2}$$
given as a sum
$$ D_3 =\sum_{1 \le i < j < k \le n }^n D_3^{ijk} $$
such that for all $\bar{Z}_1, \bar{Z}_2$ numerical 0-clouds, $\bar{Y}_1, \bar{Y}_2, \bar{Y}_3, \bar{Y}_4 $ numerical 1-clouds, $\bar{X}_1, \bar{X}_2, \bar{X}_3, \bar{X}_4, \bar{X}_5, \bar{X}_6$ numerical 2-clouds, we have
$$ D_3^{ijk}(\bar{Z}_1||\bar{Y}_1|\bar{X}_1i\bar{X}_2| \bar{Y}_2 | \bar{X}_3 j \bar{X}_4|\bar{Y}_3 | \bar{X}_5 k \bar{X}_6 | \bar{Y}_4||\bar{Z}_2) = \hat{D}_3^{i|j|k}(\Delta_{ijk} \bar{X}, \Delta_{ijk} \bar{Y}, \Delta_{ijk} \bar{Z} ) \ ,$$
$$ D_3^{ijk}(\bar{Z}_1||\bar{Y}_1|\bar{X}_1i\bar{X}_2j\bar{X}_3| \bar{Y}_2 | \bar{X}_4 k \bar{X}_5|\bar{Y}_3 ||\bar{Z}_2) =\hat{D}_3^{ij|k}(\Delta_{ijk} \bar{X}, \Delta_{ijk} \bar{Y}, \Delta_{ijk} \bar{Z} ) \ ,$$
$$ D_3^{ijk}(\bar{Z}_1||\bar{Y}_1|\bar{X}_1i\bar{X}_2| \bar{Y}_2 | \bar{X}_3 j \bar{X}_4k \bar{X}_5|\bar{Y}_3 ||\bar{Z}_2) = \hat{D}_3^{i|jk}(\Delta_{ijk} \bar{X}, \Delta_{ijk} \bar{Y}, \Delta_{ijk} \bar{Z} ) \ ,$$
$$ D_3^{ijk}(\bar{Z}_1||\bar{Y}_1|\bar{X}_1i\bar{X}_2j \bar{X}_3 k \bar{X}_4| \bar{Y}_2 ||\bar{Z}_2) = \hat{D}_3^{ijk}(\Delta_{ijk} \bar{X}, \Delta_{ijk} \bar{Y}, \Delta_{ijk} \bar{Z} ) \ .$$

where $\hat{D}_3^{i|j|k}, \hat{D}_3^{ij|k}, \hat{D}_3^{i|jk}, \hat{D}_3^{ijk}$ are the polynomials
\begin{adjustwidth}{-30pt}{-30pt}
\begin{align*}
& \hspace{8cm}\hat{D}_3^{i|j|k} := \\
&Z_1 || Y_1 | X_1 i X_2| Y_2 | X_3(j+1) X_4| Y_3 | X_5 (k+2) X_6 | Y_4 ||  Y_1 | X_1 (i+1) X_2| Y_2 | X_3(j+2) X_4| Y_3 | X_5 (k+3) X_6 | Y_4 || Z_2 \\ \\
& \hspace{8cm}\hat{D}_3^{ij|k} :=\\
&Z_1 || Y_1 | X_1 i X_2(j+1) X_3| Y_2|X_4 (k+2) X_5 | Y_3 || Y_1 | X_1 (i+1) X_2 X_3 | X_1X_2(j+2) X_3| Y_2| X_4 (k+3) X_5| Y_3 || Z_2 +\\
& + Z_1||Y_1 | X_1  X_2(j+1) X_3|X_1 i X_2 X_3 | Y_2 | X_4 (k+2) X_5| Y_3 || Y_1 | X_1 (i+1) X_2(j+2) X_3| Y_2 | X_4 (k+3) X_5 | Y_3 || Z_2 \\ \\
&\hspace{8cm}\hat{D}_3^{i|jk}  := \\
&Z_1 || Y_1 | X_1 i X_2 | Y_2 | X_3(j+1)X_4 (k+2) X_5 | Y_3 || Y_1 | X_1 (i+1) X_2 | Y_2 | X_3(j+2)X_4 X_5|X_3 X_4 (k+3) X_5| Y_3 || Z_2 +\\
& + Z_1||Y_1 |X_1 i X_2 | Y_2 | X_3 X_4 (k+2) X_5| X_3(j+1) X_4 X_5 | Y_3 || Y_1 | X_1 (i+1) X_2| Y_2 |X_3(j+2) X_4 (k+3) X_5| Y_3 || Z_2 \\ \\
&\hspace{8cm}\hat{D}_3^{ijk} := \\
&Z_1 ||Y_1|X_1X_2(j+1)X_3(k+2)X_4|X_1iX_2X_3X_4| Y_2 || Y_1|X_1X_2(j+2)X_3X_4 |X_1(i+1)X_2X_3(k+3)X_4|Y_2 || Z_2 + \\
&+Z_1||Y_1|X_1X_2(j+1)X_3(k+2)X_4|X_1iX_2X_3X_4 |Y_2||Y_1|X_1(i+1)X_2(j+2)X_3X_4|X_1X_2X_3(k+3)X_4|Y_2 || Z_2 + \\
&+ Z_1||Y_1|X_1X_2(j+1)X_3X_4 | X_1iX_2X_3(k+2)X_4|Y_2||Y_1|X_1(i+1)X_2(j+2)X_3X_4|X_1X_2X_3(k+3)X_4|Y_2 || Z_2 +\\
&+ Z_1||Y_1 | X_1iX_2X_3(k+2)X_4| X_1X_2(j+1)X_3X_4 | Y_2 || Y_1 | X_1(i+1)X_2X_3X_4 |X_1X_2(j+2)X_3(k+3)X_4| Y_2 || Z_2 +\\
&+ Z_1||Y_1 | X_1X_2X_3(k+2)X_4 |X_1iX_2(j+1)X_3X_4| Y_2 || Y_1 | X_1(i+1)X_2X_3X_4 |X_1X_2(j+2)X_3(k+3)X_4|Y_2 || Z_2 + \\
&+ Z_1||Y_1 |X_1X_2X_3(k+2)X_4 |X_1iX_2(j+1)X_3X_4| Y_2 || Y_1| X_1(i+1)X_2X_3(k+3)X_4|X_1X_2(j+2)X_3X_4|Y_2 || Z_2 + \\
&+ Z_1||Y_1|X_1iX_2(j+1)X_3(k+2)X_4|Y_2 || Y_1 | X_1(i+1)X_2X_3X_4|X_1X_2(j+2)X_3X_4|X_1X_2X_3(k+3)X_4|Y_2 || Z_2 + \\
&+ Z_1||Y_1|X_1 X_2 X_3 (k+2) X_4|X_1X_2(j+1)X_3X_4|X_1iX_2X_3X_4| Y_2 || Y_1  |X_1(i+1)X_2(j+2)X_3(k+3)X_4| Y_2 || Z_2 + \\
&+ Z_1||Y_1|X_1iX_2(j+1)X_3X_4|Y_2 || Y_1 |  X_1(i+1)X_2X_3(k+2)X_4  | Y_2 || Y_1 |X_1X_2(j+2)X_3(k+3)X_4| Y_2 || Z_2 + \\
&+ Z_1||Y_1|X_1X_2(j+1)X_3(k+2)X_4|Y_2||Y_1|  X_1iX_2X_3(k+3)X_4  |Y_2|| Y_1 |X_1(i+1)X_2(j+2)X_3X_4|Y_2 || Z_2 \ .
\end{align*}
\end{adjustwidth}
Furthermore, $D_3^{ijk}(\Gamma) = 0$ whenever either $i <_0 j$ or $j <_0 k$ in $\Gamma$. Finally, let 
$$\Delta: \textrm{Aut}(\{i,j,k\}) \cong \Sigma_3 \overset{\textrm{diag}}{\to} \Sigma_3 \times \Sigma_3 \cong \textrm{Aut}(\{i,j+1,k+2\}) \times \textrm{Aut}(\{i+1,j+2,k+3\}) $$
$$ \to \textrm{Aut}(\{i,i+1,j+1,j+2,k+2,k+3\}) \ .$$
Then for all $\sigma \in \textrm{Aut}(\{i,j,k\})$:
$$ D_3^{ijk}( \Fox_{\sigma} \Gamma) = \Fox_{\Delta \sigma} D_3^{ijk}(\Gamma) \ .$$
\end{definition}

The core expressions for $i|j|k, ij|k, i|jk$ are mimicked from the lower-degree differentials. The differential for $ijk$ is instead quite peculiar, with 10 asymmetrical core expressions that we can summarise with the following picture (where $i,j,k$ correspond to $1,2,3$)%manca un uno nel terzo grafo

\begin{figure}[H]
	\centering
	\includegraphics[scale=0.24]{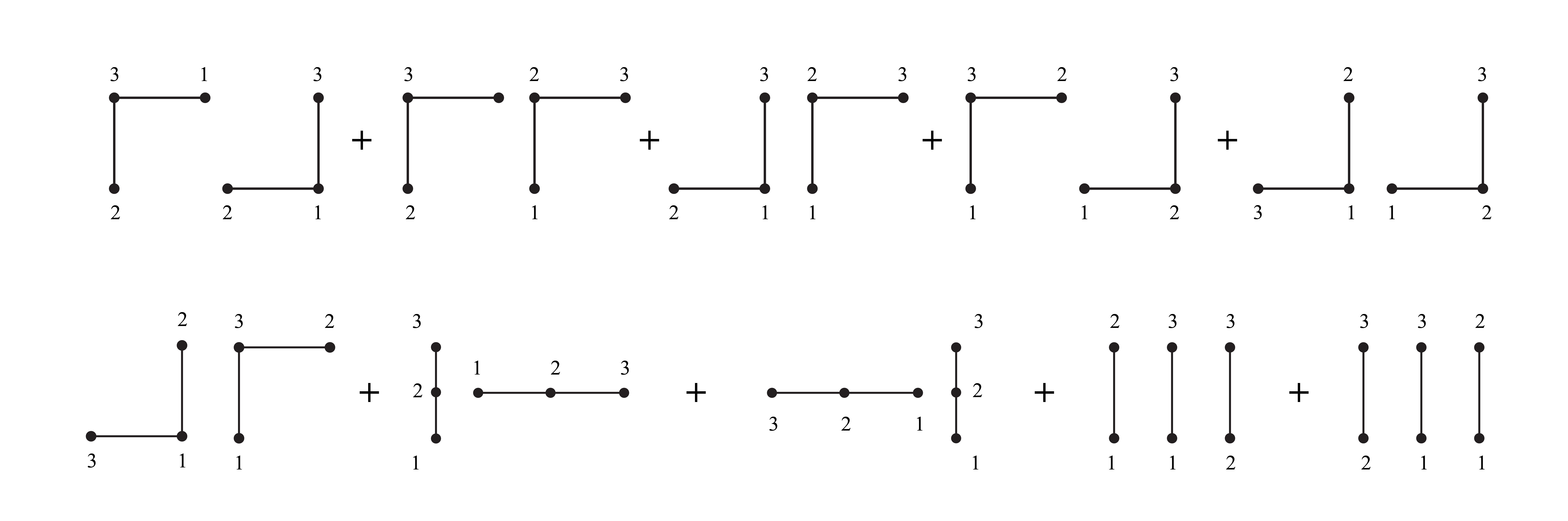}
	\caption{The core positions appearing in the differential $D_3$ for $m=3$}
	\label{core-positions}
\end{figure}

\begin{remark} \label{diff-mult-rmk} Note that the differentials of Definitions \ref{d1-fn} \ref{d2-fn} \ref{d3-fn} can be used to define
$$ D_A : \MFN_3(n)_{\bullet} \to \MFN_3(n+|A|)_{\bullet+|A|-1} $$
for any multiset $A \subset [n+1]$ such that $|A| \le 3$. For $|A| \le 1$, one can use \ref{d1-fn} together with the standard Fox-Neuwirth differential. For $|A| \ge 2$, if the indices are distinct and $A \subset \{1, \ldots, n\}$ the definitions \ref{d2-fn} \ref{d3-fn} apply directly; in all other cases, set $D_A=0$. 
\end{remark}

We now want to apply Theorem \ref{bdthm} to the differentials defined in this section. This requires the verification of the two hypothesis of Thm. \ref{bdthm}: the "sliced" multicomplex equation, and the "bound" for the differentials in terms of the semicosimplicial structure on Fox-Neuwirth Trees. In the following sections, we will prove the following two lemmas:
\begin{lemma} \label{slice-mcpx-lemma}
The maps $D_A : \MFN_3(n)_{\bullet} \to \MFN_3(n+|A|)_{\bullet+|A|-1} $  satisfy the sliced multicomplex equation, as stated in Thm. \ref{bdthm}.
\end{lemma}

\begin{lemma} \label{differential-bounds} For all sets $A\subset [n+1]$ such that with $|A| \le 3$ and $\Gamma \in \FNP_m(n)$, we have $D_A(\Gamma) \in \mu_I(\Gamma)$.
\end{lemma}

Applying Theorem \ref{bdthm} to $X_n = \NBZ_m(n)$ and $\mathcal{P}_n = \FNP_m(n)$, equipped with the semicosimplicial structure given in \ref{fn-cosimplicial}, we get one of the main result of the paper:

\begin{theorem} \label{sinha-mcpx} The first three pages of Sinha Spectral Sequence in homology for ambient dimension $m=3$ and $\mathbb{F}_2$ coefficients are isomorphic to the $3$-truncated spectral sequence associated to $\MFN_3$.
\end{theorem}

\subsection{The strategy for the multicomplex equation} \label{strategy}
The multicomplex equations up to degree $3$ are four:
$$ D_0^2 = 0 \ ,$$
$$D_1 D_0 = D_0 D_1 \ ,$$
$$ D_2 D_0 + D_1 D_1 + D_0 D_2 = 0 \ ,$$
$$ D_3 D_0 + D_2 D_1 + D_1 D_2 + D_0 D_3 = 0 \ .$$
The first equation is already verified, since $\MFN_3(n)$ is the cellular chain complex of $\NBZ_3(n)$. Regarding the second one, we will "slice" the equation into $(n+2)$ sub-equations
$$ D_1^{i} D_0 = D_0 D_1^i \ ,$$
since $D_1 = \sum_{i=0}^{n+1}D_1^i$. Notice that the sub-equations correspond to the different relabelings $\Delta_0, \ldots, \Delta_{n+1}$ we make on the clouds. This gives us an idea on how to slice higher dimensional equations. There are four ways to obtain a $\Delta_{ij}$ relabeling with two differentials ($i < j$):
$$ (\textrm{nothing}, \Delta_{ij} ), \ \ (\Delta_i, \Delta_j), \ \ (\Delta_{j+1}, \Delta_i), \ \ (\Delta_{ij}, \textrm{nothing}) \ .$$
We introduced a precise way to perform successive relabelings in definition \ref{union-with-shift}, but they can be done by hand at this level. The corresponding terms are
$$ D_0 D_2^{ij} + D_1^i D_1^j + D_1^{j+1} D_1^i + D_2^{ij} D_0 = 0$$
Since we have two indices, we also have to distinguish whether $i <_2 j$ or $i <_1 j$; there is no need to analyze the case $j <_{\Gamma} i$ because of the equivariance property. We analyze separately the extremal case $i=0$ or $j=n+1$, because they have a different $D_1$ definition. However, these are simpler, because $D_2$ vanishes when an index is extremal. The corresponding equations are 
%limiti sugli indici
$$ D_1^0 D_1^j + D_1^{j+1}D_1^0 = 0.  $$
$$ D_1^0 D_1^{n+1} + D_1^{n+1}D_1^0 = 0. $$
$$ D_1^i D_1^{n+1} + D_1^{n+1} D_1^i = 0$$
 There is another degenerate case that must be considered: the case in which $i=j$. There is no obstruction to doubling the same point twice. There are only two ways of obtaining $\Delta_{ii}$, since $D_2$ cannot have repeated indices:
$$ (\Delta_{i+1}, \Delta_i), \ \ \ (\Delta_i, \Delta_i)$$
Corresponding to the equation
$$ D_1^{i+1} D_1^i + D_1^i D_1^i = 0$$
The next equation is long to verify. It has been carried out with the assistance of a computer\footnote{That is, the program asked for core expressions which we manually typed in,
 and it would complete with $X$'s and $Y$'s, as well as doing some basic automatic manipulations on strings that substantially sped up the verification. 
 The result was printed and then checked to be zero modulo $2$ with a basic sort-and-count algorithm.} but without an automatic program, since the differentials were uncomfortable
  to implement symbolically. There are $8$ possible ways of obtaining the relabeling $\Delta_{ijk}$ in the general case $i<j<k$:
$$ (\textrm{nothing}, \Delta_{ijk}), \ \ \ (\Delta_i, \Delta_{jk}), \ \ \ (\Delta_{j+1}, \Delta_{ik}), \ \ \ (\Delta_{k+2}, \Delta_{ij} ) $$
$$ (\Delta_{ij}, \Delta_k), \ \ \ (\Delta_{i,k+1}, \Delta_j), \ \ \ (\Delta_{j+1,k+1}, \Delta_i), \ \ \ (\Delta_{ijk}, \textrm{nothing}) $$
Corresponding to the equation
$$ D_0 D_3^{ijk} + D_1^i D_2^{jk} + D_1^{j+1} D_2^{ik}+ D_1^{k+2} D_2^{ij} + D_2^{ij} D_1^k + D_2^{i,k+1} D_1^j + D_2^{j+1,k+1} D_1^i + D_3^{ijk} D_0 = 0$$
Since we have three terms, we have to distinguish:
\begin{itemize}
    \item $i <_2 j <_2 k$ (aligned);
    \item $i <_2 j <_1 k$ (left mixed);
    \item $ i <_1 j <_2 k$ (right mixed);
    \item $i <_1 j <_1 k$ (planar)
\end{itemize}
We analyze the extremal cases separately, obtaining (for both $j <_1 k$ and $j<_2k$):
$$ D_1^0 D_2^{jk} +D_2^{j+1,k+1} D_1^0 = 0$$
$$ D_1^{n+3} D_2^{ij} +D_2^{ij} D_1^{n+1}= 0$$
The degenerate cases in which two indexes are equal correspond to ($i=j$, with $i<_1 k$ or $i <_2 k$):
$$ D_1^i D_2^{ik} + D_1^{i+1} D_2^{ik} + D_2^{i,k+1} D_1^i + D_2^{i+1,k+1} D_1^i = 0$$ 
and ($j=k$, with $i<_1 j$ or $i <_2 j$):
$$ D_1^{j+1} D_2^{ij}+ D_1^{j+2} D_2^{ij} + D_2^{ij} D_1^j + D_2^{i,j+1} D_1^j  = 0$$
The computations of the $3+6+12 = 21$ different cases can be found on github \cite{multicpx-eqn}. We give here an example of calculation, in order to show how one applies two differentials in succession.

\begin{example} Let us compute the term $D_1^i D_1^j(\Gamma)$ in case $i<_1 j$. Firstly, we have to put $\Gamma$ in "monomial form", knowing that the doubled points will be $i,j$. There exist numerical clouds $X_1, X_2, X_3, X_4, Y_1, Y_2, Y_3, Z_1, Z_2$ such that 
$$\Gamma = Z_1 || Y_1 | X_1 i X_2 | Y_2 | X_3 j X_4 | Y_3 || Z_2$$
Since $Z_1, Z_2$ always stay in the same position, we omit them from the calculation and suppose 
$$ \Gamma' = Y_1 | X_1 i X_2 | Y_2 | X_3 j X_4 | Y_3 $$
Define $Y_4 = Y_1 | X_1 i X_2 | Y_2$.  Thus 
$$ D_1^j (\Gamma') = \hat{D}_1^j(  Y_4 | X_3 j X_4 | Y_3 ) = Y_4 | X_3 j X_4 | Y_3 || Y_4 | X_3 j X_4 | Y_3 $$
We have to distribute the 2-clouds of $Y_4$ in two. The clouds contained in $Y_1, Y_2$ will automatically distribute if we repeat the variable; however, $X_1i X_2$ has to be distributed manually. We then obtain two terms:
$$ Y_1 | X_1i X_2 | Y_2 | X_3 j X_4 | Y_3 || Y_1| Y_2| X_3 j X_4 | Y_3+ Y_1|Y_2 | X_3 j X_4 | Y_3 || Y_1|X_1i X_2| Y_2 | X_3 j X_4 | Y_3 $$
Before applying $D_i$, we have to reformulate the two summands in terms of $i$. Define $Y_5 = Y_2 | X_3 j X_4 | Y_3 $ and $Z_3 = Y_1| Y_2| X_3 j X_4 | Y_3 $. We have
$$ D_1^i D_1^j (\Gamma') = D_i(Y_1 | X_1i X_2 | Y_5 || Z_3+ Z_3 || Y_1|X_1i X_2| Y_5) = $$
$$ = Y_1 | X_1i X_2 | Y_5|| Y_1 | X_1i X_2 | Y_5 || Z_3+ Z_3 || Y_1|X_1i X_2| Y_5 || Y_1 | X_1i X_2 | Y_5$$
The $Z$ variable is not repeated, so we can simply substitute it. Regarding $Y_5$, we apply the same above argument, getting the four terms:
$$ Y_1 | X_1i X_2 | Y_2 | X_3 j X_4 | Y_3|| Y_1 | X_1i X_2 | Y_2 |  Y_3 || Z_3+ Z_3 || Y_1|X_1i X_2| Y_2 | X_3 j X_4 | Y_3 || Y_1 | X_1i X_2 | Y_2 | Y_3+$$
$$+Y_1 | X_1i X_2 | Y_2 |  Y_3|| Y_1 | X_1i X_2 | Y_2 | X_3 j X_4 | Y_3 || Z_3+ Z_3 || Y_1|X_1i X_2| Y_2 | Y_3 || Y_1 | X_1i X_2 | Y_2 | X_3 j X_4 | Y_3$$
\end{example}

These equations are enough to prove Lemma \ref{slice-mcpx-lemma}.

\subsection{Bounds on Differential Formulas}

In order to prove the "bound" on the multicomplex differentials, needed to apply theorem \ref{bdthm} we introduce a notion of "polynomial inequality". We start with the "sum inequality":
\begin{definition} Let $\{\Gamma_i\}_{i \in I}, \{\Gamma'_j\}_{j \in J} \subset \FNP_3(n)$  two finite collections of Fox-Neuwirth trees. We say that 
$$\sum_{i \in I} \Gamma_i \le \sum_{j \in J} \Gamma'_j $$
if there exists a surjective function $\alpha: I \to J$ such that $\Gamma_i \le \Gamma'_{\alpha(i)}$.
\end{definition}
It is not hard to see that it is an order relation. The only non trivial property is antisymmetry, which follows from the cycle decomposition of a permutation on a finite set. 
\begin{definition} Let $M,N \in \Fox[X_i, Y_j, Z_k]_{IJK}$ be two Fox Neuwirth monomials. We say that $M \le N$ if for all total evaluations in $\bar{X}, \bar{Y}, \bar{Z}$ we have
$$ \ev_{\bar{X} \bar{Y} \bar{Z} } (M) \le \ev_{\bar{X} \bar{Y} \bar{Z} } (N) $$
with respect to the above sum inequality notion. Given two polynomials $P = \sum_{i \in I} M_i, Q = \sum_{j \in J} N_j$, we say that $P \le Q$ if there exists a surjective function $\alpha: I \to J$ such that $M_i \le N_{\alpha(i)}$ for all $i \in I$.
\end{definition}

\
Here is a key observation about polynomial inequalities:
\begin{remark} \label{mu-inequality} Let $M, P \in \Fox [X_i, Y_j, Z_k]_{IJK}$ such that $M$ is a monomial without repeated variables. Fix numerical clouds $\bar{X}_i, \bar{Y}_j, \bar{Z}_k$ for all $i \in I, j \in J, k \in K$. Note that $\ev_{\bar{X}\bar{Y}\bar{Z}}(M) $ is just one Fox-Neuwirth tree $\Gamma$, since there are no repeated variables. Then $P \le M$ if and only if $\ev_{\bar{X}\bar{Y}\bar{Z}}(P) \in \mu(\Gamma)$, referring to notation \ref{mu-notation}. 
\end{remark}
Since our differentials are defined as evaluations of polynomials, we will try to use successive inequalities from the differential to the monomial representing the tree $d_I \Gamma$. 
%in order to obtain differential bounds% similar to \ref{differential-bounds-2}. %tolto perchè manca la referenza
The three tools for polynomial inequalities that we will use are the following:

\begin{lemma}[High Split lemma] Let $M \in \Fox[X_i]_{I}$ be a monomial without bars, $Y$'s and $Z$'s. Consider a partition of its support $|M| = C_1 \sqcup C_2$. Then 
$$ \res_{C_1} M | \res_{C_2}M \le M $$
\end{lemma}
\begin{proof} Write $M = \tilde{M}_0 A^1 \ldots A^d \tilde{M}_d$ with $\tilde{M}_i$ being a (possibly empty) monomial with only $X$'s and $|A| = A^1 \sqcup \ldots \sqcup A^d$. Consider $A^i_{\alpha} = A^i \cap C_{\alpha} $. The thesis is equivalent to 
$$ \tilde{M}_0 A_1^1 \ldots A_1^d \tilde{M}_d | \tilde{M}_0 A_2^1 \ldots A_2^d \tilde{M}_d \le \tilde{M}_0 A^1 \ldots A^d \tilde{M}_d$$ 
Consider a tuple of numerical clouds $\bar{X} = (\bar{X}_i)_{i\in I}$ with disjoint support and $|M| \cap |\bar{X}_i|=\emptyset$. If $X_i$ appears $n$ times in $M$, it will appear $2n$ times in $\res_{C_1} M | \res_{C_2}M$. When evaluating, we want to show that the term on the left distributing $\bar{X}_i$ via the partitions 
$$\bar{X}_i=S^1_0(i) \sqcup \ldots \sqcup S^1_n(i) \sqcup S^2_0(i) \sqcup \ldots \sqcup S^2_n(i)$$
 is $\le$ as a Fox-Neuwirth Tree than the corresponding distribution 
 $$\bar{X}_i=(S^1_0(i) \sqcup S^2_0(i)) \sqcup \ldots \sqcup (S^1_n(i) \sqcup S^2_n(i))$$
  on the right. For any two numbers $x\neq y$ appearing, we have to check that the depth-order on the left is smaller than the depth-order on the right. The proof is slightly different depending on the following four cases:
\begin{itemize}
    \item $x,y$ both belong to  $|\bar{X}_i|$ for some $i$;
    \item $x,y$ come from two different variables $|\bar{X}_i|, |\bar{X}_j|$;
    \item $x,y$ both belong to $|M|$;
    \item $x$ belongs to $|M|$ and $y$ belongs to $|\bar{X}_i|$
\end{itemize}
We show the first case as an example. If $x \in S^1_p(i), y \in S^1_q(i)$ for some $p,q$, they will have the same depth order in both terms. If $x \in S^1_p(i), y \in S^2_q(i)$, the will have depth order $1$ on the left and $2$ on the right, so we don't have to check the order. Note also that the corresponding function $\alpha$ on partitions
$$ \alpha (S^1_0(i), \ldots, S^1_n(i), S^2_0(i), \ldots, S^2_n(i)  )_{i\in I} = ( S^1_0(i) \sqcup S^2_0(i), \ldots, S^1_n(i) \sqcup S^2_n(i))_{i\in I} $$
is surjective (by taking e.g. $S^2_k(i) = \emptyset$ for all $i$). The other cases are completely analogous.
\end{proof}
\begin{lemma}[Low Split lemma]. Let $M \in \Fox[X_i, Y_j]_{IJ}$ be a monomial without double bars, $Z$'s and $\partial Y$'s. Consider a partition of its support $|M| = C_1 \sqcup C_2$. Then 
$$ \res_{C_1} M || \res_{C_2}M \le M $$
\end{lemma}
The proof is the same as the High Split Lemma.

\begin{lemma}[Juxtaposition lemma]. Let $A, B,M,N \in \Fox[X_i, Y_j, Z_k]_{IJK}$ be monomials without derived variables, and such that $M,N$ do not contain $Z$ variables\footnote{This condition is not strictly necessary, but ensures that the expressions in the lemma do not contain repeated $Z$ variables.}. Suppose $M \le N$. Then
$$ A |^{p} M |^{q} B \le A |^{p} N |^q B$$
for all $p,q \in \{0,1,2\}$.
\end{lemma}
\begin{proof} It is enough to prove that $A |^p M \le A |^p N$; the monotonicity of adding a monomial on the right will be analogous. Let us fix numerical clouds $\bar{X}_i, \bar{Y}_j, \bar{Z}_k$.   In order to prove the inequality we have to describe a map from the distributions on the left to the distributions on the right. Since the choice of partitions for each variable is independent, for the sake of notational simplicity we will only describe the map $\alpha^{L_p}$ for a given variable.

Suppose $L_p$ occurs $m$ times in $M$, $n$ times in $N$ and $a$ times in $A$. Consider a valid distribution of $\bar{L}_p$ in $S^1_1 \sqcup \ldots \sqcup S^1_a \sqcup S^2_1 \sqcup \ldots \sqcup S^2_m$ that arises as a term of the evaluation of $A|M$ in $\bar{L}_p$. Set $S^1 = S^1_1 \sqcup \ldots \sqcup S^1_a$ and $S^2 = S^2_1 \sqcup \ldots S^2_m$. We will call $S_1$ the \textit{left side} of the distribution, and $S_2$ the right side. Since $M \le N$, we know that evaluating $L_p$ in $\res_{S_2} \bar{X}_i$ gives rise to a sum inequality. In particular, there exists a distribution $S^2 = T_1 \sqcup \ldots \sqcup T_n$ such that $M$ on the distribution $S^2_1 \sqcup \ldots \sqcup S^2_m$ is $\le$ than $N$ on the distribution $T_1 \sqcup \ldots \sqcup T_n $. We map the distribution $S^1_1 \sqcup \ldots S^1_a \sqcup S^2_1 \sqcup \ldots S^2_m$ on $A|M$ to the distribution $S^1_1 \sqcup \ldots \sqcup  S^1_a \sqcup T_1 \sqcup \ldots \sqcup T_n$ on $A|N$. 

In order to show the inequality between mapped partitions, consider some tuple of numerical clouds $\bar{X}, \bar{Y}, \bar{Z}$ and numbers $x,y \in |\bar{X} \bar{Y} \bar{Z}| $. If both $x,y$ are 

If $x,y$ both belong to the left side of a distribution, their depth order and index are determined by $A$ for both terms. If $x,y$ both belong to the right side of a distribution, the inequality holds by construction, as a consequence of $M\le N$. If $x$ belongs to the left side of a distribution and $y$ to a right side, their depth index is $p$ and $x < y$ for both terms. 
\end{proof}

\begin{remark} \label{constant-ineq} The combination of these three lemmas gives an efficient way of verifying inequalities for "well-behaved" monomials. Consider $A,B \in \Fox_3[X_i,Y_j,Z_k]_{IJK}$, and let $c(A), c(B)$ the numerical clouds obtained by eliminating all the variables. The process is similar to Def. \ref{elimination}, with constants substituted by variables.

Suppose that $c(A)\le c(B)$. It is possible to see that every Fox-Neuwirth inequality can be realized via a chain of high-splits, low-splits and juxtapositions. This follows from the characterization of the boundary of Fox-Neuwirth strata given in \cite{Giusti}; the dual construction is given, in our language, by splitting a 2-cloud or a 1-cloud and juxtaposing with the remaining clouds.

In case the variables in $A,B$ are arranged adequately, the chain of splits and juxtapositions that relates $c(A)$ and $c(B)$ may upgrade to an inequality of $A\le B$. This always happen for the differential formulas of the multicomplex, as the variables are arranged in a regular pattern depending on the position of indices $i,j,k$. We chose not to make precise this procedure in general, as it may hinders the comprehension of the simple inequalities involved.
\end{remark}

We are ready to show the second hypothesis of Theorem \ref{bdthm}, that is Lemma \ref{differential-bounds}.
\begin{proof} We treat each case $|I|=0,1,2,3$ separately. Notice that we can suppose $I \subset [n]$ if $|I| \neq 1$ because otherwise the differential would be zero.
\begin{description}
\item[$\boxed{|I|=0}$] This case is straightforward, since by the very definition 
$$D_0(\Gamma) = \sum_{\Gamma' \lhd \Gamma} \Gamma' \in \mu(\Gamma) $$
\item[$\boxed{|I|=1}$] Suppose $i$ is an internal index. We write $\Gamma = \bar{Z}_1 || \bar{Y}_1 | \bar{X}_1 i \bar{X}_2 | \bar{Y}_2 || \bar{Z}_2$ as usual. Then we have, for all $i$:
$$ D_1^i(\Gamma) =  \ev_{\bar{X}\bar{Y}\bar{Z}}(Z_1 || Y_1 | X_1 i X_2 | Y_2 || Y_1 | X_1 (i+1) X_2 | Y_ 2 || Z_2)$$
Applying the low split lemma to $Y_1|X_1i(i+1)X_2|Y_2$ and $C_1 = \{i\}, C_2 = \{i+1\}$, we get the polynomial inequality
$$ Y_1 | X_1 i X_2 |Y_2 || Y_1 | X_1 (i+1) X_2 | Y_2 \le Y_1|X_1i(i+1)X_2|Y_2$$
By juxtaposing $Z_1, Z_2$ we get
$$ D_i(\Gamma) = \ev_{\bar{X} \bar{Y} \bar{Z}} (Z_1 || Y_1 | X_1 i X_2 |Y_2 || Y_1 | X_1 (i+1) X_2 | Y_2 || Z_2 ) \le d_i(\Gamma)$$
Which means, by remark \ref{mu-inequality}, that $D_i(\Gamma) \in \mu_i(\Gamma)$.
In case $i=0$ or $i=n+1$, we have
$$ D_1^0(\Gamma) = 1 || \Gamma \le 1 \Gamma = d_0 \Gamma, \ \ \ D_1^{n+1}(\Gamma) = \Gamma || (n+1) \le \Gamma (n+1) = d_{n+1} \Gamma$$
which completes the proof.
\item[$\boxed{|I|=2}$] Here we have to distinguish two subcases, depending on the relation between $I= \{i,j\}$ and $\Gamma$. Without loss of generality, assume $i < j$. If $i<_0 j$ in $\Gamma$, we have $D_I( \Gamma ) =0$. If $i <_1 j$ in $\Gamma$, writing  
$$\Gamma = \bar{Z}_1 || \bar{Y}_1 | \bar{X}_1 i \bar{X}_2 | \bar{Y}_2 | \bar{X}_3 j\bar{X}_4 | \bar{Y}_4 || \bar{Z}_2$$
we have 
$$ D_2^{ij} (\Gamma) = \ev_{\bar{X}\bar{Y}\bar{Z}}(Z_1 || Y_1 | X_1 i X_2 | Y_2 | X_3 (j+1) X_4 | Y_4 || Y_1 | X_1 (i+1) X_2 | Y_2 | X_3 (j+2) X_4 | Y_4 || Z_2 )$$
The strategy above repeats unchanged. The thesis follows by the low split lemma applied to 
$$d_{ij}(\Gamma) =  \bar{Z}_1 || \bar{Y}_1 | \bar{X}_1 i(i+1) \bar{X}_2 | \bar{Y}_2 | \bar{X}_3 (j+1)(j+2)\bar{X}_4 | \bar{Y}_4 || \bar{Z}_2, \ \ \ C_1 = \{i,(j+1)\}, C_2 = \{(i+1), (j+2)\}$$
Followed by juxtaposing $Z_1, Z_2$.

Finally, if $i <_2 j$, we can write 
$$\Gamma = \bar{Z}_1 || \bar{Y}_1 | \bar{X}_1 i \bar{X}_2j \bar{X}_3 | \bar{Y}_2 || \bar{Z}_2 $$
We have
$$ D_2^{ij}(\Gamma) = \ev_{\bar{X}\bar{Y}\bar{Z}}(Z_1 || Y_1 | X_1 i X_2 (j+1) X_3 | Y_2 ||Y_1 | X_1 (i+1) X_2 X_3 | X_1 X_2 (j+2) X_3 | Y_2 || Z_2) +$$
$$+ \ev_{\bar{X}\bar{Y}\bar{Z}}(Z_1 || Y_1 | X_1 (j+1) X_2 X_3 | X_1 X_2 i X_3 | Y_2 ||Y_1 | X_1 (i+1) X_2 (j+2) X_3 | Y_2 || Z_2) $$
Applying the low split lemma to 
$$d_{ij}(\Gamma) = \bar{Y}_1 | \bar{X}_1 i(i+1) \bar{X}_2 (j+1)(j+2) \bar{X}_3 | \bar{Y}_2, \ \ C_1 = \{i,(j+1)\}, \ \ C_2 = \{(i+1), (j+2)\}$$
We get 
$$  Y_1 | X_1 iX_2(j+1) X_3 | Y_2 || Y_1 | X_1(i+1)X_2(j+2)X_3|Y_2 ) \le Y_1 | X_1 i(i+1)X_2(j+1)(j+2)X_3|Y_2$$
By the high split lemma we also have
$$ X_1 X_2(j+1) X_3 | X_1 i X_2 X_3 \le X_1iX_2(j+1)X_3 $$
Juxtaposing $Y_1$ and $Y_2 || Y_1 | X_1(i+1)X_2(j+2)X_3|Y_2$ we obtain
$$ Y_1 | X_1 (j+1) X_2 X_3 | X_1 X_2 i X_3 | Y_2 ||Y_1 | X_1 (i+1) X_2 (j+2) X_3 | Y_2 \le  $$
$$ \le Y_1 | X_1 iX_2(j+1) X_3 | Y_2 || Y_1 | X_1(i+1)X_2(j+2)X_3|Y_2  \le $$
$$ \le  Y_1 | X_1 i(i+1)X_2(j+1)(j+2)X_3|Y_2 $$
Juxtaposing $Z_1, Z_2$ we get the desired inequality for the second term of $D_2$. The first term is completely analogous, if one applies the high split lemma to the right part of the monomial.
\item[$\boxed{|I|=3}$] Without loss of generality, $i < j < k$ both as numbers and in the order of $\Gamma$. If $i<_0 j$ or $j <_0 k$ in $\Gamma$, we have $D_I( \Gamma ) =0$, so we can suppose both indices are $\ge 1$. 

In the planar case $i <_1 j <_1 k$, we have 
$$Z_1 || Y_1 | X_1 i X_2| Y_2 | X_3(j+1) X_4| Y_3 | X_5 (k+2) X_6 | Y_4 ||  Y_1 | X_1 (i+1) X_2| Y_2 | X_3(j+2) X_4| Y_3 | X_5 (k+3) X_6 | Y_4 || Z_2  $$
$$ \le  Z_1 ||Y_1|X_1i(i+1)X_2(j+1)(j+2)X_3(k+2)(k+3)X_4| Y_2   || Z_2 $$
By the low split lemma applied to $C_1 = \{i,(j+1), (k+2)\}$, plus juxtaposition of $Z_1, Z_2$.

For the left mixed case $i <_2 j <_1 k$, note that the "core positions"
\begin{align*}
&    i  (j+1)  |   (k+2)   ||   (i+1)   |   (j+2)  |   (k+3)     \\
&+    (j+1)  | i    |   (k+2)   ||   (i+1) (j+2) |   (k+3)   \ ,  
\end{align*}
resembles the $|I|=2$ case. The same chain of inequalities applied to indices $i,j$, indeed, yields the result. The right mixed case:
\begin{align*}
&    i  | (j+1)  (k+2)   ||   (i+1)  |   (j+2)   |   (k+3)     \\
&+   i  |    (k+2)  | (j+1)    ||   (i+1)  | (j+2)  (k+3)    \ .
\end{align*}
is analogous, with the procedure applied to indices $j,k$.

For the aligned case $i <_2 j <_2 k$, let us recall the heuristics outlined in Remark \ref{constant-ineq}: restrictions to the constants dictate which chain of split and juxtapositions to use. The "core positions" appearing in $D_3$ are 
\begin{align*}
& (j+1) (k+2)|i ||(j+2) |(i+1)(k+3) \\
& +(j+1) (k+2)|i ||(i+1)(j+2) | (k+3)\\
& +(j+1) |i (k+2)||(i+1)(j+2) | (k+3) \\
& +i (k+2)|(j+1) ||(i+1)|(j+2) (k+3)\\
& +(k+2)|i(j+1) ||(i+1) |(j+2) (k+3)\\
& +(k+2)|i(j+1) ||(i+1) (k+3)|(j+2) \\
& +i(j+1) (k+2)||(i+1) |(j+2) | (k+3)\\
& + (k+2)|(j+1) | i ||(i+1)(j+2) (k+3)\\
& +i(j+1) ||(i+1) (k+2)||(j+2) (k+3)\\
& +(j+1) (k+2)||i (k+3)||(i+1)(j+2) \ .
\end{align*}
All these terms are $\le$ than $C:= i(i+1)(j+1)(j+2)(k+2)(k+3)$. Indeed, the two copies of an index always have depth index $0$ in the core positions and $2$ in $C$. Different indices, when appearing in the same $2$-cloud, have the same $i<j<k$ order than in $C$. This is enough. Let us analyze the inequality of associated monomials for the first term. The chain of splits and juxtapositions relating the first term to $C$ is:
\begin{align*}
& (j+1)(k+2)|i||(j+2)|(i+1)(k+3)  \\
& \le i(j+1)(k+2)||(j+2)|(i+1)(k+3) \\
& \le i(j+1)(k+2)||(i+1)(j+2)(k+3) \\
&  \le i(i+1)(j+1)(j+2)(k+2)(k+3)  
\end{align*}
The corresponding chain for monomials is given by:
\begin{align*}
& Y_1|X_1X_2(j+1)X_3(k+2)X_4|X_1iX_2X_3X_4| Y_2 || Y_1|X_1X_2(j+2)X_3X_4 |X_1(i+1)X_2X_3(k+3)X_4|Y_2  \\
& \le Y_1|X_1iX_2(j+1)X_3(k+2)X_4| Y_2 || Y_1|X_1X_2(j+2)X_3X_4 |X_1(i+1)X_2X_3(k+3)X_4|Y_2 \\
& \le Y_1|X_1iX_2(j+1)X_3(k+2)X_4| Y_2 || Y_1|X_1(i+1)X_2(j+2)X_3(k+3)X_4|Y_2 \\
&  \le Y_1|X_1i(i+1)X_2(j+1)(j+2)X_3(k+2)(k+3)X_4| Y_2  
\end{align*}
And the claim follows by juxtaposing $Z_1, Z_2$ at the extremes. In the first equality, we used high split and juxtaposition of the right monomial; in the second, high-split and juxtaposition on the left. The third inequality follows from the low split lemma. The other terms can be analyzed in a similar fashion.
\end{description}
\end{proof}

\begin{remark} \label{casem=2}
Let us conclude the section by stressing that the strategy used for $m=3$ works for $m=2$, with explicit formulas for all higher differentials. The structure, however, has no similar applications, as the non-collapse of the Sinha Spectral Sequence in finite characteristics is already known (see Section \ref{sp-seq-intro}). The $k$-th differential $D_k : \MFN_2(n)_{\bullet} \to \MFN_2(n+k)_{\bullet+k-1}$ is given as a sum
$$ D_k =\sum_{\substack{I \subset \{1,\ldots,n\} \\ |I|=k}}D_I $$
such that, if $I=\{i_1 < \ldots < i_k\}$, we have:
$$ D_I( \bar{Y}_1 | \bar{X}_1 i_1 \bar{X}_2 \ldots \bar{X}_k i_k \bar{X}_{k+1} | \bar{Y}_2) = \hat{D}_I(\Delta_{I} \bar{X}, \Delta_{I} \bar{Y}, \Delta_{I} \bar{Z} )$$
With
$$\hat{D}_I:= Y_1 | X_1 i_1 X_2 \ldots X_k (i_k+k-1) X_{k+1} | X_1 (i_1+1) X_2 \ldots X_k (i_k+k) X_{k+1} | Y_2  $$
for all numerical 0-clouds $\bar{Y}_1,\bar{Y}_2$ and numerical 1-clouds $\bar{X}_1, \ldots, \bar{X}_k$. Notice that there are no $2$-clouds in dimension $m=2$, as the maximum depth on Fox-Neuwirth trees is one.

\end{remark}

%% file: applications.tex
\section{User's Guide to the Higher Differentials}\label{usersguide}

In this section we describe the MATLAB software we created to implement the multicomplex differentials of $\MFN_3$. The routines have been exploited to show the following

\begin{theorem} \label{nc-theorem}
 The Sinha Spectral Sequence $E_{p,q}^r$ in homology does not collapse at page $3$ for ambient dimension $m=3$ and $\F_2$ coefficients.
\end{theorem}
 In this section we use the alternative notation 
$  \FN(k,d) :=  \MFN_3(k)_d $, and all coefficients are taken in $\F_2$.

Let us begin by implementing Fox-Neuwirth trees.

\subsection{Fox-Neuwirth generators}

Recall that a Fox-Neuwirth generator of $\FN(k,d)$ is determined by a permutation in the symmetric group $\Sigma_k$ and by an ordered partition of $d$ as sum of $k-1$ values in $\{0,1,2\}$ (see Section \ref{configurations}). For example, the permutation $P=[1 \,3\, 2 \,4]$ and the triple $N=[1 \,0 \,2]$ represent a generator in $\FN(4,3)$.  In general we represent a sum of generators of this form as a pair of matrices $P,N$ with the same number of rows, and respectively $k$ and $k-1$ columns, such that the rows of $P$ are permutations, and the corresponding rows of $N$ are ordered partitions as above. 

\begin{software} The routine 
\begin{code}
M = partfox(d,k-1,2)
\end{code}
creates a matrix $M$ with $k-1$ columns where the rows are the ordered partitions of $d$ into $k-1$ values in $\{0,1,2\}$, and thus correspond to generators of $\FN(k,d)$ modulo the symmetric group $\Sigma_k$, that are generators of the unordered complex $\textrm{UFN}(k,d)=\FN(k,d)/\Sigma_k$. The number of the rows is the coefficient of $x^d$ in the polynomial $(1+x+x^2)^{k-1}$.
\end{software}
\begin{software} The differential $D_0$ is implemented in the routine
\begin{code}
[DP, DN] = dfn3(P,N)
\end{code}
using the representation for linear combinations of Fox-Neuwirth generators via a pair of matrices.
\end{software}
For example, for $P=[1\, 3 \,2 \,4]$ and $N=[1 \,0\, 2]$,  we obtain 
$$DP=\begin{bmatrix} 3 1 2 4 \\ 1 3 2 4 \\ 1 3 4 2 \\ 1 3 2 4 \end{bmatrix} \quad
DN=\begin{bmatrix} 0 0 2 \\ 0 0 2 \\ 1 0 1 \\ 10 1 \end{bmatrix} \ $$
since $$D_0(1|3||24)=3||1||24+1||3||24+1|3||4|2+1|3||2|4.$$
There is also a routine that writes down the matrix of 
$$D_0:\FN(k,d) \to \FN(k,d-1)$$  
in equivariant form, since $D_0$ is equivariant with respect to action of the symmetric group $\Sigma_k$.
\begin{software} The routine
\begin{code}
[I,J,p,t,s] = fn3ter(k,d)
\end{code}
outputs a sparse matrix with values in $\F_2[\Sigma_k]$, in the following format. The rows correspond to generators of $\FN(k,d-1)$ modulo $\Sigma_k$, and the columns to generators of $\FN(k,d)$ modulo $\Sigma_k$. Notice that $\FN(k,d)$ is a free left $\F_2[\Sigma_k]$-module on partitions $\{x_1,\dots,x_j,\dots\}$  associated to the columns, and similarly $\FN(k,d-1)$ is a free left $\F_2[\Sigma_k]$-module on partitions $\{y_1,\dots,y_i,\dots\}$  associated to the rows. Writing 
$$D_0(x_j)=\sum_{i,j}  p_{i,j} y_i$$
with $p_{i,j} \in \F_2[\Sigma_k]$, the matrix will contain in the position $(i,j)$ the element $p_{i,j}$. Since we store the matrix in a sparse format, this means that for some index $h$ we have $I(h)=i, J(h)=j, p(h)=p_{i,j}^c$, where $p_{i,j}^c$ is the number identifying some permutation that is a summand of $p_{i,j}$. The outputs $t,s$ contain respectively the number of rows and the number of columns of the matrix.
\end{software}
Let us proceed to explain how to work with (co)homology classes of configuration spaces.

\subsection{(Co)-Homology of configuration spaces}
Recall that the cohomology of $\Conf_k(\R^3)$ in dimension\footnote{Note that the cohomology is concentrated in even dimensions \cite{Salvatore2000}.} $2d$ is generated by classes of the form
$$\prod_{l=1}^d \alpha_{i_l ,j_l} \ , $$
where $\alpha_{i,j}=\pi^*_{i,j}([S^2]),$  $\pi_{i,j}:\Conf_k(\R^3) \to S^2$ picks the direction from the $i$-th to the $j$-th point, and
$[S^2] \in H^2(S^2)$ is the fundamental class.
\begin{software} A basis of $H^{2d}(\Conf_k(\R^3))$ is computed by the routine 
\begin{code}
G = genera(k,d)
\end{code}
that outputs an array listing all matrices of shape $2 \times d$ with values in $\{1,\dots,k\}$
$$\begin{bmatrix} j_1 \dots j_d \\ i_1 \dots i_d \end{bmatrix}$$ 
such that $j_1 < \dots <j_k$ and $i_l < j_l$ for each $l=1,\dots,k$, which represents the cohomology class 
$$\prod_{l=1}^d \alpha_{i_l ,j_l} \ . $$
\end{software}
For example, $\texttt{genera}(4,2)$ is an array $2 \times 2 \times 11$ because $H^4(\Conf_4(\R^3))$
is 11-dimensional, and the last generator is
$$\texttt{G(:,:,11)}= \begin{bmatrix} 3  & 4 \\ 2 & 3 \end{bmatrix} \ ,$$ 
that represents $\alpha_{2,3} \alpha_{3,4}$.

It is useful to consider the normalized (co)homology.
 Consider the maps $$s_j:\Conf_{k}(\R^3) \to \Conf_{k-1}(\R^3)$$
  forgetting the $(j+1)$-th point and renumbering. The normalized homology is 
 $$NH_*(\Conf_{k}(\R^3))   =  \cap_{i=0}^{k-1} \textrm{ker}(s_i)_* \subset H_*(\Conf_{k}(\R^3)),$$ and dually the normalized cohomology is obtained as a quotient
  $$NH^*((\Conf_{k}(\R^3)) = H^*((\Conf_{k}(\R^3)) /   \sum_{i=0}^{k-1} Im(s_i^*).$$

\begin{software} A basis for the normalized cohomology of $\Conf_k(\R^3)$ in dimension $2d$ is provided by the routine 
\begin{code}
G = corde(k,d)
\end{code}
that generates an array of matrices similar to \texttt{genera}, but where {\em all} numbers between $1$ and $k$ appear.
\end{software}

For example, $\texttt{corde(4,2)}$ is a list of three matrices, representing respectively
$\alpha_{2,3}\alpha_{1,4}, \,  \alpha_{1,3} \alpha_{2,4}$ and $\alpha_{1,2} \alpha_{3,4}$.

\

The homology basis dual to \texttt{genera} has generators represented by "iterated planetary systems"  \cite{Salvatore2000}, i.e. embeddings $(S^2)^d \to \Conf_k(\R^3)$ such that 
\begin{itemize}
\item for each  factor $\alpha_{i_l, j_l}$ there is a point $j_l$ orbiting around $i_l$ 
(notice that $i_l <j_l$);
\item for "planets" orbiting around the same "sun" a smaller index corresponds to a larger orbit. 
\end{itemize}
This duality restricts to a duality between the normalized cohomology basis by \texttt{corde} and the normalized homology basis by planetary systems with no "isolated stars". 

\subsection{The multicomplex differentials}
The truncated multicomplex structure on $\FN$ introduced in Section \ref{fn-mcpx-3} results in linear maps
$$D_i:\FN(k,d) \to \FN(k+i,d+i-1), \ \ \ 0 \le i \le 3 $$
with $D_0$ being the differential of the chain complex $\FN(k,\bullet)$ seen in the previous subsection, and $D_1$ resembling the differential induced by the Kontsevich cosimplicial structure (see Def. \ref{d1-fn}). Because of Thm. \ref{sinha-mcpx}, the implementation of these maps is the first step towards computations in the Sinha spectral sequence.
\begin{software}
The routine computing the first differential $D_1$ is implemented as
\begin{code}
[D1P, D1N] = d1fn3veloce(P,N)
\end{code}
The algorithm sums $(k+2)$ contributions, corresponding to doubling each of the $k$ points, plus adding a new point at one of the two extremes.
\end{software}
For example, for $\texttt{P}=[1 2],\, \texttt{N}=2$ we obtain 
$$\texttt{D1P}=\begin{bmatrix} 2 1 3 \\ 1 3 2 \end{bmatrix}, \ \ \texttt{D1N}=\begin{bmatrix} 0 2\\ 2 0 \end{bmatrix} $$
since $D_1(12)=D_1^0(12)+D_1^1(12)+D_1^2(12)+D_1^3(12)= 1||23+(13||2+1||23)+(12||3+2||13)+12||3=
 2||13+13||2.$
The next differential $D_2$ doubles pairs of points that have at least a coordinate in common, distributes the others, and renumbers (see Def. \ref{d2-fn}). The formula for the duplication depends on whether the number of coordinates in common is one or two.
\begin{software}
The routine computing the second differential $D_2$ is implemented as
\begin{code}
D2PD2N = d2gen(P,N)
\end{code}
Here the matrices of permutations and partitions are stacked together in a single matrix, as opposed to the case of $D_1$ and $D_0$. In other words, if $\texttt{P}$ has $k$ columns, then the first $k+2$ columns of $\texttt{D2PD2N}$ form a matrix $\texttt{D2P}$ having permutations as rows and the last $k+1$ columns of $\texttt{D2PD2N}$ form a matrix $\texttt{D2N}$ that has ordered partitions as rows. 
\end{software}
For example, if $\texttt{P}=[1 2], \, \texttt{N}=2$ then 
$$\texttt{D2PD2N}=\begin{bmatrix} 1 3 2 4 \; 2 0 1 \\3 1 2 4 \; 1 0 2 \end{bmatrix}$$
since $D_2(12)=13||2|4+3|1||24$.
Finally, we implement a partial computation of $D_3$, that doubles triples of points with coordinates in common, distributes and renumbers. The implementation applies to generators with ordered partitions with no $1's$ appearing and with at most three consecutive $2's$. 
\begin{software}
The routine computing the third differential $D_3$ is implemented as
\begin{code}
D3PD3N = d3fn3(P,N) \ .
\end{code}
It provides the summands of $D_3$ with no $1's$ in the ordered partition. As in the case of $D_2$, the output is formatted as a single matrix stacking together
a matrix with $k+3$ columns with permutations as rows,  and a matrix with $k+2$ columns with ordered partitions as rows.
\end{software}
For example,
$$\texttt{d3fn3}([1  2  3 ],[2  2 ])= \begin{bmatrix}
 1    3     2    5     4     6   \;  2     0     2     0    2     \\
 3     5     1     6     2     4   \;   2     0     2     0   2
\end{bmatrix} $$
since the only summands in $D_3(123)$ with no single bars (i.e. $1's$ in the partition) are $13||25||46$ and $35||16||24$. 
\subsection{The $E_2$ term and Turchin's calculations}
Recall from Section \ref{kons-cosimp} that the Kontsevich spaces, homotopy equivalent to configuration spaces, have a cosimplicial structure.
Thus the homology of the configuration spaces forms a cosimplicial graded vector space,
with codegeneracies induced by the $s_i$ and cofaces 
induced by maps $$d_i:\Conf_k(\R^3) \to \Conf_{k+1}(\R^3)$$  such that   
\bit
\item $d_i$ "doubles" the $i$-th point for $i=1,\dots,k$ and renumbers
\item $d_0$ "inserts" a point at $-\infty$ and renumbers
\item $d_{k+1}$ "inserts" a point at $+\infty$.
\eit

The first page of the Sinha spectral sequence is 
$$\overline{E}^1_{k,2d} =H_{2d}(\Conf_k(\R^3))$$ 
with basis provided by \texttt{genera(k,d)} as seen in the previous section.
Notice that we change the sign of $k$ with respect to the standard bigrading of a cosimplicial homology spectral sequence,
so that in the spectral sequence $d^r$ has bidegree $(r,r-1)$ as the differential $D_r$ of the multicomplex $\FN$. 

The normalized homology 
$E^1  \subseteq \overline{E}^1 $ has a basis provided by \texttt{corde(k,d)}, that is a subset of the full basis.
The differential $d^1$ of the first page of the spectral sequence is the sum $d^1=\sum_i (d_i)_*$.
By a standard argument the $d^1$ restricts to $E^1$, and the inclusion $(E^1,d^1) \subset  (\overline{E^1},d^1)$ is a quasi-isomorphism, so that $H(E^1,d^1) \cong H(\overline{E^1},d^1) \cong E^2$,
We implement the computation of the differential 
$$d^1_{k-1,2d}: E^1_{k-1,2d} \to E^1_{k,2d}$$ 
by 
\begin{code}
A=bousfield(k,d)
\end{code}
where $A$ is a sparse matrix with as many rows as the dimension of 
$E^1_{k,2d}$ and as many columns as the dimension of $E^1_{k-1,2d}$.
This allows the computation of the dimension of the $E^2$-term. Our computations
confirm those by Turchin \cite{tourtchine2005bialgebra}.

Remember that $E^\infty_{k,2d}$ wants to compute a subquotient of $H_{2d-k}(\overline{Emb_3})$, where
$\overline{Emb_3} \simeq Emb_3 \times \Omega^2 S^2$ is the space of long knots modulo immersions, 
and $Emb_3$ is the space of long knots in $\R^3$. 
However the spectral sequence converges to the homology of spaces of knots only in higher dimensions. 
There is also a version of the spectral sequence $\underline{E}$ for the homology of the space of long knots $Emb_3$ 
with $\underline{E}^1_{k,2d}=H_{2d}(Conf_k(\R^3) \times (S^2)^k) $.
Moreover there is a homomorphisms of spectral sequences 
$E \to \underline{E}$ that wants to compute the map in homology induced by the projection $\overline{Emb}_3 \to Emb_3$.

Turchin shows in \cite{tourtchine2005bialgebra} that the non-zero part of $E^2_{k,2d}$ is concentrated in degrees such that $d< k \leq 2d$;
he also shows that there is a differential bigraded algebra structure on $E^1$
and $H(E^1)=E^2$ is the tensor product of $\underline{E}^2$ and a free polynomial algebra on {\em framing classes}
$\iota, Q_1(\iota), \dots, Q_1Q_1(\iota), \dots $ such that 
$\iota \in E^2(2,2),  \, Q_1(\iota) \in E^2(3,4), \, Q_1Q_1(\iota) \in E^2(5,8)$
and in general $Q_1^n(\iota) \in E^2(2^n+1,2^{n+1})$. We can consider this free polynomial algebra as
the homology of the union of unordered configuration spaces in $\R^2$, and as
a subalgebra of $H_*(\Omega^2 S^2)$,
with  $Q_1^n \in H_{2^n-1} (\Omega^2 S^2)$. 

Here are some cycle representatives of the generators in $E^1$: 
$$\iota=[(\alpha_{1,2})^*]$$
$$\iota^2=[(\alpha_{1,2}\alpha_{3,4} )^*]$$
$$Q_1(\iota)= [(\alpha_{1,2} \alpha_{1,3})^*]$$

We believe that these classes survive to the term
$E_\infty$, and so the potential non trivial differentials involve "Vassiliev" class from the factor $\underline{E}^2$.

The first such class  is
$$\beta=[(\alpha_{1,3}\alpha_{2,4} )^*] \in E^2(4,4)$$
that together with $\iota^2$ spans $E^2(4,4)$.

We are interested in a class in $E^2(6,8)$. This vector space has dimension 2, generated by 
$$Q_1(\iota)^2 = [(\alpha_{1,2} \alpha_{1,3} \alpha_{4,5} \alpha_{4,6})^*]$$
and a Vassiliev class $v \in \underline{E}^2(6,8)$.
This class is represented by a sum of 36 generators out of the 130 in $E^1(6,8)$.
To compute it we first do
\begin{code}
K=ker(bousfield(7,4))
\end{code}
This expresses a basis of $ker(d^1:E^1(6,8) \to E^1(7,8))$ with respect to the basis 
of $E^1(6,8)$, namely \texttt{K} is a matrix $130 \times 25$ such that each column is the coordinate vector of a basis element of the kernel
with respect to the basis of $E^1(6,8)$.
In particular the first basis vector of \texttt{K} is exactly the representative above of
$Q_1(\iota)^2$, that is the 79th element in the basis of $E^1(6,8)$.

It turns out that 
the $130 \times 24$  matrix 
\begin{code}
bou64=bousfield(6,4)
\end{code}
 expressing $$d^1:E^1(5,8) \to E^1(6,8)$$ has rank $23$. If we add the first column of $K$ to \texttt{bou64} the rank increases to $24$, since we have added the first generator of $E_2$ not in the image of $d^1$. If we keep adding columns the ranks remains constant until we add the 7th column of \texttt{K} (the rank becomes 25), so that column is the representative $v$ we are interested in. The 36 generators of $E^1(6,8)$ whose sum is $v$ are in position
 
 1     5     7    12    13    14    19    20    22    23    25    28    30    33    35    36    41    42    45    47    48    50    51  53    54    56    61    65    67    68    73    74    75    76   102   109

This list is obtained by the command 
\begin{code}
find(K(:,7))
\end{code}

Our objective is to show that $d^3(v)$ is not trivial.
($d^2$ is always trivial for dimensional reasons).

\subsection{Fox Neuwirth (co)cycles and (co)homology}

We explain how to construct (co)cycle representatives of (co)homology basis generators in the Fox Neuwirth complex.
There is a basis of the (co)homology of the configuration spaces, due to Sinha, that is different from the planetary system basis. Consider first the top-dimensional cohomology group $H^{2(k-1)} (\Conf_k(\R^3))$, that  has dimension $(k-1)!$. The Sinha basis 
contains all elements of the form
 $$\alpha_{1,i_1} \alpha_{i1,i2} \dots  \alpha_{i_{k-2},i_{k-1}}$$
 where $(i_1  \dots i_{k-1})$ varies among the permutations of $\{2, \dots ,k\}$.

\medskip

Geometrically, if we want to evaluate this class on a transverse cycle of dimension $2(k-1)$, we should count (mod 2) how many configurations in the cycle have all points $(x_i,y_i,z_i)$ ordered vertically, i.e. with $x_1=\dots=x_k, \, y_1=\dots,y_k$, and   $z_1 < z_{i_1} < \dots < z_{i_{k-1}}$. In the Fox Neuwirth complex this cohomology class is represented by the cocycle  sending the tree
$(1<_2 i_1 <_2 \dots <_2 i_{k-1})=1i_1\dots i_{k-1}$ to 1 and all other trees to zero.
  
  \
 More generally an element of the Sinha cohomology basis in any dimension  will be given by a collection of vertical configurations such that the lower point (with respect to the coordinate z) in each vertical configuration has the lower index appearing in it, and distinct configurations have no coordinates in common. If we ask that these vertical configurations are ordered from left to right according to their lower index,  we can associate
 to each cohomology class of the Sinha basis a Fox-Neuwirth tree, that we call {\em Sinha tree}.  For example the cohomology class $\alpha_{1,5}\alpha_{5,3}\alpha_{2,4}$  
is associated to the tree $(1<_2 5 <_2 3 <_0 2 <_2 4)=153||24$. The "odd" configuration comes before the "even" because $1 < 2$.
We warn that the cochain sending this Fox Neuwirth tree to 1 and the others to 0 is NOT a cocycle when the points are not all vertically aligned. 
A cocycle representing a cohomology class of the Sinha basis will in general evaluate to 1 on all trees in which the given vertical configurations appear, permuted in all possible ways.
For example a representing cocycle of the class above 
is the sum of the duals of $1<_2 5 <_2 3 <_0 2 <_2 4$ and $2 <_2 4 <_0 1<_2 5 <_2 3$, i.e. 
$(153||24)^*+(24||153)^*$.
 
\
Now we go back to the planetary system basis. Each product of the form $\prod_{s=1}^l \alpha_{i_s,j_s}$ in the basis defines 
 a graph with $k$ vertices, with an edge from $i_s$ to $j_s$ for $s=1,\dots,l$.
We can write down a representing cycle of the corresponding homology generator in the FN complex, inspired 
by the geometry of the planetary system representing the class.
  If the graph is connected, the cycle has a summand for each configuration that is vertically aligned. When the graph is not connected, we order its components from left to right according to the lowest indices similarly as above, and consider configurations that are vertically aligned for each connected component.

For example the class dual to $\alpha_{1,2}$  is represented by the sum of two generators
$12+21$, the class dual to $\alpha_{1,3}\alpha_{2,4}$ is represented by
$13||24+31||24+13||42+31||42$,
and more generally a homology basis generator associated to a graph with $h$ edges is represented by a cycle with $2^h$ tree summands.
The routine 
\begin{code}
PN=ciclofn3(x)
\end{code}
produces a cycle representative of any given element $x \in H_{2d}(\Conf_k(\R^3))$, 
presented as a $2 \times d \times m$ array if it is the linear combination of $m$ planetary generators. The output \texttt{PN} is a matrix with $2k-1$ columns, obtained by stacking together the permutation part \texttt{P} (with $k$ columns) and the partition part \texttt{N}
(with $k-1$ columns). 
Notice that the partition part contains only $2's$ or $0's$, but not $1's$.
\

In particular we obtain a cycle representative $vas$ of $v$ as follows:
\begin{code}
C=corde(6,4); \;  vas = ciclofn3(C(:,:,find(K(:,7)))); 
\end{code}
This is the linear combination of 328 generators in $\FN(6,8)$.
For example the first generator has \texttt{P=[1 6 5 4 2 3]} and \texttt{N=[2 2 2 0 2]}, i.e. it is $1654||23$.

\subsection{The algorithm of equivariant reduction}

We need to be able to solve the equation $D_0(x)=c$ in $x$ for a given boundary $c$ in the Fox Neuwirth complex.
This is the bottleneck of our computation. We must use the equivariance of $D_0$ with respect to the symmetric group to construct an efficient algorithm, otherwise time and memory would not be sufficient.
Given free {\em right} $\Sigma_k$-modules $N$ and $M$ of respective ranks $n$ and $m$, a $\Sigma_k$-equivariant  homomorphism $D:N \to M$  can be represented in coordinates by multiplication by an $m \times n$ matrix $A$ with values in 
$\F_2[\Sigma_k]$. Given $c \in N$, let 
$v_c$ be its coordinate vector. We look for $x \in M$ such that $D(x)=c$. If $v_x$ and $v_c$ are the respective coordinate vectors of $x$ and $c$, the equation becomes
$A v_x = v_c$. The idea is to perform Gauss elimination, as long as we have elements of $\F_2[\Sigma_k]$ that are easily invertible: summands with a single permutation. When we run out of these, we consider just equivariance with respect 
to $\Sigma_{k-1}$: the matrix size will increase to $m*k \times n*k$ but the values will be
in $\F_2[\Sigma_{k-1}]$, and we will have potentially invertible elements.
We continue this procedure passing to $\Sigma_{k-2}$, and so on, until one solution is found, if it exists. This algorithm seems to work well. It is implemented in the following routine:
\begin{code}
 [t,Ix,px]=solvsimmnew(I,J,p,m,n,k,Ic,pc,cpk) 
\end{code}
Here we need to use the MATLAB command \texttt{perms(1:k)} that provides a
$k! \times k$ matrix where the rows are the elements of $\Sigma_k$, listed in reverse lexicographic order. In particular the last element is the identity. So any permutation is identified by the number of its row. In the input we need to prepare a matrix \texttt{cpk} of size  $k! \times k!$ with values in $\{1,\dots,k!\}$ encoding the composition of permutations. This speeds up computations. The vector $v_c$  is encoded in sparse 
form by two MATLAB vectors \texttt{Ic,pc} of the same length with the pairs \texttt{(Ic(h),pc(h))}  all distinct,
encoding that the vector $v_c$ has at the entry number \texttt{Ic(h)} a summand permutation identified by its number \texttt{pc(h)}. Similarly the matrix $A$ is encoded in sparse form
by the vectors \texttt{I,J,p} of the same length, with the triples \texttt{(I(h),J(h),p(h))} all distinct,
encoding that the entry \texttt{(I(h),J(h))} of the matrix has a summand permutation identified by  
its number \texttt{p(h)}. The output provides the smallest $t$ such that we had to perform the reduction over $\Sigma_t$. In addition \texttt{(Ix,px)} encodes in sparse form the solution
$v_x$, when it exists. 

\

We have also an algorithm multiplying matrices  
in $\F_2[\Sigma_k]$ in sparse form, that is useful to verify that a solution is correct.
\begin{code}
 [Ic,Jc,pc]=moltsimmcp(Ia,Ja,pa,n,Ib,Jb,pb,k,cpk) 
 \end{code}
Here $A$ and $B$ are matrices in $\F_2[\Sigma_k]$ with number $n$ of columns of $A$ equal to the number of rows of $B$; they are encoded in sparse form as above respectively by \texttt{(Ia,Ja,pa)} and \texttt{(Ib,Jb,b)}. the output $C=AB$ is encoded in sparse form by \texttt{(Ic,Jc,pc)}.

\

Notice that the FN-complexes are naturally {\em left}-modules over the symmetric group algebras, so that we always need to pass to the right action before applying the reduction algorithm (both for the matrix and the known term), and then also convert the solution from right to left action.

\subsection{Computing the differential}

We recall from \ref{multicpx-back}
that the procedure to compute a differential in the spectral sequence 
defined by a multicomplex  is the following, starting with $vas \in \FN(6,8)$.

\begin{enumerate}
\item Find $x \in \FN(7,9)$ such that $D_0(x)=D_1(vas)$.
\item Find $y \in \FN(8,10)$ such that $D_0(y)=D_2(vas)+D_1(x)$.
\item Consider the homology class of the cycle 
$$D_1(y)+D_2(x)+D_3(vas) \in \FN(9,10)$$
and show that is a non trivial element of $E^2(9,10) \cong E^3(9,10)$.
\end{enumerate}

The computation of $D_1(vas) \in \FN(7,8)$ is performed by the command
\begin{code}
[pe1,nu1]=d1fn3veloce(vas(:,1:6),vas(:,7:11)); 
\end{code}
where as usual $\texttt{pe1}$ encodes permutations and $\texttt{nu1}$ partitions:
 \texttt{pe1} is a $4672 \times 7$ matrix and \texttt{nu1} is a $4672 \times 6$ matrix,
showing that $D_1(vas)$ is a sum of $4672$ generators. The first rows of the matrices stacked together are
respectively
$$	[1     7     6     5     2     4     3   ,\quad    2     2     2     0     2     0] $$
so the corresponding generator is $176||24||3$.
In order to solve the equation $D_0(x)=D_1(vas)$ we need to encode $D_1(vas)$
in sparse form as linear combination with coefficients in $\F_2[\Sigma_7]$. This is done by the command
\begin{code}
[Inoto,pnoto]=recofox([pe1 nu1])
\end{code}
producing the pair of vectors \texttt{(Inoto,pnoto)}.
We need a $7! \times 7!$  matrix \texttt{cp7} with values in $\{1,\dots,7!\}$ encoding the multiplication in $\Sigma_7$. We also need a vector \texttt{ip7} of length $7!$ with values in $\{1,\dots,7!\}$ encoding the number of the inverse permutation, in the sense
that the $i$-th entry of the vector is the number (in the reverse lexicographic order) of the inverse of the $i$-th permutation. 
These can be obtained by%parte nuova
\begin{code}
cp7=comppermnew(7);

ip7=invpermnew(7);
\end{code}
We pass from the left to the right action description  of our vector $D_1(vas)$ simply inverting its permutation coefficients by
\begin{code}
pnoto=ip7(pnoto);
\end{code}
Then we input the matrix of $D_0:\FN(7,9) \to \FN(7,8)$ by
\begin{code}
[I7,J7,p7,ii7,jj7]=fn3ter(7,9);
\end{code}
followed by  
\begin{code}
p7=ip7(p7);
\end{code}
 that also switches from left to right description.
Then we can solve the system 
\begin{code}
[dull,Isol,psol]=solvsimmnew(I7,J7,p7,ii7,jj7,7,Inoto,pnoto,cp7);
\end{code}
and convert the solution into left description by \texttt{psol=ip7(psol);}

The solution is the sum of 5088 generators. We want to convert it from a sparse vector form into an explicit list of FN trees. So we first compute the basis
of $\FN(7,9)$ modulo $\Sigma_7$ by
\begin{code}
base=partfox(9,6,2);
\end{code}
and then using the list of permutations \texttt{P7=perms(1:7);}
we write 
\begin{code}
pe7=P7(psol,:);   \; nu7=base(Isol,:); 
\end{code}
that is the desired list of generators of $x$.
For example the first rows of \texttt{pe7} and \texttt{nu7} are
\begin{code}
[ 3     6     5     4     2     7     1  ,\quad  2     1     2     0     2     2]
\end{code}
representing $36|54||271$.
Next we compute $D_1(x) \in \FN(8,9)$ by 
\begin{code}
[pe8,nu8]=d1fn3veloce(pe7,nu7);
\end{code}
The result has $140480$ summands, so \texttt{pe8} is a $140480 \times 8$ matrix
and \texttt{nu8} a $140480 \times 7$ matrix. Their first rows are
\begin{code}
[ 8     6     5     4     7     1     2     3,   \quad    0     0     2     1     2   2  2]
\end{code}
representing $8||6||54|7123$,
We convert $D_1(x)$ into sparse form by
\begin{code}
[I,p]=recofox([pe8 nu8] );
\end{code}
and then into right description by \texttt{p=ip8(p);} where 
the vector $ip8$ encodes the inverse permutation number in $\Sigma_8$, and is obtained as
\texttt{ip8=invpermnew(8);}

\

We also need to compute $D_2(vas)$. This is done by the routine
\begin{code}
d2=triald2fn3(vas(:,1:6),vas(:,7:11));  %d2gen should be fine
\end{code}
that computes $D_2$ for a sum of trees with no $1$'s in the partition,
and then we put the result into sparse vector form by 
\begin{code}
[Ierr,perr]=recofox(d2);
\end{code}
and convert it to right module form by 
\begin{code}
perr=ip8(perr);
\end{code}
There are $13904$ summands in  $D_2(vas)$.

\

We compute the sum $D_1(x)+D_2(vas)$ in sparse vector form by the routine
\begin{code}
compl=agg([I',p'],[Ierr',perr']);  \; Icompl=compl(:,1)'; \;  pcompl=compl(:,2)'; 
\end{code}
The result has $140520$ summands.
We want to find next $y \in \FN(8,10)$ such that $D_0(y)=D_1(x)+D_2(vas)$.
This is possible because there are no non-trivial homology classes in degree 9 
and $D_1(x)+D_2(vas)$ is a cycle by the multicomplex identities. 

\

We encode the matrix of $D_0:\FN(8,10) \to \FN(8,9)$  by 
\begin{code}
[Ib,Jb,pb,ib,jb]= fn3ter(8,10);  \;  pb=ip8(pb); 
\end{code}
This is a $266 \times 161$ matrix with values in $\F_2[S_8]$.
We solve the system by 
\begin{code}
[dull,Ifine,pfine]=solvsimmnew(Ib,Jb,pb,266,161,8,Icompl,pcompl,cp8);
\end{code}
This is the longest computation: it takes almost 4 hours on a PC.
Here $cp8$ is a $8! \times 8!$ matrix with values in $\{1,\dots,8!\}$ encoding 
composition in $\Sigma_8$, obtained by
\texttt{cp8=comppermnew(8);}

Then we switch to left description by 
\begin{code}
pfine=ip8(pfine);
\end{code}
The result $y$ is a sum of $109910$ generators. To convert it into a list of FN generators we do 
\begin{code}
P8=perms(1:8); \; prima=P8(pfine,:); \; basis=partfox(10,7,2); \; seconda=basis(Ifine,:); 
\end{code}
For example the first row of \texttt{prima} and \texttt{seconda}  are
\begin{code}
[ 3     1     6     8     4     7     5     2, \quad 2     0     2     0     2     2     2 ]
\end{code}
representing $31||68||4752$.
\

Next we compute $D_1(y) \in \FN(9,10)$ by
\begin{code}
[pe,nu]=d1fn3veloce(prima,seconda);
\end{code}
The result has $4631462$ generators.
The first rows of \texttt{pe} and \texttt{nu} are
\begin{code}
 [9     7     6     5     4     8     1     2     3, \quad  0     0     2     1     1     2     2     2 ]
 \end{code}
representing $9||7||65|4|8123$.
We need to add $D_2(x)+D_3(vas)$  to $D_1(y)$ in order to get a cycle (by the multicomplex identities). The point is that in order to recognise the homology class of 
$D_1(y)+D_2(x)+D_3(vas)$ it is sufficient to evaluate it against the Sinha cocycles, and
$D_2(x)$ provides no contribution, because its generators have always a $1$ in the 
partition part. So it is enough to add $D_3(vas)$ to $D_1(y)$  and then evaluate this sum on the Sinha cocycles. For practical reasons we choose to evaluate first $D_1(y)$, then $D_3(vas)$, and then take the sum of the evaluations. 

\

First of all we get rid of the generators of $D_1(y)$ with a $1$ in the partition part, that are most of them, and do not contribute.
\begin{code}
[Ieh,Jeh]=find(nu==1); \; Ieh=unique(Ieh); \; buoni=setdiff(1:4631462,Ieh);  

pebuoni=pe(buoni,:); \; nubuoni=nu(buoni,:); 
\end{code}
In this way we cut down to $240802$ generators. The first rows of \texttt{pebuoni} and \texttt{nubuoni}
are
\begin{code}
[ 8     9     6     7     4     5     1     3     2 ,\quad  0     2     0     2     0     2     2     2]
\end{code}
representing $8||96||74||5132$.
Notice that each Fox Neuwirth tree evaluates non trivially with at most one Sinha cocycle.
For each generator of $D_1(y)$ that evaluates non trivially on a Sinha cocycle the routine   
\begin{code}
[dev,quanto]=sinha(pebuoni,nubuoni); 
\end{code}
adds a new row to the matrices \texttt{dev} and \texttt{quanto},
describing a tree that also evaluates non-trivially on that cocycle, and looks like a Sinha tree, i.e. a bunch of vertical configurations.
The only difference with a Sinha tree is that the possible "isolated point" is always on the right, regardless of its label
(whereas in the Sinha basis all components are ordered according to their lower indices, isolated or not).
The routine \texttt{sinha} is an ad hoc routine that knows the possible partitions 
in \texttt{nubuoni}. There can be repeated rows in the
stacked matrix \texttt{[dev \; quanto]}. 

 \texttt{dev} is a $17350 \times 9$ matrix and
\texttt{quanto} is a $17350 \times 4$ matrix .
As an example the first row of \texttt{[dev \; quanto]} is 
\begin{code}
[ 1     5     3     2     4     7     6     9     8 \quad  4     2     2     1 ]
\end{code}
meaning that some generator of $D_1(y)$ contains the vertical configurations 
$1532$ $47$ $69$ $8$ , separated by a set of double bars
(for example it could be $47||1532||8||69$).

%$$1<_2  5<_2 3<_2 4, \; 4<_2 <_2 7, \; 6 <_2 9 \; $$

We need to 
count rows mod 2, so we do
\begin{code}
[undev,iad,icd]=unique([dev \;  quanto],'rows'); 
\end{code}
finding that there are $7079$ distinct rows, and then 
\begin{code}
for \; indice=1:7079; \; nice(indice)=nnz(icd==indice); \; end  

nice=mod(nice,2); \; scremati=undev(find(nice),:); 
\end{code}
The result is a linear combination of $690$ Sinha trees, since most rows appear in pairs and cancel each other.
For example the first row of \texttt{scremati} is
\begin{code}
[ 1     3     2     5     4     8     6     7     9,    \quad  2     2     2     3]
\end{code}
It turns out that the rows are all non-degenerate, in the sense that there is no vertical configuration with a single point, i.e. no $1's$ appear on the right hand side. Thus the rows are actual Sinha trees, representing generators in normalized homology of the dual Sinha basis, and the matrix 
\texttt{scremati}
represents their sum. We want to write this homology class in terms of the standard planetary basis we used in homology.
We observe that there is a bijection between Sinha trees and planetary basis vectors: this associates to a planetary element the vertical configuration where the height  of each planet is greater than the height of the star it revolves around. 
For example $\alpha_{1,3} \alpha_{1,4} \alpha_{2,5}$ goes to 
$[1 4 3  2 5 \quad 3 2]$  since there is a vertical configuration with $z_1<z_4<z_3$, as
1 is a "star", 4 a "planet with small orbit", and 3 a "planet with large orbit"; 
similarly there is a vertical configuration with $z_2<z_5$ since 2 is a star with a planet 5.
We convert all non-degenerate (or normalized) planetary generators in $H_{10}(\Conf_9(\R^3))$,
spanning a $2520$-dim. vector space, into the associated Sinha trees by
\begin{code}
Msinha=cordesinha(corde(9,5));
\end{code}
so that \texttt{Msinha} is a matrix $2520 \times 13$.
Its first row is 
\begin{code}
[ 1     9     8     2     7     3     6     4     5,  \quad    3     2     2     2]
\end{code}
associated to $$\alpha_{4,5} \alpha_{3,6} \alpha_{2,7} \alpha_{1,8} \alpha_{1,9} $$
 %5     6     7     8     9
   %  4     3     2     1     1
We can express the linear combination \texttt{scremati} by changing basis into the planetary basis by
\begin{code}
vettore=cambiobase(scremati,Msinha); 
\end{code}
Since here we have at most 3 points in a vertical configuration, appearing once, a Sinha generator will be the sum of two planetary generators if it has a vertical configurations $z_i < z_j  < z_k$ with $i<j<k$, and will be a planetary generator in all other cases (where $i<k>j$ or where there are just 2 points in each vertical configuration).
Now \texttt{vettore} is a vector with $2520$ representing the evaluation of $D_1(y)$ on appropriate cocycles for the dual planetary basis
(that are a single Sinha cocycle, or the sum of two of them).

\

We need to do a similar procedure for the term $D_3(vas)$:
we compute it in
\begin{code}
XN=d3fn3(vas(:,1:6),vas(:,7:11)); 
\end{code}
for example the first row of \texttt{XN} is 
\begin{code}
[ 1     8     6     9     4     7     5     2     3,   \quad  2     2     0     2 0 2 0 2 ]
\end{code}
i.e. $186||94||75||23$,
and \texttt{XN} has $3932$ rows.
The rows that evaluate non trivially against Sinha cocycles, similarly as before, are non-degenerate, in the sense that they are FN trees
of vertical configurations with "no isolated point", and without $1's$ as indices. 
Namely $vas$ is a sum of non-degenerate trees,
 and in the formula for $D_3$ the part that does not introduce $1's$ sends non-degenerate trees to sums of non-degenerate trees.

The routine

\begin{code}
[dev1,quanto1]=sinha1(XN(:,1:9),XN(:,10:17)); 
\end{code}
produces a matrix with a row that is a non-degenerate Sinha tree, for each non-trivial evaluations of a generator of $D_3(vas)$ on
a Sinha cocycle associated with that tree.
\texttt{sinha1} knows exactly which vertical configurations can occur in \texttt{XN}.
The first rows of \texttt{dev1} and \texttt{quanto1} are
\begin{code}
[  1     4     2     8     3     6     5     9     7, \quad  2 2 2 3],
\end{code}
%(consistent with the row above)
and both \texttt{dev1} and \texttt{quanto1} have $184$ rows.
It is not necessary to reduce to a unique set of rows (there are only two repeated rows) since the routine \texttt{cambiobase} adds the contribution of each row mod 2.
\begin{code}
vettore1=cambiobase([dev1 \; quanto1],Msinha); 
\end{code}
This is the evaluation of $D_3(vas)$ on the cocycles for the planetary basis cohomology as above.

\

Finally we add the contibutions of $D_1(y)$ and $D_3(vas)$  mod 2
\begin{code}
vettoretot=xor(vettore,vettore1); 
\end{code}
finding the homology class $d^3(vas)$ of the {\em cycle} $D_1(y)+D_2(x)+D_3(vas)$, that is in normalized homology.

We check that the resulting vector  is in $ker(d^1)$.
The matrix 
\begin{code}
bou105=bousfield(10,5); 
\end{code}
 is a matrix $945 \times 2520$
representing $d^1:\FN(9,5) \to \FN(10,5)$. The command
\begin{code}
nnz(mod(vettoretot*bou105',2)) 
\end{code}
gives 0 so indeed $d^3(vas) \in ker(d^1)$.
Now we check whether $d^3(vas) \in Im(d^1)$.
The matrix
\begin{code}
bou95=bousfield(9,5); 
\end{code}
is a matrix $2520 \times 2380$ representing 
$d^1:\FN(8,5) \to \FN(9,5)$.
Its rank mod 2 is $1574$, computed with
\begin{code}
bitrango32(bou95)  
\end{code}
If we add an extra column, consisting of \texttt{vettoretot}, we obtain a matrix of 
rank $1575$, computed with
\begin{code}
bitrango32([bou95 \;  vettoretot'])  
\end{code}
This shows that $d^3(vas) \notin Im(d^1)$ and the differential is not trivial!

\subsection{Non-triviality of the differential in Emb}

As described earlier there is a morphism of spectral sequences $E \to \underline{E}$ induced by the projection $\pi: \overline{Emb} \to Emb$, and the projection at the level of $E^2$ term kills exactly the ideal generated by the framing classes.
Turchin shows that the multiplication on the $E^1$ term is defined by stacking graphs describing cohomology classes and renumbering,
and that $vas$ projects to a non trivial class in $\underline{E}^2 = \underline{E}^3$.
We show that also $d^3(vas)$ projects to a non trivial class in $\underline{E}^3$.

\begin{theorem}\label{ncc-theorem}
The differential $d^3:\underline{E}^3(6,8) \to \underline{E}^3(9,10)$ is non trivial. 
%The domain has dimension 1 and the range dimension 4.
\end{theorem}

There are 11 generators in $E^2(9,10)$, and 7 of them are in the ideal generated by the framing classes. 
Four of these seven classes are
$$c \iota^2, \, eQ(\iota) , \, Q(\iota)\iota^3,\, Q(\iota)\iota z$$ 
where $$c=(\alpha_{2,3} \alpha_{1,4} \alpha_{2,5} )^*+ (\alpha_{1,3} \alpha_{1,4} \alpha_{2,5})^* \in E^2(5,6)$$
$$e=( \alpha_{1,4} \alpha_{3,5} \alpha_{2,6} )^* \in E^2(6,6)$$
$$z=( \alpha_{1,3} \alpha_{2,4})^* \in E^2(4,4)$$
We represent them as vectors with 2520 entries in $\F_2$, since $E^1(9,10)$ has dimension 2520, by 
\begin{code}
C=corde(9,5); 

v1=recog(cat(3,[3 \, 4 \, 5 \, 7 \, 9; 2\,  1\,  2\,  6\,  8] ,[3\,  4\,  5\,  7\,  9; 1\,  1\,  2\,  6\,  8  ]),C); 

v2=recog( [4\,  5\,  6\,  8\,  9; 1\,  3\,  2\,  7\, 7]     ,C); 

v3= recog( [2\, 3\, 5\, 7\, 9; 1\, 1\, 4\, 6\, 8]  ,C); 

v4=recog(  [3\, 4\, 6 \,7\, 9; 1\, 2\, 5\, 5\, 8] ,C);                             
\end{code}
Other 3 of the seven classes are the multiplication of $\iota$ by 3 classes in $E^2(7,8)$
whose coordinate vectors with respect to the normalized planetary basis \texttt{corde(7,4)}  are respectively the columns n. 22,30 and 32
in the matrix 
\begin{code}
K=ker(bousfield(8,4))
\end{code}
and have respectively 35, 40 and 38 non zero entries.
The algorithm \texttt{piuc=piucorda} represents the linear map $E^1(7,8) \to E^1(9,10)$ that is multiplication by $\iota \in E^1(2,2)$ as a vector
with $210$ entries (the dimension of $E^1(7,8)$ ) with values in $\{1,\dots,2520\}$, since 2520 is the dimension of $E^1(9,10)$. 
We encode the coordinate vectors of our 3 classes by
\begin{code}
v5(2520)=0; \, v6(2520)=0; \, v7(2520)=0; 

v5(piuc(find(K(:,22))))=1; \, v6(piuc(find(K(:,30))))=1; \, v7(piuc(find(K(:,32))))=1; 
\end{code}
Adding these as columns to the matrix of $d^1:E^1(8,10) \to E^1(9,10)$ we obtain a matrix 
\begin{code}
M=[bou95 v1' v2' v3' v4' v5' v6' v7'];
\end{code}
of rank 1581. If we stack $d^3(vas)$ as an extra column we obtain a matrix \texttt{[M vettoretot']} of rank 1582, and this proves 
that $d^3(vas)$ is not in the ideal generated by the framing elements in $E^2(9,10)=E^3(9,10)$,
so its projection to $\underline{E}^3(9,10)$ is not zero.